\documentclass[11pt]{amsart}

\usepackage[dvipsnames]{xcolor}
\usepackage{enumitem}
\usepackage{amssymb,graphicx,mathtools}
\usepackage{float}
\usepackage{tikz}
\usetikzlibrary{calc,arrows.meta}
\usepackage[margin=1in]{geometry}
\usepackage{etoolbox}

\usepackage[square,longnamesfirst]{natbib}

\theoremstyle{plain}
\newtheorem{theorem}{Theorem}[section]
\newtheorem{proposition}[theorem]{Proposition}

\newtheorem{lemma}[theorem]{Lemma}
\theoremstyle{definition}
\newtheorem{problem}{Problem}

\usepackage[
citecolor=PineGreen,
colorlinks=true,
linkcolor=RoyalBlue,
urlcolor=BrickRed
]{hyperref}

\newcommand{\Z}{\mathbb Z}
\newcommand{\R}{\mathbb R}
\renewcommand{\P}{\mathbb P}
\newcommand{\E}{\mathbb E}
\newcommand{\defn}[1]{\textbf{#1}}

\tikzset{
	ledge/.style={draw=black!70,line width=0.5pt},
	lvert/.style={circle,fill=black,inner sep=1.1pt},
	dedge/.style={draw=black!35,line width=0.5pt,dashed},
	thickedge/.style={draw=black,line width=1.5pt},
	walkarrow/.style={-{Stealth[length=2.2mm]},line width=0.9pt}
}

\makeatletter
\patchcmd{\@settitle}{\uppercasenonmath\@title}{\Large}{}{}
\patchcmd{\@setauthors}{\MakeUppercase}{\large}{}{}
\makeatother

\title{Eulerian walkers on $\Z^2$ have range exponent $2/3$}
\author{Ahmed Bou-Rabee}
\author{Yuval Peres}

\begin{document}
	
	\begin{abstract}
		In the Eulerian walker model (also known as rotor walk), each site of the
		square lattice begins with an arrow pointing to one of its four neighbors.  A
		walker that starts at the origin repeatedly turns the arrow at its current
		site clockwise by $90^\circ$ and steps in the new direction.
		Priezzhev, Dhar, Dhar, and Krishnamurthy
		\citeyearpar{PriezzhevDharDharKrishnamurthy1996} introduced this as a model of
		self-organized criticality and conjectured that, for independent uniform
		initial directions, the region explored in the first $t$ steps has radius of
		order $t^{1/3}$.  We establish this conjecture and further show that the walker visits every lattice site
		infinitely often, and that the region it has visited by time $t$, rescaled by
		$t^{1/3}$, converges to a convex body.
	\end{abstract}
	
	\maketitle

	\begin{figure}[!b]
		\centering
		\begin{tikzpicture}[scale=1.25,
			rot/.style={-{Stealth[length=3.2mm]},line width=1.7pt},
			ghost/.style={-{Stealth[length=2.2mm]},line width=0.7pt,draw=black!30}]
			
			\foreach \s in {0,1,2} {
				\begin{scope}[xshift={\s*4.3cm}]
					\foreach \x in {-1,0,1} \draw[ledge] (\x,-1.45) -- (\x,1.45);
					\foreach \y in {-1,0,1} \draw[ledge] (-1.45,\y) -- (1.45,\y);
					\foreach \p in {(-1,-1),(0,-1),(1,-1),(-1,0),(-1,1),(0,1),(1,1)}
					\node[lvert] at \p {};
				\end{scope}
			}

			\begin{scope}
				\draw[ghost] (0,0) -- (0.72,0);
				\draw[ghost] (0,0) -- (0,-0.72);
				\draw[ghost] (0,0) -- (-0.72,0);
				\node at (0.24,0.60) {\scriptsize$N$};
				\node at (0.60,0.24) {\scriptsize$E$};
				\node at (-0.24,-0.60) {\scriptsize$S$};
				\node at (-0.60,-0.24) {\scriptsize$W$};
				\draw[rot] (0,0) -- (0,0.78);
				\node[lvert] at (1,0) {};
				\node[circle,fill=RoyalBlue,draw=white,line width=0.7pt,inner sep=2.6pt]
				at (0,0) {};
			\end{scope}
			
			\begin{scope}[xshift=4.3cm]
				\draw[ghost] (0,0) -- (0,0.72);
				\draw[ghost] (0,0) -- (0,-0.72);
				\draw[ghost] (0,0) -- (-0.72,0);
				\draw[-{Stealth[length=2.4mm]},line width=0.9pt,draw=black!65]
				(0.30,0.62) arc (64:26:0.70);
				\draw[rot] (0,0) -- (0.78,0);
				\node[lvert] at (1,0) {};
				\node[circle,fill=RoyalBlue,draw=white,line width=0.7pt,inner sep=2.6pt]
				at (0,0) {};
			\end{scope}
			
			\begin{scope}[xshift=8.6cm]
				\draw[ghost] (0,0) -- (0,0.72);
				\draw[ghost] (0,0) -- (0,-0.72);
				\draw[ghost] (0,0) -- (-0.72,0);
				\draw[rot] (0,0) -- (0.78,0);
				\draw[-{Stealth[length=2.4mm]},line width=0.9pt,draw=black!65]
				(0.12,0.34) .. controls (0.5,0.62) and (0.7,0.5) .. (0.92,0.20);
				\node[lvert] at (0,0) {};
				\node[circle,fill=RoyalBlue,draw=white,line width=0.7pt,inner sep=2.6pt]
				at (1,0) {};
			\end{scope}
			
			\node at (2.15,0) {$\longrightarrow$};
			\node at (6.45,0) {$\longrightarrow$};
		\end{tikzpicture}
		\caption{One step on the square lattice.  The rotor at the walker turns
			from north to east, and the walker then follows it.}
		\label{fig:mechanism}
	\end{figure}
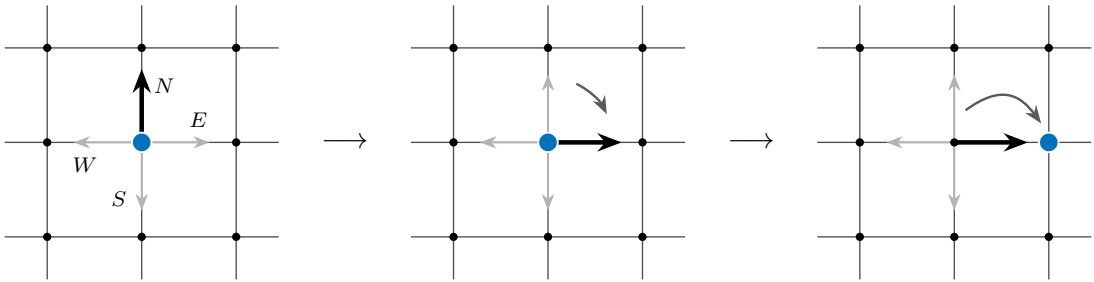

	\section{Introduction}\label{sec:introduction}
	
	An \defn{Eulerian walker}, also called a \defn{rotor walk}, is a
	derandomized version of random walk on a locally finite graph.  Each vertex
	carries a rotor pointing toward one of its neighbors.  The walker starts at a
	vertex, and at each step the rotor at the walker's position advances in a
	fixed cyclic order and the walker follows it
	(Figure~\ref{fig:mechanism}).  The walk is deterministic once the initial
	rotors are fixed.
	
	The model was introduced by \citet*{PriezzhevDharDharKrishnamurthy1996} in
	connection with self-organized criticality.  For independent uniform rotors
	on the square lattice, they  predicted that the region visited by the Eulerian walker in $t$ steps
	has radius of order $t^{1/3}$.  \citet{KapriDhar2009} later performed
	simulations which led them to conjecture this region has a limit shape.  We prove these
	conjectures and further show that the walker is recurrent: it visits every
	lattice site infinitely often.  We also establish these results on every
	doubly periodic graph of maximum degree three.
	Figure~\ref{fig:limit-shapes} shows Eulerian walkers on the square,
	honeycomb, and square--octagon lattices.

	\begin{figure}[!t]
		\centering
		\includegraphics[width=0.76\textwidth]{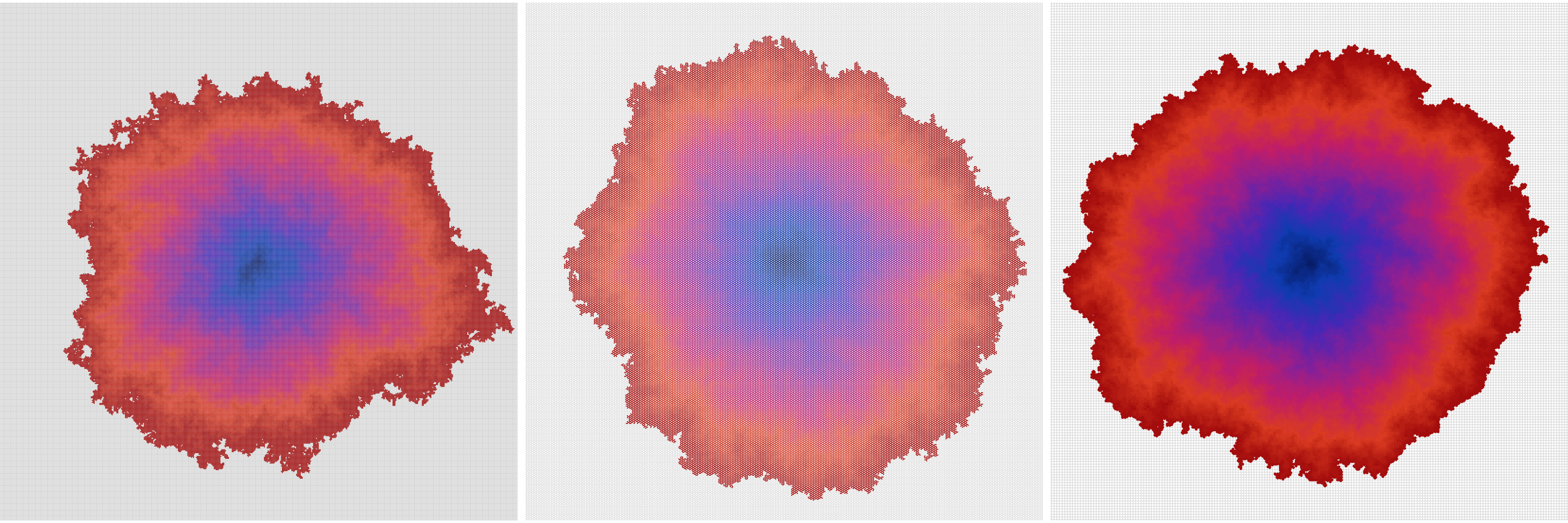}\\[4pt]
		\includegraphics[width=0.76\textwidth]{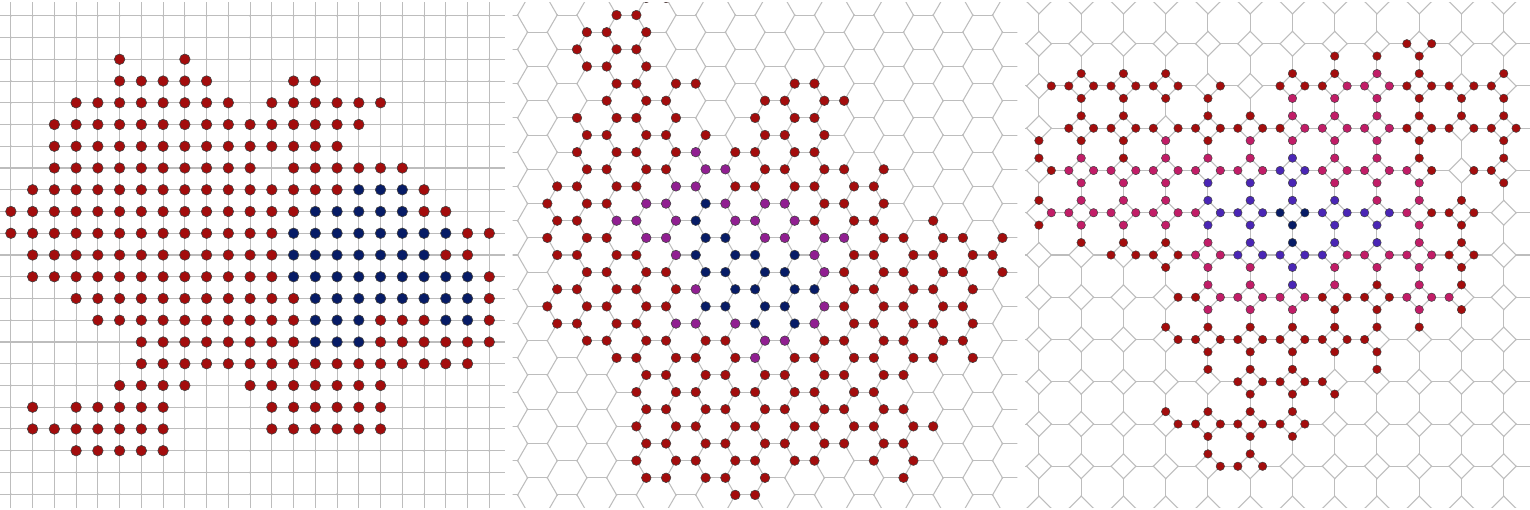}
		\caption{The sites visited by Eulerian walkers with independent uniform
			initial rotors and clockwise rotor mechanisms on the square, honeycomb, and
			square--octagon lattices, from left to right.  The color records the circuit
			in which each site was first visited, from blue to red.  The top row shows
			the ranges at time $4\times10^6$, and the bottom row shows the ranges at
			time $1200$.}
		\label{fig:limit-shapes}
	\end{figure}
	
	\subsection{Main results}\label{ssec:main-results}
	
	In this paper we consider only undirected graphs with no loops and no
	multiple edges.  Let $G=(V,E)$ be a connected, locally finite graph with
	vertex set $V$ and edge set $E$, and write $\deg(v)$ for the number of
	neighbors of $v$.  Each edge $\{x,y\} \in E$ gives two
	\defn{directed edges}, $x\to y$ and $y\to x$.  A
	\defn{rotor configuration} assigns to each vertex $v$ one outgoing directed
	edge $\rho(v)$.  A \defn{rotor mechanism} specifies a cyclic permutation
	$\pi_v$ of the directed edges out of each vertex $v$.  Given an initial position $X_0$ and an
	initial rotor configuration $\rho_0$, define, for every integer $t\geq0$,
	\[
	\rho_{t+1}(v)
	\coloneqq
	\begin{cases}
		\pi_v(\rho_t(v)),&v=X_t,\\
		\rho_t(v),&v\ne X_t
	\end{cases}\, ,
	\]
	and let $X_{t+1}$ be determined by
	\[
	\rho_{t+1}(X_t)=(X_t\to X_{t+1})\, .
	\]
	On the square lattice, the clockwise rotor mechanism has cyclic order
	$N,E,S,W$.
	
	The \defn{range} at time $t$ is
	\[
	R_t\coloneqq\{X_0,\ldots,X_t\}\, .
	\]
	The walk is \defn{recurrent} if it visits every vertex infinitely often and
	\defn{transient} otherwise.  Fix a
	starting vertex $o$.  A \defn{circuit} consists of $\deg(o)$ successive
	returns to $o$.  Define the time at which the first $n$ circuits have been
	completed by
	\begin{equation}\label{eq:circ-time}
		T(n)\coloneqq
		\inf\left\{t\geq0:X_t=o\ \text{and}\ \#\{0\leq s<t:X_s=o\}
		\geq n\deg(o)\right\}\, .
	\end{equation}
	When $T(n)<\infty$, let $A_n\coloneqq R_{T(n)}$, the range after $n$
	circuits, and let $A_0\coloneqq\{o\}$.
	
	A \defn{doubly periodic graph} in $\R^2$ is a connected, locally finite
	graph $G=(V,E)$ with vertex set~$V\!\subset\R^2$, on which a rank-two lattice $\Lambda\subset\R^2$
	acts   by translation. We  assume these translations are   automorphisms of $G$ with finitely many  orbits in $V$.  
	A rotor mechanism on $G$ is \defn{doubly periodic} if its cyclic orders are
	invariant under $\Lambda$.
	
	\begin{theorem}\label{thm:main}
		Let $G$ be either the square lattice with its clockwise rotor mechanism, or
		a doubly periodic graph in $\R^2$ of maximum degree 3, with a doubly periodic rotor
		mechanism.  Suppose that the initial rotors are independent and uniform among
		the directed edges out of each vertex, and start the rotor walk at a fixed vertex
		$o$.  The following statements hold almost surely.
		
		\begin{enumerate}[label=\textup{(\roman*)}]
			\item The walk is recurrent.
			
			\item There exist a deterministic compact convex set
			$B\subset\R^2$, containing the origin in its interior, and a
			deterministic constant $\kappa>0$ such that, in Hausdorff distance,
			\[
			n^{-1}A_n\longrightarrow B
			\qquad\text{and}\qquad
			t^{-1/3}R_t\longrightarrow\kappa B\, .
			\]
			
			\item The limit $\lim_{t\to\infty}|R_t|t^{-2/3}$ exists, is deterministic,
			and is positive and finite.
		\end{enumerate}
	\end{theorem}
	
	In the degree-three case, only conclusions~\textup{(ii)}
	and~\textup{(iii)} of Theorem~\ref{thm:main} require periodicity.  Recurrence holds for independent uniform rotors on every
	infinite connected graph of maximum degree three
	(Proposition~\ref{prop:subcubic-recurrence}).
	
	For comparison, the expected number of sites visited by simple random walk
	on $\Z^2$ by time $t$ has order $t/\log t$ \citep{DvoretzkyErdos1951},
	whereas the rotor walk in Theorem~\ref{thm:main} visits order $t^{2/3}$.  The analogy with random walk might lead one to expect that biased initial rotors yield a transient Eulerian walk. Perhaps surprisingly, this only holds for large biases; the next proposition  shows that recurrence and the scaling of the range are stable with respect to perturbations of the uniform law.

	\begin{proposition}\label{prop:small-perturbations}
		Fix a graph and rotor mechanism satisfying the hypotheses of
		Theorem~\ref{thm:main}.  There is $\delta>0$ with the following property.
		Suppose that the initial rotors are independent and that the law at every
		vertex has total variation distance less than $\delta$ from the uniform law.
		Then conclusion~\textup{(i)} of Theorem~\ref{thm:main} holds, and if these
		laws are invariant under the translation lattice, then
		conclusions~\textup{(ii)} and~\textup{(iii)} hold as well.
	\end{proposition}
	
	For example, Proposition~\ref{prop:small-perturbations} applies, for every
	sufficiently small $\varepsilon>0$, to independent square-lattice rotors with
	\[
	\P\{\rho(v)=N\}=\frac14+\varepsilon,\qquad
	\P\{\rho(v)=S\}=\frac14-\varepsilon,\qquad
	\P\{\rho(v)=E\}=\P\{\rho(v)=W\}=\frac14\, .
	\]

	The lower bound $|R_t|\geq c t^{2/3}$ in Theorem~\ref{thm:main} is known in
	much greater generality: \citet[Theorem~1.1]{FlorescuLevinePeres2016} prove it on two-dimensional lattices
	for every initial rotor configuration.  On the other hand, as proved
	in \citet[Theorem~1.2]{FlorescuLevinePeres2016} and remarked in
	\citet[Section~1]{AngelHolroyd2012}, if the initial rotor configuration has an
	infinite path directed toward the starting vertex, then the walk is
	transient and the range grows linearly.  Our proof of
	Theorem~\ref{thm:main} thus uses randomness in an essential way.
	
	On the other hand, Theorem~\ref{thm:main} is false on some doubly periodic planar graphs. For an integer $M\geq1$, let $G_M$ be the square lattice with $M$ leaves
	attached to every lattice vertex, all of them drawn between the edge to
	the west and the edge to the north, so that the clockwise order at a lattice
	vertex is
	\[
	N,E,S,W,L_1,\ldots,L_M\, ,
	\]
	see Figure~\ref{fig:pendant}.
	
	\begin{proposition}[A transient doubly periodic plane graph]
		\label{prop:pendant-counterexample}
		There exists $M_0<\infty$ such that for every $M\geq M_0$, the rotor walk on
		$G_M$ with the clockwise rotor mechanism and independent uniform initial
		rotors is transient. 
	\end{proposition}
	
	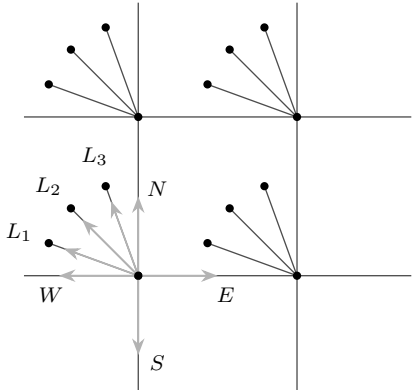
\begin{figure}[t]
		\centering
		\begin{tikzpicture}[scale=2.1,
			ghost/.style={-{Stealth[length=2.2mm]},line width=0.7pt,draw=black!30}]
			\foreach \x in {0,1} \draw[ledge] (\x,-0.72) -- (\x,1.72);
			\foreach \y in {0,1} \draw[ledge] (-0.72,\y) -- (1.72,\y);
			\foreach \x in {0,1} \foreach \y in {0,1} {
				\foreach \a in {110,135,160} {
					\draw[ledge] (\x,\y) -- ($(\x,\y)+(\a:0.6)$);
					\node[lvert] at ($(\x,\y)+(\a:0.6)$) {};
				}
				\node[lvert] at (\x,\y) {};
			}
			\draw[ghost] (0,0) -- ++(0,0.5);
			\draw[ghost] (0,0) -- ++(0.5,0);
			\draw[ghost] (0,0) -- ++(0,-0.5);
			\draw[ghost] (0,0) -- ++(-0.5,0);
			\foreach \a in {110,135,160} \draw[ghost] (0,0) -- ++(\a:0.5);
			\node[lvert] at (0,0) {};
			\node at (0.12,0.55) {\scriptsize $N$};
			\node at (0.55,-0.12) {\scriptsize $E$};
			\node at (0.12,-0.55) {\scriptsize $S$};
			\node at (-0.55,-0.12) {\scriptsize $W$};
			\node at ($(0,0)+(160:0.80)$) {\scriptsize $L_1$};
			\node at ($(0,0)+(135:0.80)$) {\scriptsize $L_2$};
			\node at ($(0,0)+(110:0.80)$) {\scriptsize $L_3$};
		\end{tikzpicture}
		\caption{The graph $G_M$ for $M=3$.}
		\label{fig:pendant}
	\end{figure}
	
	The lattice steps of the rotor walk on $G_M$ form a rotor walk on $\Z^2$
	whose induced rotors are independent and point west with probability
	$(M+1)/(M+4)$.
	For large enough $M$ this bias must force the walk to escape
	to infinity.  Since our result on $\Z^2$ is stable with respect to the
	randomness (Proposition~\ref{prop:small-perturbations}), the bias $M$ must be
	large.  We carry this out in Section~\ref{sec:counterexample}. 
	
	\smallskip
	We collect several open questions raised by these results in
	Section~\ref{sec:further-questions}.
	
	\subsection{Proof overview}\label{ssec:proof-overview}
	
	\citet{PriezzhevDharDharKrishnamurthy1996} gave a heuristic argument for why
	$|R_t|\sim t^{2/3}$ on $\Z^2$.  They observed that the rotors used by the
	walker at large time are arranged nearly into an Eulerian circuit, so the
	walker cannot visit a new site without making roughly four more departures
	from every site already in the range.  Growing the radius of visited
	sites from $r$ to $r+1$ therefore costs on the order of $r^2$ steps, giving
	$dr/dt\approx r^{-2}$, hence $r\approx t^{1/3}$ and a range of about
	$r^2\approx t^{2/3}$ vertices.
	
	Parts of this heuristic were formalized in \citet{FlorescuLevinePeres2016},
	who showed, for more general graphs, that if $T(n)<\infty$, then between
	$T(n)$ and $T(n+1)$ every directed edge is used at most once.  Moreover, if
	$T(n+1)<\infty$, then during this interval the walk departs exactly
	$\deg(v)$ times from every vertex $v\in A_n$
	(Lemma~\ref{lem:one-circuit}).  Thus, if $T(n)<\infty$, then
	$A_n$ contains the ball of radius $n$.  These facts were used in
	\citet{FlorescuLevinePeres2016} to give a lower bound of $t^{2/3}$ for the
	range on two-dimensional graphs.
	
	Our proof takes these facts as its starting point and proceeds in three
	steps.
	\begin{enumerate}[label=\textup{(\arabic*)}]
		\item The circuit ranges evolve by a monotone rule that depends only on the
		current range, so the number of circuits needed to reach one vertex from
		another behaves like a first-passage time: it obeys a triangle inequality,
		and the subadditive ergodic theorem gives it a linear growth rate in each
		direction.  These rates form a passage function whose unit ball is the limit
		shape, and the theorem reduces to showing that the function does not vanish.
		\item To rule that out, we show that a walk reaching a distant vertex in few
		circuits forces the initial rotors to contain a long path along which they
		favor motion outward.  This step is deterministic: it turns a statement about
		growth into one about the initial configuration.
		\item Such paths are unlikely.  This is the only step that depends on the
		graph, and on the square lattice it is the bulk of the work.
	\end{enumerate}
	\medskip
	
	\noindent\textbf{Step 1.}  Recall $A_n$ and $T(n)$ from
	Section~\ref{ssec:main-results}.  For a finite set $S$ of vertices, place one particle at the
	head of each directed edge leaving $S$, declare $S$ absorbing, and move the
	particles by the rotor rule, one step at a time in an arbitrary order, until
	every particle has stepped into $S$.  When this procedure terminates, let
	$\Phi(S)$ be the union of $S$ and the set of vertices visited during it.  Observe that if
	$T(n)<\infty$, then
	\begin{equation}\label{eq:an-monotone-map}
		A_n=\Phi^n(\{o\})\, .
	\end{equation}
	For vertices $x$ and $y$, define
	\begin{equation}\label{eq:stationary-passage-time}
		\tau(x,y)\coloneqq\min\{n\geq0:y\in\Phi^n(\{x\})\}\, ,
	\end{equation}
	that is, the number of circuits required for the walk started at $x$ to reach
	$y$.  Section~\ref{sec:circuits} extends this definition to configurations on
	which some boundary routing fails to terminate:
	Lemma~\ref{lem:least-action} makes $\Phi$ well defined,
	Proposition~\ref{prop:circuit-iterate} gives
	\eqref{eq:an-monotone-map}, and Proposition~\ref{prop:passage} gives the
	triangle inequality for $\tau$.  Since $\tau$ is also stationary in law and at
	most the graph distance, the subadditive ergodic theorem provides a
	deterministic limit $\tau(o,nz)/n\to\mu(z)$, and the unit ball of $\mu$ is the
	limit shape.  The main
	difficulty is to show that $\mu$ does not vanish, that is, to rule out
	abnormally fast growth.
	
	\medskip
	
	\noindent\textbf{Step 2.}  We convert the question of fast growth into one
	about the initial rotors.  The idea is that if a circuit is abnormally
	large, then the initial rotors must be abnormally configured.  To make this
	precise we use live paths, introduced by \citet[Section~2]{AngelHolroyd2011}.
	
	Here and throughout the paper, a \defn{path} $x_0,x_1,\ldots$ is a finite or
	infinite sequence of distinct vertices such that consecutive vertices are
	adjacent.  Its \defn{internal} vertices are all its vertices except the first
	and, for a finite path, the last.  A path is \defn{live at an internal vertex
		$x_i$} if, in the cyclic order beginning immediately after the initial rotor
	at $x_i$, the edge $x_i\to x_{i+1}$ occurs before $x_i\to x_{i-1}$.  The path
	is \defn{live} if this holds at every internal vertex.  See
	Figure~\ref{fig:live-path}.  As we
	will see below, a circuit that reaches a distant vertex forces a long live
	path to exist (Lemma~\ref{lem:decreasing-positions}).
	
	\begin{figure}[!ht]
		\centering
		\begin{tikzpicture}[
			scale=1.25,
			lat/.style={draw=black!22,line width=0.5pt},
			path/.style={-{Stealth[length=2.4mm]},draw=RoyalBlue,line width=1.5pt},
			rotor/.style={-{Stealth[length=3.4mm]},draw=BrickRed,line width=2.4pt},
			slot/.style={circle,draw=black!45,fill=white,inner sep=0pt,
				minimum size=4mm,font=\scriptsize}]
			\draw[lat] (-1.6,0)--(1.6,0);
			\draw[lat] (0,-1.6)--(0,1.6);
			\draw[path] (-1.35,0)--(-0.08,0);
			\draw[path] (0,-0.08)--(0,-1.35);
			\draw[rotor] (0,0)--(0,1.02);
			\node[circle,fill=black,inner sep=1.3pt] at (0,0) {};
			\node[slot] at (0.52,0) {1};
			\node[slot,draw=RoyalBlue,fill=RoyalBlue!12] at (0,-0.52) {2};
			\node[slot,draw=BrickRed,fill=BrickRed!12] at (-0.52,0) {3};
			\node[slot] at (0,0.52) {4};
			\node[font=\scriptsize] at (-1.18,0.25) {$x_{i-1}$};
			\node[font=\scriptsize] at (0.25,0.20) {$x_i$};
			\node[font=\scriptsize] at (0.30,-1.20) {$x_{i+1}$};
		\end{tikzpicture}
		\caption{The initial rotor at $x_i$ is red, and the circles number the
			outgoing edges in the order they are used.  Since $x_i\to x_{i+1}$ comes
			before $x_i\to x_{i-1}$, the path is live at $x_i$.}
		\label{fig:live-path}
	\end{figure}
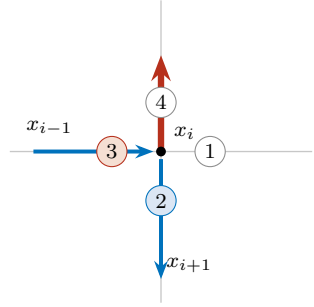
	
	In their proof of Theorem~6, \citet{AngelHolroyd2011} showed that on a
	rooted tree in which every non-root vertex has degree $d$, with uniform
	initial rotors, live paths percolate when $d\geq4$ and do not percolate when
	$d=3$.  On a tree, an infinite live path
	forces the rotor walk to be transient; this implication is false for graphs
	with cycles.  On the other hand, for every graph, a rotor walk with no infinite live
	path is recurrent (Proposition~\ref{prop:live-recurrence}).  If
	$y\in\Phi^n(\{x\})$, then some path from $x$ to $y$ fails the live condition
	at no more than $n-1$ of its internal vertices
	(Lemma~\ref{lem:decreasing-positions}).

	\medskip 
	\noindent\textbf{Step 3.}  The absence of infinite live paths does not by
	itself rule out abnormally fast growth.  On the rooted binary tree, whose
	non-root vertices have degree three, the uniform rotor walk is recurrent by
	\citet{AngelHolroyd2011}, but its range grows linearly by
	\citet[Theorem~1.1]{HussSavaHuss2020Range}.  What we need is quantitative
	control of almost-live paths: for some $\eta>0$,
	\begin{equation}\label{eq:overview-path-estimate}
		\lim_{R\to\infty}\sup_{u\to v}
		\P\left\{
		\begin{array}{c}
			\text{some path starts with $u\to v$ and reaches distance $R$ from $u$,}\\
			\text{and the live condition fails at no more than $\eta R$ internal vertices}
		\end{array}
		\right\}=0\, .
	\end{equation}
	Observe that this estimate fails on the binary tree for every $\eta>0$.  
	Section~\ref{sec:recurrence} proves that \eqref{eq:overview-path-estimate}
	suffices for Theorem~\ref{thm:main}.  It also reduces this estimate to one
	block event. Sections~\ref{sec:killed-walk} and~\ref{sec:square} verify this estimate 
	in each of our two cases.  Since the block event depends on finitely many rotors, the estimate
	persists under small perturbations of their laws, giving
	Proposition~\ref{prop:small-perturbations}. Our proof of this estimate is different in the maximal degree 3 and square lattice cases. 
	
	At a vertex of degree three the number of ways to continue a live path is
	uniform on $\{0,1,2\}$, so on any graph of maximal degree 3, an exploration of the live paths from a given
	edge is dominated by a critical Galton--Watson process and dies out
	(Section~\ref{sec:killed-walk}).  The argument combines the branching
	comparison of \citet{AngelHolroyd2011} with the exploration of
	\citet[Theorem~1]{BenjaminiSchramm1996}; it does not need periodicity and
	already gives recurrence (Proposition~\ref{prop:subcubic-recurrence}).
	On a doubly periodic graph balls grow polynomially, so within a bounded
	distance of every edge there is either a vertex of degree at most two or a
	vertex lying on a bounded cycle, and either one gives a live path a uniformly positive chance of ending or of
	repeating a vertex.  Iterating gives the required finite-scale estimate
	(Proposition~\ref{prop:degree-three-passage}).

	On the square lattice we instead compare live paths with critical bond
	percolation.
	Rotating a directed edge $v\to w$ counterclockwise through $90^\circ$ about its
	midpoint gives a dual edge from the face on the right of $v\to w$ to the face
	on its left.  Call it open when $v\to w$ comes second or third in the cyclic
	order beginning after the initial rotor at $v$
	(Figure~\ref{fig:live-dual}).  On a live path the edges between the outgoing
	edge and the reverse of the incoming one are used second or third, so their
	duals are open and chain into a directed path.
	
	We must therefore show that long open dual paths are rare.  The four dual
	edges at a vertex form a square, of which two adjacent sides are open, and each
	side is open with probability $1/2$.  Critical bond percolation on $\Z^2$ does
	not percolate, but its annular crossing probabilities remain bounded away from
	zero, so it does not give the block estimate.  The adjacency of the two open
	sides gives it instead: no open dual path traverses a fixed five-edge pattern,
	and percolation avoiding that pattern is subcritical
	(Subsection~\ref{sec:square-diminishment}).  The sides of a square are not
	independent, so we transfer the estimate to the rotor model by an exploration
	(Section~\ref{sec:square}).  This is the dual configuration of the $p=1/2$
	directed-corner model of \citet*{CoupierHenryJahnelKoppl2024}, from which the
	estimate could also be deduced, but we give a self-contained proof.

	\begin{figure}[!ht]
		\centering
		\begin{tikzpicture}[scale=1.9,
			prim/.style={draw=black!22,line width=0.5pt},
			vert/.style={circle,fill=black,inner sep=1.2pt},
			face/.style={circle,draw=black!60,fill=white,line width=0.7pt,inner sep=1.6pt},
			pedge/.style={-{Stealth[length=2.4mm]},draw=RoyalBlue,line width=1.5pt},
			dedge/.style={-{Stealth[length=2.4mm]},draw=BrickRed,line width=1.4pt},
			shut/.style={-{Stealth[length=2.0mm]},draw=black!38,line width=0.9pt,
				dash pattern=on 2pt off 1.7pt},
			slot/.style={circle,draw=black!45,fill=white,line width=0.5pt,inner sep=0pt,
				minimum size=3.2mm,font=\tiny}]
			\foreach \x in {0,1,2,3} \draw[prim] (\x,-1.1) -- (\x,1.5);
			\foreach \y in {-1,0,1} \draw[prim] (-0.55,\y) -- (3.1,\y);
			\draw[pedge] (0.08,0) -- (0.92,0);
			\draw[pedge] (1.08,0) -- (1.92,0);
			\draw[pedge] (2,0.08) -- (2,0.92);
			\draw[dedge] (0.58,-0.5) -- (1.42,-0.5);
			\draw[dedge] (1.58,-0.5) -- (2.42,-0.5);
			\draw[dedge] (2.5,-0.42) -- (2.5,0.42);
			\draw[shut] (2.42,0.5) -- (1.58,0.5);
			\draw[shut] (1.5,0.42) -- (1.5,-0.42);
			\node[slot] at (2,0.34)  {1};
			\node[slot] at (2.34,0) {2};
			\node[slot] at (2,-0.34) {3};
			\node[slot] at (1.66,0) {4};
			\foreach \z/\n/\a in {(0,0)/{x_0}/{below left},(1,0)/{x_1}/{above left},
				(2,0)/{x_2}/{above left},(2,1)/{x_3}/{right}}
			\node[vert,label={[font=\scriptsize,inner sep=3pt]\a:$\n$}] at \z {};
			\foreach \z in {(0.5,-0.5),(1.5,-0.5),(2.5,-0.5),(2.5,0.5),(1.5,0.5)}
			\node[face] at \z {};
			\node[font=\scriptsize,RoyalBlue] at (0.5,0.62) {live path};
			\node[font=\scriptsize,BrickRed] at (1.35,-0.88) {open dual edges};
		\end{tikzpicture}
		\caption{A live path in blue and the open dual edges it yields in red; the
			circles are faces.  The numbers order the four edges at $x_2$: the path arrives
			along edge $4$ and leaves along edge $1$, so the sides at $2$ and $3$ are open
			and the dashed sides closed.}
		\label{fig:live-dual}
	\end{figure}
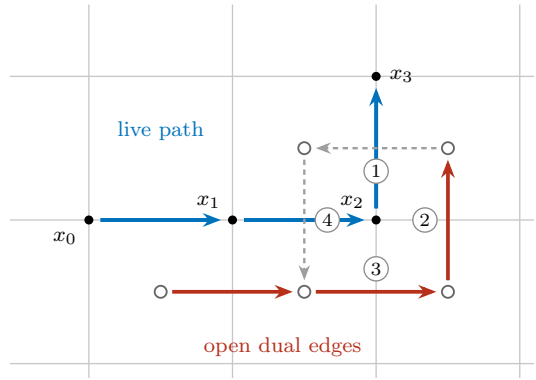
	
	\subsection{Related work}\label{ssec:related-work}  The rotor walk arose independently in statistical physics and computer science.
	\citet*{PriezzhevDharDharKrishnamurthy1996} introduced it in 1996 as a model
	of self-organized criticality related to the abelian sandpile
	\citep{BakTangWiesenfeld1987,Dhar1990}; in the same year,
	\citet*{WagnerLindenbaumBruckstein1996} introduced the same mechanism as an
	exploration rule.  Alternating routers also appeared in periodic load-balancing circuits \citep{RabaniSinclairWanka1998},
	and Propp proposed it around 2000 as a deterministic analogue
	of random walk, calling it the rotor-router model
	\citep{Propp2003,Propp2010}.  This derandomization perspective was developed
	by \citet{HolroydPropp2010} and \citet{CooperSpencer2006}, who obtained
	quantitative discrepancy bounds comparing rotor routing with random walk for
	finite Markov chains and for multiparticle dynamics on $\Z^d$, respectively.
	Early work on its dynamics studied avalanche structure and mean-square displacement in the critical state
	\citep{PovolotskyPriezzhevShcherbakov1998}.

	Propp also introduced rotor aggregation \citep{Propp2003}.   In rotor
	aggregation, $n$ particles are released one at a time at the origin, and each
	performs rotor walk until it reaches an unoccupied site, where it stops.
	Replacing every Eulerian walker by a simple random walk turns this into internal diffusion-limited aggregation, and the two models have the same
	limit shape on $\Z^d$, a Euclidean ball \citep{LawlerBramsonGriffeath1992,LevinePeres2009}.  On other graphs the shape reflects the geometry.  On
	regular trees, acyclic initial rotors produce an exact ball whenever the number
	of particles equals the size of that ball \citep{LandauLevine2009}.  For a
	particular initial rotor configuration on the two-dimensional comb, the boundary
	consists of parabolic arcs \citep{HussSava2011}.  On the Sierpi\'nski gasket,
	rotor aggregation and internal diffusion-limited aggregation both have
	graph-metric balls as limit shapes
	\citep{ChenKudlerFlam2020,ChenHussSavaHussTeplyaev2020}.  On the layered square lattice, obtained by
	modifying the routing mechanism at sites on the coordinate axes,
	\citet{KagerLevine2010} show that a suitable initial rotor configuration
	produces exact diamond-shaped clusters at certain sizes.  In the  abelian sandpile model, the scaling limit is not as simple
	\citep{PegdenSmart2013,BouRabee2021,BouRabee2024Shape} and is related to
	Apollonian circle packings \citep{LevinePegdenSmart2016,LevinePegdenSmart2017}
	and rational points on hyperbolas \citep{BouRabee2024FLattice}. See \citet{LevinePeres2017} for a survey.
	
	Propp's proposal raised the question of whether rotor walk is recurrent in
	dimension two and transient in higher dimensions
	\citep[Section~1]{FlorescuGangulyLevinePeres2014}.  Schramm settled one
	direction in an average sense: for a rotor walk restarted at its origin after
	each escape to infinity, the asymptotic frequency of these  escapes is at most the probability that a random walk from the same initial vertex escapes to infinity before returning to its origin~\citep[Theorem~10]{HolroydPropp2010}.  On $\Z^2$ the bound gives only zero escape frequency, which does not imply recurrence: with aligned rotors, the restarted walk still escapes infinitely often.  Neither the initial vertex nor a
	finite change in the initial rotors affects recurrence
	\citep{AngelHolroyd2012}, but the law of the initial rotors does.  With aligned rotors, among
	$n$ walks started successively from the origin, the number that never return is
	of order $n$ for $d\geq3$ and $n/\log n$ for $d=2$
	\citep{FlorescuGangulyLevinePeres2014}.  Every transient graph carries an
	initial configuration whose escape rate matches that of simple random walk
	\citep{Chan2020}, whereas rotors sampled from the oriented wired uniform
	spanning forest make the walk visit every vertex at most as often in
	expectation as simple random walk, and hence transient \citep{Chan2019}.
	
	As discussed above, rotor walks are relatively well understood on trees
	\citep{AngelHolroyd2011,HussSava2012,HussMullerSavaHuss2015,
		HussSavaHuss2020Range,HussSavaHuss2020Periodic}.  On other graphs, much less is
	known.  Recurrence holds for independent uniform
	rotors on the doubly infinite Sierpi\'nski gasket graph with the anticlockwise
	mechanism \citep{KaiserSavaHuss2024}.  \citet{FlorescuLevinePeres2016} prove a
	shape theorem on the comb and, through connections with the mirror model and
	critical bond percolation, recurrence on two directed lattices, the Manhattan lattice and the
	F-lattice. 
	
	Randomizing the updates, rather than only the initial rotors, makes related
	questions more tractable.  On $\Z$, an interpolation between rotor walk and
	simple random walk has a doubly perturbed Brownian scaling limit
	\citep{HussLevineSavaHuss2018}, and a quenched invariance principle holds for
	a class of walks with local memory on lattice graphs
	\citep{ChanGrecoLevineLi2021}.  One such walk is the horizontal--vertical
	walk with independent uniform initial labels, for which \citet{Chan2023}
	proves recurrence when the flip probability exceeds $1/3$.
	
	\subsection{Notation and conventions}
	
	\begin{itemize}
		\item For a graph $G=(V,E)$ with vertex set $V$ and edge set $E$, we write
		$x\sim y$ when $\{x,y\}\in E$, and $d_G(x,y)$ for the least number of edges
		in a path from $x$ to $y$.
		\item $|S|$ is the cardinality of a finite set $S$, while $|\cdot|_\infty$
		and $|\cdot|$ are the $\ell^\infty$ and Euclidean norms.
		\item $\P$ is the law of the initial rotor configuration.
		\item $c>0$ and $C>0$ are finite constants which may change from line to
		line.
	\end{itemize}
	
	\section*{Acknowledgments}
	The research of Y. Peres was supported by National Natural Science Foundation of China grant
	RFIS-W2531011.  We acknowledge use of ChatGPT 5.6.

	\section{Basic facts}\label{sec:circuits}
	In this section we prove three deterministic basic facts.  First, whenever
	they are defined, successive circuit ranges are the iterates of a monotone
	map; when all boundary routings terminate, these iterates are the balls of a
	passage time satisfying the triangle inequality
	(Propositions~\ref{prop:circuit-iterate}--\ref{prop:passage}).  Second, a walk
	that reaches a distant vertex in few circuits forces a path along which the
	live condition fails at few vertices
	(Lemma~\ref{lem:decreasing-positions}).  Third, the absence of an infinite
	live path implies recurrence (Proposition~\ref{prop:live-recurrence}).

	Throughout this section, $G$ is infinite, connected, and locally finite, the
	rotor mechanism and initial rotor configuration are fixed and arbitrary, and $o$ is the
	starting vertex of the walk. These results can be extended to certain directed graphs, but we do not need this; see
	\citet{HolroydLevineMeszarosPeresProppWilson2008}.
	
	\subsection{Circuits and routing}
	
	Recall from Section~\ref{ssec:main-results} that $R_t=\{X_0,\ldots,X_t\}$ is
	the range at time $t$, and that a circuit consists of $\deg(o)$ successive
	returns to $o$.  The completion time $T(n)$ of the first $n$ circuits is
	defined in \eqref{eq:circ-time}, and when $T(n)<\infty$ the set $A_n=R_{T(n)}$
	is the range at that time, with $A_0=\{o\}$.
	
	\begin{lemma}[\citealp{FlorescuLevinePeres2016}, Lemmas~2.1
		and~2.4]\label{lem:one-circuit}
		For every integer $n\geq0$ such that $T(n)<\infty$, the walk traverses each
		directed edge at most once during the time interval
		$T(n)\leq t<T(n+1)$.  If $T(n+1)<\infty$, then during this interval the
		walk departs from every vertex $x\in A_n$ exactly $\deg(x)$ times.
	\end{lemma}

	Outside a finite set $S$, it is convenient to view a circuit as a system of
	particles, one for each directed edge leaving $S$.  Moving the particles in
	different orders produces the same final state.
	
	Fix a nonempty finite set $S\subseteq V$ and regard its vertices as sinks:
	they carry no rotors and remove every particle that enters them.
	A \defn{particle-and-rotor state} is a pair $\xi=(\sigma,\rho)$ consisting of
	a particle configuration $\sigma:V\setminus S\to\Z_{\geq0}$ and a rotor
	configuration $\rho$ on $V\setminus S$.  A vertex is \defn{occupied} when it
	contains at least one particle.  To \defn{actuate} an occupied vertex, advance
	its rotor and move one particle along the new rotor edge, removing the
	particle if that edge enters $S$.
	
	A state $\xi'$ is a \defn{successor} of $\xi$ if one actuation transforms
	$\xi$ into $\xi'$, and a state is \defn{stable} if no particles remain.  A
	\defn{legal routing} is a finite or infinite sequence of states, each a
	successor of the one before.  It is \defn{complete} if it is finite and its
	last state is stable.  A \defn{boundary routing} of $S$ is a legal routing
	started with one particle at $x$ for each directed edge $s\to x$ with
	$s\in S$ and $x\notin S$, and with the fixed initial rotors on
	$V\setminus S$.  We say that the boundary routing of $S$ \defn{terminates} if
	a complete boundary routing of $S$ exists.
	
	The following is an immediate consequence of the abelian property, see
	\citet[Lemma~3.9]{HolroydLevineMeszarosPeresProppWilson2008}.

	\begin{lemma}[Abelian property]\label{lem:least-action}
		Let $\xi_0,\ldots,\xi_n$ and
		$\widehat\xi_0,\ldots,\widehat\xi_m$ be legal routings with
		$\widehat\xi_0=\xi_0$.
		\begin{enumerate}[label=\textup{(\alph*)}]
			\item If $\xi_n$ is stable, then $m\leq n$, and each vertex is actuated no
			more often in the second routing than in the first.
			\item If $\xi_n$ and $\widehat\xi_m$ are both stable, then $m=n$, their
			final states agree, and each vertex is actuated equally often in the two
			routings.
		\end{enumerate}
	\end{lemma}
	
	Call a boundary routing \defn{one-particle-at-a-time} when it orders the
	directed edges $s\to x$ from $S$ to $V\setminus S$, routes each corresponding
	particle until it enters $S$ before starting the next, and stops only after the
	last particle enters $S$.  If a particle never enters $S$, the routing instead
	continues with that particle forever.  Rotor states are retained, and
	$s\to x$ counts as the first edge of the corresponding route.
	
	\begin{lemma}[Boundary routing]\label{lem:boundary-routing}
		No legal boundary routing traverses a directed edge more than once, counting
		the initial edges from $S$.  Every one-particle-at-a-time boundary routing is
		finite if and only if the boundary routing of $S$ terminates.  If so, each is
		complete and has the same actuation counts as every complete boundary routing.
	\end{lemma}
	
	\begin{proof}
		If $u\to v$ were the first repeated directed edge, then $u\notin S$ would
		have been actuated $\deg(u)+1$ times.  Legality would therefore require at
		least $\deg(u)+1$ incoming traversals at $u$, counting the initial edges from
		$S$.  One of the $\deg(u)$ incoming directed edges would then
		have repeated earlier, a contradiction. The remaining two assertions follow immediately from Lemma~\ref{lem:least-action}.
	\end{proof}
	
	Define the \defn{circuit map} $\Phi$, on the nonempty finite sets
	$S\subseteq V$ whose boundary routing terminates, by
	\begin{equation}\label{eq:phi-definition}
		\Phi(S)
		\coloneqq
		S\cup\{x\notin S:\text{$x$ is actuated in a complete routing}\}\, .
	\end{equation}
	Set $\Phi^0(S)\coloneqq S$; a positive iterate is defined only if every
	intervening boundary routing terminates.  Lemma~\ref{lem:least-action} makes
	$\Phi(S)$ independent of the chosen complete routing, and by construction it
	is independent of the rotor states inside $S$.
	
	\begin{proposition}[Circuit ranges are iterates]
		\label{prop:circuit-iterate}
		Let $n\geq0$ and suppose that $T(n)<\infty$.  Then $T(n+1)<\infty$ if and
		only if the boundary routing of $A_n$ terminates,
		and in that case
		\begin{equation}\label{eq:circuit-iterate}
			A_{n+1}=\Phi(A_n)\, .
		\end{equation}
	\end{proposition}
	
	\begin{proof}
		Starting at time $T(n)$, delete the actuations in $A_n$ and regard the
		traversals leaving $A_n$ as initial boundary edges.  This gives a legal
		boundary routing.
		If the boundary routing of $A_n$ terminates,
		Lemma~\ref{lem:boundary-routing} bounds the number of outside actuations.
		Before time $T(n+1)$ the walk departs from
		each vertex $x$ at most $\deg(x)$ times by Lemma~\ref{lem:one-circuit}, so the
		actuations inside the finite set $A_n$ are finitely many as well.  Hence
		$T(n+1)<\infty$.  Conversely, if $T(n+1)<\infty$, the walk uses every
		directed edge leaving $A_n$ once between $T(n)$ and $T(n+1)$, so the routing
		is complete.  The abelian property gives
		\eqref{eq:circuit-iterate}.
	\end{proof}
	
	Lemma~\ref{lem:boundary-routing} and Lemma~\ref{lem:least-action} imply monotonicity 
	of the circuit map. 
	\begin{proposition}[Monotonicity]\label{prop:monotonicity}
		Let $S\subseteq T$ be nonempty finite sets whose boundary routings
		terminate.  Then
		\begin{equation}\label{eq:monotonicity}
			\Phi(S)\subseteq\Phi(T)\, .
		\end{equation}
	\end{proposition}

	\subsection{The passage time}
	As $G$ is countable, the event that the boundary routing of every nonempty
	finite set terminates is measurable.  If the graph and rotor
	mechanism are doubly periodic, this event is translation invariant.
	For $x,y\in V$, define
	\begin{equation}\label{eq:passage-definition}
		\tau(x,y)
		\coloneqq
		\begin{cases}
			\min\{n\in\Z_{\geq0}:y\in\Phi^n(\{x\})\},
			&\begin{gathered}
				\text{if the boundary routing of every nonempty}\\
				\text{finite set terminates},
			\end{gathered}\\
			d_G(x,y),&\text{otherwise.}
		\end{cases}
	\end{equation}
	In the first case, Proposition~\ref{prop:circuit-iterate} identifies
	$\tau(x,y)$ with the least $n$ for which the range through the first $n$
	circuits of the rotor walk started at $x$ contains $y$.  The second case
	defines $\tau$ on every rotor configuration while preserving the stationarity,
	measurability, and triangle inequality used below.
	
	\begin{proposition}[Passage time]\label{prop:passage}
		For all vertices $x$, $y$, and $z$,
		\begin{equation}\label{eq:triangle}
			\tau(x,z)\leq\tau(x,y)+\tau(y,z)\, ,
		\end{equation}
		and
		\begin{equation}\label{eq:passage-upper}
			\tau(x,y)\leq d_G(x,y)\, .
		\end{equation}
		If the boundary routing of every nonempty finite set terminates, then
		for every integer $n\geq0$,
		\begin{equation}\label{eq:passage-balls}
			A_n=\{x:\tau(o,x)\leq n\}\, .
		\end{equation}
	\end{proposition}
	
	\begin{proof}
		If the boundary routing of some nonempty finite set does not terminate, then
		$\tau=d_G$, so the triangle inequality for $d_G$ proves \eqref{eq:triangle},
		and \eqref{eq:passage-upper} holds with equality.  Suppose instead that the
		boundary routing of every nonempty finite set terminates.  For every such $S$,
		set
		\[
			\partial S\coloneqq\{y\in V\setminus S:y\sim x\text{ for some }x\in S\}\, .
		\]
		Every vertex of $\partial S$
		initially carries a particle and is therefore actuated in a complete
		routing, so $S\cup\partial S\subseteq\Phi(S)$.  Iterating this inclusion
		along a shortest path from $x$ to $y$ proves \eqref{eq:passage-upper}.
		Set $m\coloneqq\tau(x,y)$ and
		$k\coloneqq\tau(y,z)$.  Since $\{y\}\subseteq\Phi^m(\{x\})$,
		Proposition~\ref{prop:monotonicity} gives
		\[
		\Phi^k(\{y\})
		\subseteq\Phi^k\bigl(\Phi^m(\{x\})\bigr)
		=\Phi^{m+k}(\{x\})\, .
		\]
		The left-hand side contains $z$, proving \eqref{eq:triangle}.  Finally,
		Proposition~\ref{prop:circuit-iterate} gives $A_n=\Phi^n(\{o\})$, and
		\eqref{eq:passage-definition} gives \eqref{eq:passage-balls}.
	\end{proof}
	
	\subsection{Live paths}\label{ssec:live-paths}
	
	Recall from Section~\ref{ssec:proof-overview} that a path is live when, at
	every internal vertex, its forward edge precedes the reverse of its incoming
	edge in the cyclic order beginning after the initial rotor.  The next lemma has two parts.  The first extracts a live
	path from a single boundary routing.  The second extracts, from the iterates
	of $\Phi$, a path at which the live condition fails at few vertices.
	
	\begin{lemma}[Extracting live paths]\label{lem:decreasing-positions}
		\leavevmode
		\begin{enumerate}[label=\textup{(\roman*)}]
			\item Let $S\subseteq V$ be nonempty and finite and let $y\notin S$.  If a
			particle in a one-particle-at-a-time boundary routing of $S$ visits $y$,
			counting its initial boundary edge as its first step, then
			there is a live path $x_0,x_1,\ldots,x_m=y$ such that $x_0\in S$ and
			$x_i\notin S$ for $1\leq i\leq m$.
			\item Let $n\geq1$, suppose that $\Phi^i(\{x\})$ is defined for
			$1\leq i\leq n$, and let $y\in\Phi^n(\{x\})$.  There is a path from
			$x$ to $y$, contained in $\Phi^n(\{x\})$, that is live at all but at most
			$n-1$ internal vertices.
		\end{enumerate}
	\end{lemma}
	
	\begin{proof}
		For part~\textup{(i)}, stop the routing at its first visit to $y$.  By
		Lemma~\ref{lem:boundary-routing} no directed edge is traversed twice, so the
		edges traversed out of a vertex $v\notin S$ are an initial segment of the
		cyclic order beginning immediately after the initial rotor at $v$.
		
		Consider the completed particle routes and the current route, counting the
		initial boundary edge of each.  Every completed route runs from $S$ to $S$,
		while the current route runs from $S$ to $y$.  Hence every vertex outside
		$S\cup\{y\}$ has as many
		incoming as outgoing traversals, while $y$ has one more incoming than outgoing.
		Deleting every pair of oppositely directed traversals preserves this, so the
		surviving edges contain a path from $S$ to $y$.  Truncating that path after its
		last vertex in $S$ gives $x_0,\ldots,x_m=y$ with $x_0\in S$ and $x_i\notin S$
		for $1\leq i\leq m$.
		
		Let $x_i$ be an internal vertex.  The edge $x_i\to x_{i+1}$ was traversed,
		since it survived the deletion.  The edge $x_i\to x_{i-1}$ was never traversed,
		since otherwise it would have been deleted together with $x_{i-1}\to x_i$,
		which survived.  The traversed edges out of $x_i$ form an initial segment of
		the cyclic order, so $x_i\to x_{i+1}$ precedes $x_i\to x_{i-1}$ in that order.
		This is the live condition at $x_i$.
		
		For part~\textup{(ii)}, let $j$ be the least integer with
		$y\in\Phi^j(\{x\})$ and induct on $j$.  The case $j=0$ is trivial, and
		part~\textup{(i)} with $S=\{x\}$ settles $j=1$.
		
		Let $j\geq2$ and set $S\coloneqq\Phi^{j-1}(\{x\})$, so that $y\in\Phi(S)$
		and, by minimality of $j$, $y\notin S$.  A complete one-particle-at-a-time
		boundary routing of $S$ has a particle that visits $y$, so part~\textup{(i)}
		gives a live path from a vertex $z\in S$ to $y$ that meets $S$ only at $z$ and
		lies in $\Phi^j(\{x\})$.  The inductive path from $x$ to $z$ lies in $S$, so
		the two meet only at $z$ and concatenate to a path from $x$ to $y$ inside
		$\Phi^j(\{x\})$.  The live condition fails at most $j-2$ times on the
		inductive path, never on the live path, and possibly once at the junction $z$,
		hence at most $j-1\leq n-1$ times in all.
	\end{proof}
	
	From nonexistence of infinite live paths, we deduce recurrence of the walk.
	
	\begin{proposition}[Live paths and recurrence]\label{prop:live-recurrence}
		If $G$ contains no infinite live path, then the boundary routing of every
		nonempty finite set terminates, $T(n)<\infty$ for every $n$, and the rotor
		walk is recurrent.
	\end{proposition}
	
	\begin{proof}
		Suppose that the boundary routing of some nonempty finite set $S$ does not
		terminate, and choose a one-particle-at-a-time boundary routing of $S$.  By
		Lemma~\ref{lem:boundary-routing} it is infinite and traverses no directed edge
		twice, and a finite set contains only finitely many directed edges, so the
		routing leaves every finite set.  Given an integer $R\geq1$, some particle
		therefore visits a vertex at graph distance $R$ from $S$, and
		Lemma~\ref{lem:decreasing-positions}\textup{(i)} converts that visit into a
		live path of length at least $R$ that starts in $S$ and never returns to $S$.
		
		All finite positive-length live paths that start in $S$ and otherwise avoid $S$
		form a forest
		in which the parent of a path is obtained by deleting its last vertex.  Its
		roots are the paths with a single edge, one for
		each directed edge leaving $S$, so there are finitely many; each node has
		finitely many children because $G$ is locally finite.  By the
		previous paragraph the forest has nodes at every depth, so some root has
		descendants at every depth.  K\"onig's lemma gives an infinite branch there,
		whose union is an infinite live path, contrary to hypothesis.
		
		Thus the boundary routing of every nonempty finite set terminates, and
		Proposition~\ref{prop:circuit-iterate} gives $T(n)<\infty$ for every $n$ by
		induction.  Hence the walk returns to $o$ infinitely often.  A rotor walk that
		visits one vertex infinitely often visits every vertex infinitely often
		\citep[Lemma~6]{HolroydPropp2010}, and recurrence does not depend on the
		starting vertex \citep[Theorem~1]{AngelHolroyd2012}.
	\end{proof}

	\section{From an almost-live-path estimate to shape and range}
	\label{sec:recurrence}
	
	This section derives Theorem~\ref{thm:main} and
	Proposition~\ref{prop:small-perturbations} from one estimate: for some
	$\eta>0$, the probability that an almost-live path reaches distance $R$ tends
	to zero uniformly in its first edge.  Proposition~\ref{prop:path-reduction}
	converts this estimate into recurrence and a uniform $o(1)$ lower-tail
	estimate for the passage time.
	Propositions~\ref{prop:passage-limit}--\ref{prop:circuit-clock} then give the
	limit shape and range asymptotics.  Finally,
	Lemma~\ref{lem:block-live-paths} reduces the estimate to one finite-scale
	crossing bound and shows that the bound persists under small perturbations of
	the rotor laws.  Sections~\ref{sec:killed-walk} and~\ref{sec:square} verify
	this bound in the two cases of Theorem~\ref{thm:main}.
	
	\subsection{Recurrence, shape, and range}
	
	For a directed edge $e=u\to v$, an integer $R\geq1$, and $\eta>0$, define the
	almost-live-path event
	\begin{equation}\label{eq:almost-live-path-event}
		\mathcal L_\eta(e,R)
		\coloneqq
		\left\{
		\begin{array}{c}
			\text{some path starts with $e$ and reaches graph distance $R$ from $u$,}\\
			\text{and the live condition fails at no more than $\eta R$ internal vertices}
		\end{array}
		\right\}\, .
	\end{equation}
	
	\begin{proposition}[Recurrence, shape, and range from live paths]\label{prop:path-reduction}
		Let $G$ be an infinite connected graph of bounded degree with an arbitrary
		rotor mechanism, and let $\P$ be a law for the initial rotors.
		Suppose that, for some $\eta>0$,
		\begin{equation}\label{eq:criterion-path-hypothesis}
			\lim_{R\to\infty}\sup_{u\to v}
			\P\{\mathcal L_\eta(u\to v,R)\}=0\, .
		\end{equation}
		Then the following statements hold.
		\begin{enumerate}[label=\textup{(\roman*)}]
			\item Almost surely, the boundary routing of every nonempty finite set
			terminates and the walk is recurrent.
			\item Writing $\tau$ for the passage time defined in
			\eqref{eq:passage-definition}, there is $a>0$ such that
			\begin{equation}\label{eq:criterion-passage}
				\lim_{R\to\infty}
				\sup_{\substack{x,y\in V\\ d_G(x,y)\geq R}}
				\P\{\tau(x,y)\leq ad_G(x,y)\}=0\, .
			\end{equation}
			\item If, in addition, $G$ and the rotor mechanism are doubly periodic and the
			rotor law is invariant and ergodic under the translation lattice, then there
			exist a deterministic compact convex set $B\subset\R^2$ containing the origin in its interior
			and deterministic constants $\kappa,c_*>0$ such that, almost surely,
			\[
			n^{-1}A_n\longrightarrow B,
			\qquad
			t^{-1/3}R_t\longrightarrow\kappa B,
			\qquad
			|R_t|=c_* t^{2/3}+o(t^{2/3})\, ,
			\]
			where $A_n$ is the range through the first $n$ circuits, $R_t$ is the range at
			time $t$, and the set limits are in Hausdorff distance.
		\end{enumerate}
	\end{proposition}
	
	\begin{proof}[Proof of \textup{(i)} and \textup{(ii)}]
		An infinite live path beginning with a directed edge $e$ belongs to
		$\mathcal L_\eta(e,R)$ for every $R$.  Thus
		\eqref{eq:criterion-path-hypothesis}, followed by a countable union over the
		directed edges, rules out all infinite live paths almost surely.
		Proposition~\ref{prop:live-recurrence} then implies~\textup{(i)}.
		
		For part~\textup{(ii)}, set $a\coloneqq\min\{\eta,1/2\}$ and fix distinct
		vertices $x,y$.  By
		Lemma~\ref{lem:decreasing-positions}\textup{(ii)}, on the probability-one
		event supplied by part~\textup{(i)},
		\begin{equation}\label{eq:fast-passage-inclusion}
			\{\tau(x,y)\leq ad_G(x,y)\}
			\subseteq
			\bigcup_{z\sim x}\mathcal L_\eta(x\to z,d_G(x,y))\, .
		\end{equation}
		Indeed, the lemma supplies a path from $x$ to $y$ with at most
		$\tau(x,y)-1\leq\eta d_G(x,y)$ failures of the live condition.  The bounded
		degree of $G$, \eqref{eq:criterion-path-hypothesis}, and a union bound give
		\eqref{eq:criterion-passage}.
	\end{proof}
	
	Parts~\textup{(i)} and~\textup{(ii)} do not require periodicity of the graph.
	Part~\textup{(iii)} does, and the rest of this subsection is devoted to its
	proof, and we assume the additional hypotheses of
	Proposition~\ref{prop:path-reduction}\textup{(iii)}: the graph and the rotor
	mechanism are doubly periodic, and the rotor law is invariant and ergodic
	under the translation lattice, which we write as $\Lambda$. We also identify the vertex $o$ with the origin.

	The periodic placement makes graph distance and Euclidean distance comparable:
	there is a constant $C<\infty$ such that
	\begin{equation}\label{eq:distance-comparison}
		C^{-1}|x-y|-C
		\leq d_G(x,y)
		\leq C|x-y|+C\, ,
	\end{equation}
	for all vertices $x$ and $y$.  The symmetry of $d_G$ and
	Proposition~\ref{prop:passage} also give
	\begin{equation}\label{eq:passage-four-endpoints}
		\bigl|\tau(x_1,y_1)-\tau(x_2,y_2)\bigr|
		\leq d_G(x_1,x_2)+d_G(y_1,y_2)\, ,
	\end{equation}
	for all vertices $x_1,x_2,y_1,y_2$. 
	
	\begin{proposition}\label{prop:passage-limit}
		There is a deterministic continuous subadditive function
		$\mu:\R^2\to[0,\infty)$ such that $\mu(\theta x)=\theta\mu(x)$ for every
		$\theta\geq0$ and
		\begin{equation}\label{eq:uniform-passage}
			\lim_{R\to\infty}
			\sup_{\substack{x\in V\\ |x|\geq R}}
			\frac{|\tau(o,x)-\mu(x)|}{|x|}=0\, ,
		\end{equation}
		almost surely.  The function $\mu$ does not depend on the base point $o$.
	\end{proposition}
	
	\begin{proof}
		Set $\mu(0)\coloneqq0$. Fix nonzero $z\in\Lambda$.  The stationary
		subadditive array
		$\{\tau(o+mz,o+nz):0\leq m<n\}$ has an integrable linear bound by
		\eqref{eq:passage-upper} and \eqref{eq:distance-comparison}.  The
		subadditive ergodic theorem \citep[Theorems~3 and~5]{Kingman1968} gives,
		almost surely and in $L^1$, simultaneously for every base point and every
		$z\in\Lambda$, an a priori random limit
		\begin{equation}\label{eq:directional-limit}
			\mu(z)\coloneqq
			\lim_{n\to\infty}\frac{\tau(o,o+nz)}n\, .
		\end{equation}
		The usual argument, using the triangle inequality and
		\eqref{eq:passage-four-endpoints}, extends $\mu$ to the stated deterministic
		function and gives \eqref{eq:uniform-passage}.
	\end{proof}

	\begin{proposition}[Shape of the passage balls]\label{prop:circuit-shape}
		Let $\mu$ be the passage function from
		Proposition~\ref{prop:passage-limit}.  Suppose that the passage-ball identity
		\eqref{eq:passage-balls} holds almost surely for every $n\geq0$ and that
		$\min_{|u|=1}\mu(u)>0$.  Define
		\begin{equation}\label{eq:limit-shape-body}
			B\coloneqq\{x\in\R^2:\mu(x)\leq1\}\, .
		\end{equation}
		Then $B$ is compact and convex and contains the origin in its interior.
		Almost surely,
		\begin{equation}\label{eq:circuit-shape}
			n^{-1}A_n\longrightarrow B\, ,
		\end{equation}
		in Hausdorff distance.  Almost surely, for every $\varepsilon\in(0,1)$,
		\begin{equation}\label{eq:shape-sandwich}
			(1-\varepsilon)nB\cap V
			\subseteq A_n
			\subseteq(1+\varepsilon)nB\cap V\, ,
		\end{equation}
		for every sufficiently large $n$.
	\end{proposition}
	
	\begin{proof}
		Homogeneity and continuity give $\lambda|x|\leq\mu(x)\leq C|x|$ with
		$\lambda\coloneqq\min_{|u|=1}\mu(u)>0$, so $B$ is compact and contains a
		neighborhood of the origin, and subadditivity makes it convex.  Comparing
		$\tau(o,\cdot)$ with $\mu$ through \eqref{eq:uniform-passage} and applying
		\eqref{eq:passage-balls} gives \eqref{eq:shape-sandwich} for all large $n$,
		from which \eqref{eq:circuit-shape} follows.
	\end{proof}
	
	It remains to convert circuit number into time.
	
	\begin{proposition}[Circuit clock]
		\label{prop:circuit-clock}
		Suppose $T(n)<\infty$ for every $n$ for a rotor walk on a connected locally
		finite graph.  Then, for every integer $n\geq0$,
		\begin{equation}\label{eq:universal-clock-squeeze}
			\sum_{x\in A_n}\deg(x)
			\leq T(n+1)-T(n)
			\leq\sum_{x\in A_{n+1}}\deg(x)\, .
		\end{equation}
		Suppose that, for some $\alpha,\beta>0$,
		\[
		|A_n|=\alpha n^2+o(n^2),
		\qquad
		\sum_{x\in A_n}\deg(x)=\beta n^2+o(n^2)\, .
		\]
		Then
		\[
		T(n)=\frac \beta3n^3+o(n^3),
		\qquad
		\lim_{t\to\infty}\frac{|R_t|}{t^{2/3}}
		=\alpha\left(\frac3\beta\right)^{2/3}\, .
		\]
		If
		$n^{-1}A_n\to B$ in Hausdorff distance for a compact set $B$, then
		\[
		t^{-1/3}R_t\longrightarrow\left(\frac3\beta\right)^{1/3}B\, ,
		\]
		in Hausdorff distance.
	\end{proposition}
	
	\begin{proof}
		By Lemma~\ref{lem:one-circuit}, between times $T(n)$ and $T(n+1)$ the walk
		departs exactly $\deg(x)$ times from each $x\in A_n$ and every departure
		leaves a vertex of $A_{n+1}$, which is \eqref{eq:universal-clock-squeeze};
		with the second assumed asymptotic, summing gives
		$T(n)=\beta n^3/3+o(n^3)$.  Choosing $n$ with $T(n)\leq t<T(n+1)$ gives
		$A_n\subseteq R_t\subseteq A_{n+1}$ and $nt^{-1/3}\to(3/\beta)^{1/3}$, from
		which the assumed asymptotics give the stated limits for $|R_t|t^{-2/3}$
		and $t^{-1/3}R_t$.
	\end{proof}
	
	\begin{proof}[Proof of Proposition~\ref{prop:path-reduction}\textup{(iii)}]
		Fix a nonzero $z\in\Lambda$ and $b<aC^{-1}|z|$, where $a$ is as in
		\eqref{eq:criterion-passage}.  By \eqref{eq:distance-comparison},
		$bn\leq a d_G(o,o+nz)$ for all sufficiently large $n$.  Hence
		\eqref{eq:criterion-passage} gives
		\[
			\P\left\{\frac{\tau(o,o+nz)}n\leq b\right\}\longrightarrow0\, .
		\]
		Since $\tau(o,o+nz)/n\to\mu(z)$ in probability, $\mu(z)\geq b$.
		Letting $b\uparrow aC^{-1}|z|$ gives
		$\mu(z)\geq aC^{-1}|z|$, and homogeneity and continuity extend this to
		$\R^2$.  Part~\textup{(i)} and
		Proposition~\ref{prop:passage} give \eqref{eq:passage-balls}, so
		Proposition~\ref{prop:circuit-shape} yields \eqref{eq:shape-sandwich} and
		$n^{-1}A_n\to B$; since $\partial B$ is Lebesgue null, counting lattice
		points of each vertex orbit on both sides verifies the hypotheses of
		Proposition~\ref{prop:circuit-clock} with deterministic $\alpha,\beta>0$,
		whence $\kappa=(3/\beta)^{1/3}$ and $c_*=\alpha(3/\beta)^{2/3}$.
	\end{proof}

	\subsection{A finite-scale criterion}\label{sec:block}
	In this section we reduce the almost-live-path estimate \eqref{eq:criterion-path-hypothesis}
	to one inequality at a single scale: if a
	live path is sufficiently unlikely to cross one large enough block, then domination by
	subcritical percolation carries the bound to every $R$.
	
	Let $G$ be a doubly periodic graph with an arbitrary rotor mechanism, and identify its translation lattice with $\Z^2$, so that the vertex set splits into finite sets indexed by $\Z^2$.  For an integer
	$L\geq1$ and $z\in\Z^2$, let the \defn{block} $Q_z$ be the union of the sets
	indexed by $Lz+\{0,\ldots,L-1\}^2$, and let
	\begin{equation}\label{eq:block-definitions}
		Q_z^+\coloneqq\bigcup_{|w-z|_\infty\leq1}Q_w\, .
	\end{equation}
	Denote the law of the initial rotor at each vertex $x$ by law $\nu_x$ and write $\P_0$ for the
	resulting product law.  Let $\mathcal C_z$ be the event that a live path
	starts in $Q_z$, stays in $Q_z^+$ except for its last vertex, and ends
	outside $Q_z^+$.
	
	\begin{lemma}[Block estimate]
		\label{lem:block-live-paths}
		There are $\varepsilon>0$ and $L_0\geq1$ with the following property.  Let
		$L\geq L_0$ and suppose that
		\begin{equation}\label{eq:block-crossing-hypothesis}
			\sup_{z\in\Z^2}\P_0(\mathcal C_z)<\varepsilon\, .
		\end{equation}
		Then \eqref{eq:criterion-path-hypothesis} holds under $\P_0$ for some
		$\eta>0$.  There is also $\delta>0$ such that the same $\eta$ works
		for every product law whose marginal at each vertex $x$ is within total
		variation distance $\delta$ of $\nu_x$.
	\end{lemma}
	
	\begin{proof}
		Choose $L_0$ so that, whenever $L\geq L_0$, adjacent graph vertices lie
		in blocks whose indices are at $\ell^\infty$-distance at most one.
		Choose $\varepsilon>0$ so that every $2$-dependent $\{0,1\}$-field with
		one-site probabilities at most $2\varepsilon$ is dominated by independent
		Bernoulli variables of parameter $1/8$; this is possible by
		\citet[Theorem~0.0\textup{(ii)}]{LiggettSchonmannStacey1997}.
		Since \eqref{eq:block-crossing-hypothesis} is strict, we may fix $s,\delta>0$
		small enough that
		\[
			\sup_{z}\P_0(\mathcal C_z)+|Q_0^+|(s+\delta)<2\varepsilon\, .
		\]
		Fix laws $\widetilde\nu_x$ within total variation distance $\delta$ of
		$\nu_x$.  Under $\widehat\P$, let the initial rotors be independent with
		laws $\widetilde\nu_x$, and mark each vertex independently with probability
		$s$; a path is live when it is live for these rotors.
		
		Let $E_z$ be the event that a path crosses $Q_z^+$ as in $\mathcal C_z$ and
		is live at every unmarked internal vertex.  Independently over vertices,
		couple the rotor at each vertex $x$ with one of law $\nu_x$ so that the two
		coincide with probability at least $1-\delta$.  A path realizing $E_z$ realizes $\mathcal C_z$ for the
		coupled rotors unless some vertex of $Q_z^+$ is marked or carries differing
		rotors, so the above display gives $\sup_z\widehat\P(E_z)<2\varepsilon$.
		Since
		$E_z$ depends only on the rotors and marks in $Q_z^+$, the field
		$(\mathbf 1_{E_z})_{z\in\Z^2}$ is $2$-dependent and is therefore dominated
		by independent Bernoulli variables of parameter $1/8$.
		
		On $\mathcal L_\eta(u\to v,R)$, choose the first witnessing path in a fixed
		ordering.  Marking all its failures has conditional probability at least
		$s^{\eta R}$ and forces $E_z$ along $cR$ blocks.  A self-avoiding sequence of
		block indices has at most seven continuations at each step, and $7/8<1$, so
		there are $\gamma,C>0$ such that
		\[
			s^{\eta R}\widehat\P\{\mathcal L_\eta(u\to v,R)\}
			\leq Ce^{-\gamma R}\, .
		\]
		Choosing $\eta\log(1/s)<\gamma$ proves
		\eqref{eq:criterion-path-hypothesis}, in fact with exponential decay.
	\end{proof}
	
	\section{Graphs of maximum degree three}
	\label{sec:killed-walk}
	
	Throughout this section the initial rotors are independent and uniform.  We
	show that on every graph of maximum degree three there is no infinite live
	path (Proposition~\ref{prop:subcubic-recurrence}), so the walk is recurrent by
	Proposition~\ref{prop:live-recurrence}.  We also show that under the additional condition 
	the graph is doubly periodic and embedded in $\R^2$,  a live path reaches distance $R$ from its first vertex with
	probability at most $Ce^{-cR}$
	(Proposition~\ref{prop:degree-three-passage}).  Together with
	Lemma~\ref{lem:block-live-paths}, this verifies
	\eqref{eq:criterion-path-hypothesis} and proves Theorem~\ref{thm:main} in this case.

	For vertices $v\sim w$, let $r_v(w)$ be the position of $v\to w$ in the
	cyclic order beginning immediately after the initial rotor at $v$.  Thus a
	path that enters $v$ from $u$ and leaves along $v\to w$ is live at $v$
	precisely when
	\[
	r_v(w)<r_v(u)\, .
	\]
	There are $r_v(u)-1$ such edges $v\to w$, and $r_v(u)$ is uniform on
	$\{1,\ldots,\deg(v)\}$.  Their expected number is
	\begin{equation}\label{eq:mean-continuations}
		\sum_{\substack{w\sim v\\ w\ne u}}\P\{r_v(w)<r_v(u)\}
		=\E[r_v(u)-1]
		=\frac{\deg(v)-1}{2}
		\leq1\, .
	\end{equation}
	Proposition~\ref{prop:subcubic-recurrence} is proved by exploring every such
	edge and dominating the resulting tree by a critical Galton--Watson tree.  For
	the quantitative estimate in Proposition~\ref{prop:degree-three-passage}, the
	probabilities on the left of \eqref{eq:mean-continuations} instead define a
	killed nonbacktracking chain, a walk being \defn{nonbacktracking} when it
	never traverses an edge and immediately traverses it in reverse.
	
	\citet[proof of Theorem~6\textup{(ii)}]{AngelHolroyd2011} carry out this
	branching comparison on a tree, where the live paths form a Galton--Watson tree
	exactly because a vertex is entered only from its parent.  On a general graph a
	vertex may be entered from several directions, so we adapt the exploration in
	\citet[proof of Theorem~1]{BenjaminiSchramm1996}.

	The exploration is defined as follows.  Fix a directed edge $e=o\to x$ and add
	it to a first-in, first-out queue.  For every $v\in V\setminus\{o\}$, set
	$M_v=0$. While the queue is nonempty, remove and process its first edge $u\to v$: if
	$r_v(u)\leq M_v$, do nothing; otherwise add to the queue, in increasing order
	of $r_v(w)$, every edge $v\to w$ not already added that satisfies
	\[
		w\sim v,
		\qquad
		w\ne o,
		\qquad
		M_v<r_v(w)<r_v(u)\, ,
	\]
	and then set $M_v=r_v(u)$.  Thus $M_v$ is the largest $r_v(u)$ over the edges
	$u\to v$ processed so far.  The added edges form a tree rooted at $e$, in which
	the parent of an edge added while $u\to v$ was being processed is $u\to v$.
	Every added edge is eventually processed, since only finitely many edges
	precede it in the queue.  A vertex $v\in V\setminus\{o\}$ is reached when an edge with
	head $v$ is added.  For each $v\in V\setminus\{o\}$, after an edge entering $v$
	is processed, every edge $v\to w$ with $w\ne o$ and $r_v(w)<M_v$ has been
	added, unless $w\to v$ was processed and increased $M_v$.
	
	\begin{proposition}[No infinite live paths]
		\label{prop:subcubic-recurrence}
		Let $G$ be an infinite connected graph of maximum degree three,
		equipped with an arbitrary rotor mechanism.  If the initial rotors are
		independent and uniform, then almost surely $G$ contains no infinite live
		path.
	\end{proposition}
	
	\begin{proof}
		Fix a directed edge $e=o\to x_1$, run the exploration from $e$, and let
		$x_0=o,x_1,\ldots,x_n$, where $n\geq1$, be a finite live path beginning with
		$e$.
		We show by induction that
		$M_{x_i}\geq r_{x_i}(x_{i-1})$ for $1\leq i\leq n$.  In particular, every
		$x_i$ with $i\geq1$ is reached.  The inequality holds at $i=1$ once
		$o\to x_1$ is processed.  Given the
		inequality at $i<n$, the live property gives
		$r_{x_i}(x_{i+1})<r_{x_i}(x_{i-1})\leq M_{x_i}$,
		so the exit $x_i\to x_{i+1}$ has been added unless $x_{i+1}\to x_i$ is an
		edge whose processing increased $M_{x_i}$.  The exit cannot enter $o$ because
		the live path is simple.  If the exit was added, it is eventually processed, and
		$M_{x_{i+1}}\geq r_{x_{i+1}}(x_i)$.  Otherwise the parent of
		$x_{i+1}\to x_i$ exists because $x_{i+1}\ne o$.  This parent is an edge
		$b\to x_{i+1}$ satisfying
		$r_{x_{i+1}}(b)>r_{x_{i+1}}(x_i)$, so again
		$M_{x_{i+1}}\geq r_{x_{i+1}}(x_i)$.
		
		Retain the first processed edge entering each reached vertex and join it to its
		nearest retained ancestor.  Every nonretained edge $u\to v$ has at most one
		child, since before it is processed $M_v\geq1$ and there is at most one
		$w\sim v$ satisfying $M_v<r_v(w)<r_v(u)\leq3$.  Thus, after contracting
		nonretained chains, a retained edge $u\to v$ has at most $r_v(u)-1$ children.

		For each reached $v\in V\setminus\{o\}$, its rotor is unexamined when the first
		edge $u\to v$ entering $v$ is processed.  Conditional on the past, $r_v(u)$ is
		uniform on $\{1,\ldots,\deg(v)\}$.  Hence the retained tree is dominated by a
		Galton--Watson tree with offspring uniform on $\{0,1,2\}$. This completes the proof. 
	\end{proof}
	
	We next show exponential decay under the extra hypothesis of doubly periodic.
	
	\begin{proposition}[Exponential decay of live paths]
		\label{prop:degree-three-passage}
		Let $G$ be a doubly periodic graph in $\R^2$ of maximum degree three, equipped with an
		arbitrary rotor mechanism and independent uniform initial rotors.
		There are constants $c,C>0$ such that, for every directed edge $u\to v$ and
		every integer $R\geq1$,
		\begin{equation}\label{eq:degree-three-live-path-tail}
			\P\left\{
			\begin{array}{c}
				\text{a live path starts with $u\to v$ and contains a vertex}\\
				\text{at graph distance $R$ from $u$}
			\end{array}
			\right\}
			\leq Ce^{-cR}\, .
		\end{equation}
	\end{proposition}
	
	\begin{proof}
		Fix $e=u\to v$ and an integer $R\geq1$.
		Let $\P_e$ be the law of the killed nonbacktracking chain
		$X_0,X_1,\ldots$, started with $X_0=u$ and $X_1=v$, that, from
		$X_{j-1}=a$ and $X_j=b$, moves to $X_{j+1}=c$ with probability
		$\P\{r_b(c)<r_b(a)\}$ for each $c\sim b$ with $c\ne a$, and is killed
		otherwise.  By \eqref{eq:mean-continuations}, its killing probability is
		$1-(\deg(b)-1)/2$, which is positive exactly when $\deg(b)\leq2$.

		We claim that
		\begin{equation}\label{eq:chain-survival}
			\P_e\{\text{the chain is not killed by step $j$,
			and $X_0,\ldots,X_j$ are distinct}\}
			\leq Ce^{-cj}
		\end{equation}
		for every integer $j\geq1$, with $c$ and $C$ depending only on $G$.
		
		By assumption, balls of radius $r$ in the graph metric have at most $Cr^2$
		vertices, while a nonbacktracking walk with distinct vertices, all of degree
		three, has two continuations at every step.  Take $m$ so large that
		$2^{\lfloor m/2\rfloor}>Cm^2$.  From every directed edge, some
		nonbacktracking walk of at most $m$ steps therefore repeats a vertex or reaches
		a vertex of degree at most two: otherwise two distinct continuations of length
		$\lfloor m/2\rfloor$ have the same endpoint, and following one to that endpoint
		and the other in reverse gives such a repetition.  Steps
		have probability at least $1/3$ and killings at least $1/2$, so given any
		distinct $X_0,\ldots,X_j$ with no killing among them, the chain follows such
		a walk with conditional probability at least
		$3^{-m}/2$, and is then killed or repeats a vertex;
		iterating over blocks of $C m$ steps gives \eqref{eq:chain-survival}.
		
		Let $\Gamma_R$ be the set of paths $\gamma=(x_0,\ldots,x_\ell)$ with $x_0=u$
		and $x_1=v$ such that $d_G(u,x_i)<R$ for $i<\ell$ and $d_G(u,x_\ell)=R$. For $\gamma\in\Gamma_R$, independence of
		the rotors at its internal vertices gives
		\[
		\P\{\gamma\text{ is live}\}
		=\P_e\{X_i=x_i\text{ for }0\leq i\leq\ell\}\, .
		\]
		The events on the right are disjoint, and every $\gamma\in\Gamma_R$ has at
		least $R$ edges.  Hence
		\begin{align*}
			\P\{\text{some $\gamma\in\Gamma_R$ is live}\}
			&\leq \sum_{\gamma\in\Gamma_R}\P\{\gamma\text{ is live}\}\\
			&=\P_e\left(\bigcup_{\substack{\gamma=(x_0,\ldots,x_\ell)\in\Gamma_R}}
			\{X_i=x_i\text{ for }0\leq i\leq\ell\}\right)\\
			&\leq Ce^{-cR}\, ,
		\end{align*}
		the last step by \eqref{eq:chain-survival} with $j=R$.
		Stopping a live path when it first reaches distance $R$ from $u$ gives an
		element of $\Gamma_R$. This completes the proof. 
	\end{proof}

	\section{Square lattice}\label{sec:square}
	
	Throughout this section $G$ is the square lattice with the clockwise rotor
	mechanism, the initial rotors are independent and uniform, and $\P_0$ is their
	product law.  We prove the following exponential decay estimate for live paths. 
	
	\begin{proposition}[Exponential decay]
		\label{prop:square-passage}
		Let $G$ be the square lattice with the clockwise rotor mechanism and
		independent uniform initial rotors, and let $\P_0$ be their product law.
		There are constants $c,C>0$ such that, for every directed edge $u\to v$ and
		every integer $R\geq1$,
		\begin{equation}\label{eq:square-live-path-tail}
			\P_0\{\text{a live path from $u\to v$ reaches graph distance $R$ from $u$}\}
			\leq Ce^{-cR}\, .
		\end{equation}
	\end{proposition}
	
	Proposition~\ref{prop:square-passage} gives the block estimate
	\eqref{eq:block-crossing-hypothesis}.  For a path realizing the block event
	$\mathcal C_z$ from
	Section~\ref{sec:block}, consider the segment from the tail of its last edge
	leaving the block $Q_z$ defined in \eqref{eq:block-definitions} to the endpoint
	of the path outside $Q_z^+$.  This segment is live, starts with one of at most
	$CL$ directed edges leaving $Q_z$, and reaches a vertex at graph distance at
	least $cL-C$ from its first vertex.
	By \eqref{eq:square-live-path-tail} and a union bound,
	\[
	\sup_z\P_0(\mathcal C_z)
	\leq CLe^{-cL}
	\longrightarrow0\, .
	\]
	For large $L$, Lemma~\ref{lem:block-live-paths} therefore verifies
	\eqref{eq:criterion-path-hypothesis} on the square lattice.
	
	\begin{proof}[Proof of Theorem~\ref{thm:main}]
		Propositions~\ref{prop:degree-three-passage} and~\ref{prop:square-passage}
		verify \eqref{eq:criterion-path-hypothesis} in both cases of
		Theorem~\ref{thm:main}.  In either case, the uniform product law is invariant
		and ergodic under the translation lattice, so
		Proposition~\ref{prop:path-reduction} gives all the stated conclusions.
	\end{proof}
	
	\begin{proof}[Proof of Proposition~\ref{prop:small-perturbations}]
		The degree-three estimate \eqref{eq:degree-three-live-path-tail} and the
		square-lattice estimate \eqref{eq:square-live-path-tail} give
		\eqref{eq:block-crossing-hypothesis} for some $L$ under the uniform rotor
		laws.  The perturbation conclusion of Lemma~\ref{lem:block-live-paths} gives
		\eqref{eq:criterion-path-hypothesis} under every sufficiently small
		total-variation perturbation allowed in
		Proposition~\ref{prop:small-perturbations}.  Proposition~\ref{prop:path-reduction}
		\textup{(i)} gives conclusion~\textup{(i)} for every allowed perturbation.  If
		the one-vertex laws are invariant under the translation lattice, their product
		law is invariant and ergodic, and Proposition~\ref{prop:path-reduction}\textup{(iii)}
		gives conclusions~\textup{(ii)} and~\textup{(iii)}.
	\end{proof}
	
	In the remainder of this section we prove
	Proposition~\ref{prop:square-passage}, by mapping live paths to directed open
	paths of a dependent percolation model on the dual lattice, in which long open
	paths are exponentially unlikely.

	Recall that $r_v(w)\in\{1,2,3,4\}$ is the number of rotor advances from the
	initial rotor at $v$ until $v\to w$ is selected.  It numbers the four edges out
	of $v$ in clockwise order, starting from the edge just after the initial rotor,
	and the four possible numberings are equally likely, independently over $v$.  A
	path $x_0,\ldots,x_m$ and an internal vertex $x_i$ satisfy
	\begin{equation}\label{eq:live-numbering}
		\text{the path is live at $x_i$}
		\quad\Longrightarrow\quad
		r_{x_i}(x_{i+1})<r_{x_i}(x_{i-1})\, .
	\end{equation}

	We carry out the proof in three steps.  We say that the dual edge of $v\to w$
	is open when $r_v(w)\in\{2,3\}$.  By \eqref{eq:live-numbering}, a live path
	maps to a directed path of open dual edges (Subsection~\ref{ssec:faces}).  Each
	dual edge is open with probability $1/2$, and they are dependent, so we must
	give an additional argument for the decay in
	\eqref{eq:square-live-path-tail}.  First, no open directed path uses three
	consecutive sides of a square, and percolation avoiding a five-edge pattern
	with that property is subcritical (Subsection~\ref{sec:square-diminishment}).  We then couple the
	dual configuration to an independent Bernoulli percolation of parameter $1/2$
	by a depth-first exploration (Subsection~\ref{ssec:square-exploration}).

	\subsection{Dual edges}\label{ssec:faces}
	
	The dual lattice has the faces of $\Z^2$ as its vertices, with two faces
	adjacent when they share an edge.  Identify the face centered at
	$z+(1/2,1/2)$ with $z\in\Z^2$.
	
	Rotate each directed edge $e=v\to w$ counterclockwise through $90^\circ$ about
	its midpoint.  The resulting dual edge runs from the face on the right of $e$
	to the face on its left, and is \defn{open} when $r_v(w)\in\{2,3\}$.  The four
	dual edges from $v$ form a unit square traversed counterclockwise, with east
	side $E$ directed north, north side $N$ west, west side $W$ south, and south
	side $S$ east.  The four numberings $r_v$ give the open pairs
	$\{E,N\},\{N,W\},\{W,S\},\{S,E\}$, each with probability $1/4$, so exactly one
	of $E,W$ and one of $N,S$ is open, independently and uniformly, and the pairs
	at distinct vertices are independent.  Each dual bond crosses one lattice edge,
	and its two directed copies are controlled by that edge's endpoints
	(Figure~\ref{fig:square-comparison}\textup{(a)}).
	
	This is the dual configuration of the $p=1/2$ directed-corner model of
	\citet*[Section~7]{CoupierHenryJahnelKoppl2024}. We do not invoke any of their results, 
	but adapt their exploration argument in
	Subsection~\ref{ssec:square-exploration}.  In their Section~7 the analogue of
	their Theorem~3.1 bounds the number of faces reachable from a given face by a
	stretched exponential with exponent $1/4$; our exponential bound on the
	distance reached improves that exponent to $1/2$.
	
	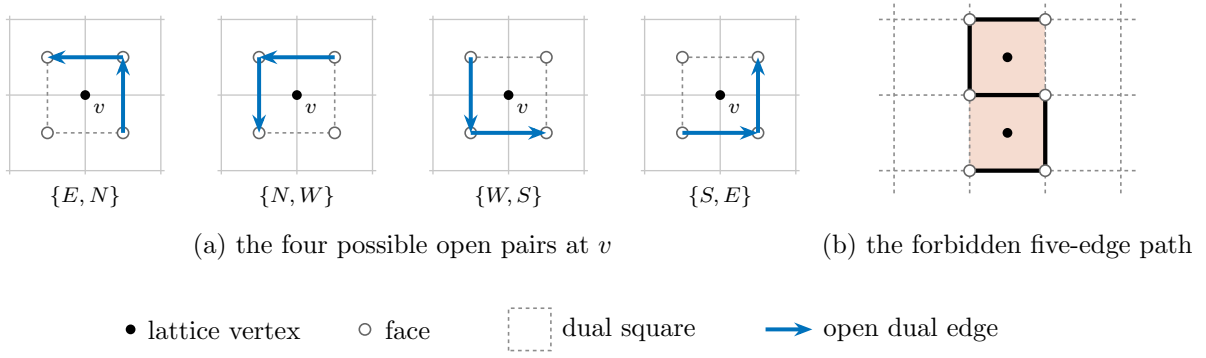
\begin{figure}[t]
		\centering
		\begin{tikzpicture}[
			scale=1.0,
			prim/.style={draw=black!22,line width=0.5pt},
			face/.style={circle,draw=black!60,fill=white,line width=0.6pt,inner sep=1.5pt},
			vert/.style={circle,fill=black,inner sep=1.3pt},
			dual/.style={draw=black!45,line width=0.6pt,dash pattern=on 1.4pt off 1.4pt},
			openedge/.style={-{Stealth[length=2.4mm]},draw=RoyalBlue,line width=1.5pt}]
			\foreach \k in {0,1,2,3} {
				\begin{scope}[xshift={\k*2.80cm}]
					\foreach \x in {-1,0,1} \draw[prim] (\x,-1.05) -- (\x,1.05);
					\foreach \y in {-1,0,1} \draw[prim] (-1.05,\y) -- (1.05,\y);
					\draw[dual] (-0.5,-0.5) rectangle (0.5,0.5);
					\node[vert,label={[font=\scriptsize,inner sep=1.5pt]below right:$v$}] at (0,0) {};
					\foreach \z in {(0.5,0.5),(-0.5,0.5),(-0.5,-0.5),(0.5,-0.5)}
					\node[face] at \z {};
				\end{scope}
			}
			\draw[openedge] (0.5,-0.5) -- (0.5,0.5);
			\draw[openedge] (0.5,0.5) -- (-0.5,0.5);
			\node[font=\scriptsize] at (0,-1.35) {$\{E,N\}$};
			\begin{scope}[xshift=2.80cm]
				\draw[openedge] (0.5,0.5) -- (-0.5,0.5);
				\draw[openedge] (-0.5,0.5) -- (-0.5,-0.5);
				\node[font=\scriptsize] at (0,-1.35) {$\{N,W\}$};
			\end{scope}
			\begin{scope}[xshift=5.60cm]
				\draw[openedge] (-0.5,0.5) -- (-0.5,-0.5);
				\draw[openedge] (-0.5,-0.5) -- (0.5,-0.5);
				\node[font=\scriptsize] at (0,-1.35) {$\{W,S\}$};
			\end{scope}
			\begin{scope}[xshift=8.40cm]
				\draw[openedge] (-0.5,-0.5) -- (0.5,-0.5);
				\draw[openedge] (0.5,-0.5) -- (0.5,0.5);
				\node[font=\scriptsize] at (0,-1.35) {$\{S,E\}$};
			\end{scope}
			\node[font=\small] at (4.20,-2.0)
			{\textup{(a)} the four possible open pairs at $v$};
			\begin{scope}[xshift=11.7cm,yshift=-1.0cm,scale=1.0]
				\foreach \x in {-1,0,1,2} \draw[dual] (\x,-0.3) -- (\x,2.3);
				\foreach \y in {0,1,2} \draw[dual] (-1.2,\y) -- (2.2,\y);
				\fill[BrickRed!13] (0,0) rectangle (1,1);
				\fill[BrickRed!13] (0,1) rectangle (1,2);
				\draw[dual] (0,0) rectangle (1,1);
				\draw[dual] (0,1) rectangle (1,2);
				\node[vert] at (0.5,0.5) {};
				\node[vert] at (0.5,1.5) {};
				\draw[draw=black,line width=1.6pt]
				(0,0) -- (1,0) -- (1,1) -- (0,1) -- (0,2) -- (1,2);
				\foreach \z in {(0,0),(1,0),(1,1),(0,1),(0,2),(1,2)}
				\node[face] at \z {};
			\end{scope}
			\node[font=\small] at (12.2,-2.0)
			{\textup{(b)} the forbidden five-edge path};
			\begin{scope}[yshift=-3.1cm,xshift=0.6cm]
				\node[vert] at (0,0) {};
				\node[right,font=\small,inner sep=4pt] at (0.08,0) {lattice vertex};
				\node[face] at (3.1,0) {};
				\node[right,font=\small,inner sep=4pt] at (3.22,0) {face};
				\draw[dual] (5.0,-0.28) rectangle (5.56,0.28);
				\node[right,font=\small,inner sep=4pt] at (5.56,0) {dual square};
				\draw[openedge] (8.4,0) -- (9.02,0);
				\node[right,font=\small,inner sep=4pt] at (9.02,0) {open dual edge};
			\end{scope}
		\end{tikzpicture}
		\caption{The dual configuration and forbidden path.  \textup{(a)} The four
			possible pairs of open dual edges around a lattice vertex, each with
			probability $1/4$.  \textup{(b)} The path in
			\eqref{eq:square-five-edge-path}. }
		\label{fig:square-comparison}
	\end{figure}
	
	\begin{lemma}[Live paths give open dual paths]
		\label{lem:square-dual-path}
		Let $m\geq1$ be an integer and let $x_0,\ldots,x_m$ be a live nearest-neighbor
		path in $\Z^2$.
		Then there is a directed path of open dual edges from the face on the right of
		$x_0\to x_1$ to the face on the right of $x_{m-1}\to x_m$.
	\end{lemma}
	
	\begin{proof}
		At a transition $p\to v\to w$, \eqref{eq:live-numbering} gives
		$r_v(w)<r_v(p)$.  If $r_v(w)=r_v(p)-1$, the face on the right of $p\to v$ and
		the face on the right of $v\to w$ coincide.  Otherwise the dual edges at
		rotor positions $r_v(p)-1,r_v(p)-2,\ldots,r_v(w)+1$ lead from the first face
		to the second, and every such position lies strictly between $1$ and $4$,
		hence in $\{2,3\}$, so every such edge is open.  Concatenating over the
		internal vertices proves the lemma.
	\end{proof}
	
	\subsection{A forbidden path in critical percolation}
	\label{sec:square-diminishment}
	
	Write $\P_p$ for Bernoulli bond percolation on $\Z^2$ with parameter $p$.
	We use $p_c(\Z^2)=1/2$ \citep{Kesten1980} and the standard subcritical bound
	\citep[Theorem~3.4]{Grimmett1989}: for each $p<1/2$, uniformly in $x$ and
	$r\geq1$,
	\begin{equation}\label{eq:square-subcritical-tail}
		\P_p\left\{x\text{ is joined to }\{y:|y-x|_\infty=r\}
		\text{ inside }\{y:|y-x|_\infty\leq r\}\right\}\leq Ce^{-cr}\, .
	\end{equation}
	
	Let $P_\star$ be the five-edge path in the dual lattice, identified with $\Z^2$
	as in Subsection~\ref{ssec:faces}, with vertices
	\begin{equation}\label{eq:square-five-edge-path}
		P_\star\coloneqq\bigl((0,0),(1,0),(1,1),(0,1),(0,2),(1,2)\bigr)\, ,
	\end{equation}
	drawn in Figure~\ref{fig:square-comparison}\textup{(b)}.  Note
	that by definition, no open directed dual path in this model
	contains a translate of $P_\star$ in either direction.
	
	\begin{lemma}[Percolation avoiding a pattern]
		\label{lem:square-constrained-bonds}
		There are constants $c,C>0$ such that the following holds for every $x\in\Z^2$
		and every integer $r\geq1$.  For bond percolation on $\Z^2$ with parameter
		$1/2$,
		\begin{equation}\label{eq:square-constrained-bond-tail}
			\P_{1/2}\left\{
			\begin{array}{c}
				\text{there is an open path from $x$ to $\{y:|y-x|_\infty=r\}$ inside}\\
				\text{$\{y:|y-x|_\infty\leq r\}$ containing no translate of $P_\star$,}\\
				\text{in either direction, as a consecutive subpath}
			\end{array}
			\right\}
			\leq Ce^{-cr}\, .
		\end{equation}
	\end{lemma}
	
	\begin{proof}
		We adapt the finite-volume pivotal comparison of
		\citet[Section~6, especially Theorem~6.1 and
		Proposition~6.3]{CoupierHenryJahnelKoppl2024}.  Roughly, cover $\Z^2$ by square
		blocks of side $20$, place an independently marked translate of $P_\star$ in
		each, and use the marks to interpolate between constrained and ordinary bond
		percolation.  A modification changing a uniformly bounded number of bonds turns
		a pivotal bond into a pivotal mark in the same block.  Russo's formula and
		integration then compare the constrained model at $p=1/2$ with independent subcritical
		bond percolation.

		Define
		\[
		U_z\coloneqq20z+\bigl([0,20]^2\cap\Z^2\bigr),
		\qquad z\in\Z^2\, .
		\]
		Fix $y\in\Z^2$ with $|y-x|_\infty=r$.  In
		each $U_z$, place three vertex-disjoint translates of $P_\star$ at graph distance at least four from
		$\Z^2\setminus U_z$, choose one containing neither $x$ nor $y$, and denote
		it by $P_z$.
		
		Let $D_n$ be the subgraph induced by the blocks $U_z$ with $|z|_\infty\leq n$,
		with $n$ large enough that $D_n$ contains the radius-$r$ box about $x$.  Assign to
		each $P_z\subset D_n$ a Bernoulli mark $\sigma_z$ of parameter $q$,
		independent of one another and of the bonds, and write $\P_{p,q}$ for the
		resulting law.  Let $\theta(p,q)$ be the $\P_{p,q}$-probability that $D_n$
		contains an open $x$--$y$ path $\gamma$ such that, whenever $\gamma$ traverses
		a selected copy $P_z$ consecutively in either direction, $\sigma_z=1$.  Thus
		$q=0$ forbids every selected $P_z$, whereas
		$\theta(p,1)=\P_p\{x\longleftrightarrow y\text{ in }D_n\}$.
		
		\emph{Step~1.}  We prove that, for a universal $C_0<\infty$,
		\begin{equation}\label{eq:square-derivative-comparison}
			\partial_p\theta(p,q)
			\leq C_0\partial_q\theta(p,q),
			\qquad \frac14\leq p\leq\frac34,\quad 0\leq q\leq1\, .
		\end{equation}
		For the open $x$--$y$ connection event defining $\theta(p,q)$, call a bond or
		mark \defn{pivotal} if flipping its state changes whether the event occurs.  The
		event is increasing in every bond and mark, so Russo's formula gives
		\[
		\partial_p\theta(p,q)
		=\sum_{e\in E(D_n)}\P_{p,q}\{e\text{ is pivotal}\},
		\qquad
		\partial_q\theta(p,q)
		=\sum_{z:P_z\subset D_n}\P_{p,q}\{\sigma_z\text{ is pivotal}\}\, .
		\]
		
		\begin{figure}[t]
			\centering
			\begin{tikzpicture}[scale=0.55,
				lat/.style={draw=black!22,line width=0.6pt},
				shut/.style={draw=BrickRed!32,line width=0.6pt,
					dash pattern=on 1.2pt off 1.4pt},
				opath/.style={draw=RoyalBlue,line width=1.6pt},
				piv/.style={draw=RoyalBlue,line width=3.0pt},
				pat/.style={draw=BrickRed,line width=1.9pt},
				blk/.style={draw=black!55,line width=0.8pt,rounded corners=2pt},
				endv/.style={circle,draw=RoyalBlue,fill=white,line width=0.9pt,
					inner sep=1.5pt},
				mk/.style={circle,fill=black,inner sep=1.2pt}]

				\begin{scope}
					\foreach \i in {0,...,7} {
						\draw[lat] (\i,0) -- (\i,7);
						\draw[lat] (0,\i) -- (7,\i);
					}
					\draw[blk] (-0.45,-0.45) rectangle (7.45,7.45);
					\node[below,font=\small] at (3.5,-0.62) {$U_z$};
					\draw[opath] (0,5) -- (2,5) -- (2,6) -- (5,6);
					\draw[piv] (5,6) -- (5,5);
					\draw[opath] (5,5) -- (6,5) -- (6,1) -- (7,1);
					\draw[opath] (-1.7,5.9) .. controls (-1.0,5.9) and (-0.6,5) .. (0,5);
					\draw[opath] (7,1) .. controls (7.7,1) and (8.1,1.7) .. (8.7,1.7);
					\node[left,font=\small] at (-1.75,5.9) {$x$};
					\node[right,font=\small] at (8.75,1.7) {$y$};
					\draw[pat,draw=black!42] (3,2) -- (4,2) -- (4,3) -- (3,3)
						-- (3,4) -- (4,4);
					\node[font=\small,text=black!55,left] at (2.9,3.5) {$P_z$};
					\node[mk] at (0,5) {};  \node[mk] at (7,1) {};
					\node[above left,font=\small,inner sep=1.5pt] at (0,5) {$s$};
					\node[below right,font=\small,inner sep=1.5pt] at (7,1) {$t$};
					\node[right,font=\small,text=RoyalBlue,inner sep=3pt]
						at (5,5.5) {$e$};
				\end{scope}

				\begin{scope}[xshift=11.6cm]
					\foreach \i in {0,...,7} {
						\draw[shut] (\i,0) -- (\i,7);
						\draw[shut] (0,\i) -- (7,\i);
					}
					\draw[blk] (-0.45,-0.45) rectangle (7.45,7.45);
					\node[below,font=\small] at (3.5,-0.62) {$U_z$};
					\draw[opath] (0,5) -- (0,2) -- (3,2);
					\draw[opath] (4,4) -- (6,4) -- (6,1) -- (7,1);
					\draw[pat,draw=RoyalBlue] (3,2) -- (4,2) -- (4,3) -- (3,3)
						-- (3,4) -- (4,4);
					\node[font=\small,text=RoyalBlue,left] at (2.9,3.5) {$P_z$};
					\draw[opath] (-1.7,5.9) .. controls (-1.0,5.9) and (-0.6,5) .. (0,5);
					\draw[opath] (7,1) .. controls (7.7,1) and (8.1,1.7) .. (8.7,1.7);
					\node[left,font=\small] at (-1.75,5.9) {$x$};
					\node[right,font=\small] at (8.75,1.7) {$y$};
					\node[endv] at (3,2) {};  \node[endv] at (4,4) {};
					\node[mk] at (0,5) {};  \node[mk] at (7,1) {};
					\node[above left,font=\small,inner sep=1.5pt] at (0,5) {$s$};
					\node[below right,font=\small,inner sep=1.5pt] at (7,1) {$t$};
				\end{scope}
			\end{tikzpicture}
			\caption{Turning a pivotal bond into a pivotal mark.  Left: an open
				$x$--$y$ path through $e$ enters $U_z$ at $s$ and leaves at $t$;
				$P_z$ is gray.  Right: the portions from $x$ to $s$ and from $t$ to
				$y$ are kept and joined through $P_z$; every other bond with an endpoint
				in $U_z$ is closed.  Open bonds are blue and closed bonds red and dashed.}
			\label{fig:square-pivotal}
		\end{figure}
		
		Assign each bond of $D_n$ to one block containing both endpoints, breaking ties
		by a fixed rule.  Since a block contains $840$ bonds, at most $840$ are assigned
		to it.  Also fix an ordering of the simple paths in $D_n$.  Given a pivotal bond
		$e$, let $U_z$ be its assigned block.  With $e$ open, let $\gamma$ be the first
		open $x$--$y$ path in this ordering that satisfies the restriction defining
		$\theta(p,q)$.  Pivotality forces $\gamma$ to use $e$.

		Let $s$ and $t$ be the first and last vertices of $\gamma$ in $U_z$.  They are
		distinct, and each is $x$, $y$, or a boundary vertex of $U_z$, so neither lies
		on $P_z$.  Removing the four internal vertices of $P_z$ leaves a $2$-connected
		graph.  Menger's theorem therefore gives vertex-disjoint paths from $s$ and $t$
		to the two endpoints of $P_z$, meeting $P_z$ only at those endpoints.  Choose
		one such pair by a fixed rule and modify the configuration as in
		Figure~\ref{fig:square-pivotal}.

		The modification leaves every $P_{z'}$, $z'\ne z$, unchanged and makes the
		constructed route through $P_z$ the only open connection across $U_z$ between
		the kept portions of $\gamma$.  If $\sigma_z=1$, this route gives the required
		open $x$--$y$ path.  If $\sigma_z=0$, any path satisfying the restriction must
		avoid the route and would therefore already exist with $e$ closed.  Thus
		$\sigma_z$ is pivotal.  The modification affects a fixed number of bonds, so
		for $p\in[1/4,3/4]$ the probability ratio is uniformly bounded, and each
		resulting configuration has at most a fixed number of preimages.  Summing over
		pivotal bonds proves \eqref{eq:square-derivative-comparison}.
		
		\emph{Step~2.}  We prove that, for some $\varepsilon>0$,
		\begin{equation}\label{eq:square-theta-domination}
			\theta(1/2,0)\leq\P_{1/2-\varepsilon}\{x\longleftrightarrow y\text{ in }D_n\}\, .
		\end{equation}
		Choose $0<\varepsilon<\min\{1/4,1/C_0\}$.  Along
		$(p,q)=(1/2-\varepsilon u,u)$, the chain rule gives
		\[
		\frac d{du}\theta(1/2-\varepsilon u,u)
		=-\varepsilon\partial_p\theta(1/2-\varepsilon u,u)
		 +\partial_q\theta(1/2-\varepsilon u,u)
		\geq(1-C_0\varepsilon)\partial_q\theta(1/2-\varepsilon u,u)
		\geq0\, .
		\]
		Integrating this inequality from $u=0$ to $u=1$ gives
		$\theta(1/2,0)\leq\theta(1/2-\varepsilon,1)$, which is
		\eqref{eq:square-theta-domination}.
		
		\emph{Step~3.}  Stop an $x$--$y$ path when it first reaches the boundary of the
		radius-$r$ box about $x$.  Equations~\eqref{eq:square-theta-domination}
		and~\eqref{eq:square-subcritical-tail} give, uniformly in $y$,
		$\theta(1/2,0)\leq Ce^{-cr}$.  A path in
		\eqref{eq:square-constrained-bond-tail} ending at $y$ contains no selected
		$P_z$, so its existence has probability at most $Ce^{-cr}$.  Summing over the
		$8r$ possible endpoints, and changing $c$ and $C$, proves
		\eqref{eq:square-constrained-bond-tail}.
	\end{proof}
	
	\subsection{A depth-first exploration}\label{ssec:square-exploration}
	
	Here we prove Proposition~\ref{prop:square-passage}.  A live path yields a
	directed path of open dual edges (Lemma~\ref{lem:square-dual-path}), and
	Lemma~\ref{lem:square-constrained-bonds} makes long open paths unlikely, but
	for independent bonds.  The dual edges are not independent, so
	the estimate has to be transferred to them.  We do this by revealing the dual
	configuration one edge at a time, in depth-first order, and comparing each
	revealed edge with a fair coin.  The comparison fails only at a forced edge,
	one whose opposite side has already been revealed closed, and the depth-first
	order leaves such an edge as the only one still to be tested, which makes
	forced edges exponentially rare.
	Lemma~\ref{lem:square-constrained-bonds} gives an exponential tail for the
	distance covered between consecutive forced edges, and the two tails together
	bound how far the exploration reaches, which proves
	Proposition~\ref{prop:square-passage}.

	We define the exploration as follows.  Fix a face $f$ and a directed dual edge
	leaving it.  Initially the visited set
	is $\{f\}$, and the active list consists of the four edges leaving $f$, in
	counterclockwise order beginning with the chosen edge.  At each step, call the
	first active edge the current edge and reveal whether it is open or closed; we
	call this \defn{testing} the edge.  Remove it from the active list.  If it is open,
	visit its head and prepend, in right-turn, straight, left-turn order, the three
	directed edges leaving that face other than the reverse of the current edge.
	After either outcome, delete every active edge whose head is visited or lies in
	a finite component of the complement of the visited set.  Continue until the
	list is empty.
	Whenever the current edge is the rotation of $v\to w$, relabel the square formed
	by the rotations of the four edges out of $v$, writing $E$ for the current edge
	and $N,W,S$ for the other three sides in counterclockwise order.  
	
	\begin{lemma}[The exploration]\label{lem:square-exploration}
		The following statements hold.
		\begin{enumerate}[label=\textup{(\roman*)}]
			\item No bond is tested twice.
			\item If the exploration terminates, then every face reachable from $f$ by
			an open directed path is visited or lies in a finite component of the
			complement of the visited set.
			\item Given the outcomes of all earlier tests, $E$ is open with probability $1$,
			$0$, or $1/2$ according as $W$ was tested closed, tested open, or not
			tested.
		\end{enumerate}
	\end{lemma}
	
	\begin{proof}
		These follow immediately from the definition of the exploration.
	\end{proof}
	
	With the notation of Lemma~\ref{lem:square-exploration}\textup{(iii)}, call the
	test of $E$ \defn{forced} if $W$ was tested closed at an earlier step.  In that
	case $E$ is necessarily open.  Let
	$K\in\Z_{\geq0}\cup\{\infty\}$ be the number of forced tests.  These are the
	only tests at which, given the earlier outcomes, the current edge has probability
	greater than $1/2$ of being open.

	\begin{lemma}\label{lem:square-active-list}
		If $W$ or $S$ was tested closed before $E$ is tested, then the active list is
		$(E)$ immediately before $E$ is selected for testing.  At most one of $W$ and $S$ was tested
		closed earlier.
	\end{lemma}

	\begin{proof}
		We prove the first assertion by ordering the active edges around the tree
		formed by the edges that first visited each face.  This order produces a Jordan curve separating
		the head of $E$ from every other active head; the deletion rule in the exploration then gives the
		claim.  The second assertion will follow from the first.

		Immediately before $E$ is tested, every visited face other than $f$ has a unique
		edge that first visited it, and the bonds underlying these edges form a finite
		plane tree $T$ rooted at $f$.  Walk counterclockwise around $T$, keeping $T$ on
		the left, so that each side of each tree edge is traversed once; call this
		closed walk the \defn{contour}.  At each vertex the tree edges divide the
		surrounding directions into gaps, and the contour passes once through each gap.
		Every active edge and every edge already tested closed leaves its tail through
		one gap; call the corresponding place on the contour its \defn{contour
		position}.  Read the contour from the gap at $f$ preceding the chosen initial
		edge.

		Retain the contour position of an edge tested closed, even if its head is later
		visited, and erase the position of an active edge when that edge is deleted.
		In contour order, the retained closed positions form an initial segment,
		followed by the active positions in active-list order.  This holds initially
		and is preserved by each operation: a closed test turns the first active
		position into the last closed position; an open test replaces the first active
		position by the three new positions in right-turn, straight, left-turn order;
		and a deletion erases an active position.

		Choose $\sigma\in\{W,S\}$ that was tested closed, taking $\sigma=S$ if both
		were.  Its contour position is retained.
		Its bond is not in $T$: every bond of $T$ was tested open when it first brought
		the exploration to a face, and no bond is tested twice
		(Lemma~\ref{lem:square-exploration}\textup{(i)}).

		Let $P$ be the unique path in $T$ joining the tails of $\sigma$ and $E$.  Join
		the tail of $\sigma$ to the tail of $E$ by an arc through the interior of the
		square with sides $E,N,W,S$.  Choose the arc to meet $T$ only at its endpoints
		and to approach them through the contour positions of $\sigma$ and $E$.  Its
		union with $P$ is a Jordan curve $J$
		(Figure~\ref{fig:square-jordan}).  The contour segment from $\sigma$ to $E$
		contains no other active position, while the complementary segment contains
		them all.  These segments run along opposite sides of $P$.  Each active edge
		leaves $T$ through its contour position, hence on the side of $P$ containing
		that position, and it crosses neither $P$, since no two dual edges cross, nor
		the interior arc, since no dual edge meets the interior of a square.  Hence $J$ separates the head of $E$ from every other
		active head.

	\begin{figure}[!ht]
		\centering
		\definecolor{cbBlue}{HTML}{0072B2}%
		\definecolor{cbVerm}{HTML}{D55E00}%
		\definecolor{cbGreen}{HTML}{009E73}%
		\begin{tikzpicture}[scale=0.8,
			grid/.style={draw=black!14,line width=0.45pt},
			jhl/.style={draw=black!15,line width=5pt,line cap=round,
				line join=round},
			tre/.style={-{Stealth[length=1.8mm]},draw=black!50,line width=0.85pt},
			arc/.style={draw=black,line width=1.3pt,dash pattern=on 3pt off 2pt},
			cur/.style={-{Stealth[length=2.6mm]},draw=cbBlue,line width=1.7pt},
			clo/.style={-{Stealth[length=2.4mm]},draw=cbVerm,line width=1.3pt,
				dash pattern=on 2.2pt off 1.8pt},
			act/.style={-{Stealth[length=2.4mm]},draw=black!75,line width=1.1pt},
			head/.style={circle,draw=black!75,fill=white,line width=0.9pt,
				inner sep=1.4pt},
			fc/.style={circle,fill=black!70,inner sep=1.0pt},
			rt/.style={circle,fill=black,inner sep=1.6pt},
			lb/.style={font=\small,inner sep=2.5pt}]
			\foreach \x in {0,...,7} \draw[grid] (\x,-1.4) -- (\x,3.4);
			\foreach \y in {-1,...,3} \draw[grid] (-0.4,\y) -- (7.4,\y);
			\fill[black!7] (6,1) -- (6,0) -- (2,0) -- (2,2) -- (5,2)
				.. controls (5.35,1.55) and (5.65,1.35) .. (6,1);
			\draw[draw=black!45,line width=0.7pt] (5,1) rectangle (6,2);
			\draw[jhl] (6,1) -- (6,0) -- (2,0) -- (2,2) -- (5,2);
			\foreach \a/\b/\c/\d in {%
				0/0/1/0, 1/0/2/0, 2/0/3/0, 3/0/4/0, 4/0/5/0, 5/0/6/0, 6/0/6/1,
				2/0/2/1, 2/1/2/2, 2/2/3/2, 3/2/4/2, 4/2/5/2,
				1/0/1/-1, 1/-1/0/-1, 4/0/4/-1, 5/0/5/-1,
				2/1/1/1, 3/2/3/3, 4/2/4/3, 6/0/7/0}
				\draw[tre] (\a,\b) -- (\c,\d);
			\foreach \x/\y in {%
				1/0, 2/0, 3/0, 4/0, 5/0, 6/0, 6/1, 2/1, 2/2, 3/2, 4/2, 5/2,
				1/-1, 0/-1, 4/-1, 5/-1, 1/1, 3/3, 4/3, 7/0}
				\node[fc] at (\x,\y) {};
			\draw[arc] (5,2) .. controls (5.35,1.55) and (5.65,1.35) .. (6,1);
			\draw[clo] (5,2) -- (5,1);
			\draw[cur] (6,1) -- (6,2);
			\draw[act] (2,1) -- (3,1);
			\draw[act] (4,0) -- (4,1);
			\node[head] at (3,1) {};
			\node[head] at (4,1) {};
			\node[head,draw=cbBlue] at (6,2) {};
			\node[rt] at (0,0) {};
			\node[lb,below left,inner sep=1.5pt] at (0,0) {$f$};
			\node[lb,left,text=cbVerm] at (5,1.5) {$\sigma$};
			\node[lb,right,text=cbBlue] at (6,1.5) {$E$};
			\node[lb] at (3.5,-0.45) {$J$};
			\node[lb,right,text=black!55] at (3,2.8) {$T$};
		\end{tikzpicture}
		\caption{The Jordan curve in the proof of
			Lemma~\ref{lem:square-active-list}.  Light
			gray arrows: the first-visit tree $T$.  The pale band marks the path in
			$T$ between the tails of $\sigma$ and $E$, which with the dashed arc
			inside their square forms $J$.  Orange dashed: $\sigma$, tested closed.
			Blue: the current edge $E$, whose head lies in the unbounded component of
			$\R^2\setminus J$.  Black: the other active edges, whose heads lie in the
			shaded bounded component.}
		\label{fig:square-jordan}
	\end{figure}
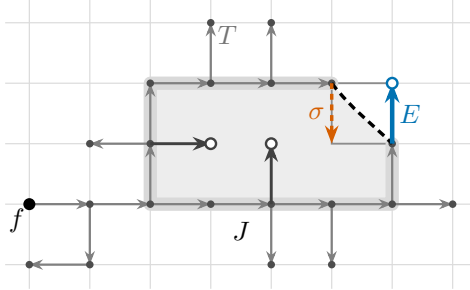
	
		Since $E$ remains active, its head lies in the infinite component of the
		complement of the visited set.  It can therefore be joined to infinity through
		unvisited faces.  Such a path crosses neither the interior arc nor $P$, whose
		vertices are visited, and hence does not meet $J$.  Thus the head of $E$ lies
		in the unbounded component of $\R^2\setminus J$.  Every other active head lies
		in the bounded component and therefore in a finite component of the complement
		of the visited set.  Its edge has consequently been deleted, leaving active
		list $(E)$.

		For the second assertion, suppose that both $W$ and $S$ were tested closed.
		Then $W$ was tested first: the tail of $S$ is the head of $W$, so testing $S$
		first would delete $W$.  At the later test of $S$, relabel the square with $S$
		as the current edge.  The old $W$ is now its $S$-side, so the first assertion
		gives active list $(S)$.  Closing $S$ then terminates the exploration before
		$E$ is tested, a contradiction.
	\end{proof}
	
	\begin{lemma}[Forced tests]
		\label{lem:square-forced-tests}
		Let $K$ be the number of forced tests.  There are constants $c,C>0$ such
		that, for every integer $m\geq1$,
		\begin{equation}\label{eq:square-forced-tail}
			\P_0\{K\geq m\}\leq Ce^{-cm}\, .
		\end{equation}
	\end{lemma}
	
	\begin{proof}
		Fix a forced test and let $g_0$ be its edge.  Condition on the outcomes of all
		tests up to and including it.  We show that the exploration terminates before
		the third forced test after $g_0$ with conditional probability at least $1/4$;
		iterating this bound over successive blocks of three forced tests gives
		\eqref{eq:square-forced-tail}.  We first describe which edges can remain active
		after a forced test, then follow them through three such tests.
		
		Let $g$ be a forced edge obtained from $g_0$ by repeatedly passing to a forced
		right turn, as described below, including $g_0$ itself.  Let $w$ be its head, let $z$ be the tail of the closed
		edge opposite $g$, and let $h$ be the edge $z\to w$.  If $z=f$, then $h$
		belonged to the initial active list.  Otherwise, $h$ was added when $z$ was
		first visited; it is not the reverse of the edge that first visited $z$, since $w$
		is unvisited.  Immediately before $g$ is tested,
		Lemma~\ref{lem:square-active-list} makes it the sole active edge, so $h$ was
		tested or deleted earlier.  It was not opened, since that would have visited
		$w$.  It was not deleted either: had $h$ been deleted, its head $w$ would
		thereafter remain visited or in a finite component of the complement of the
		visited set, whereas $g$, which has the same head, remains active.  Hence $h$
		was tested closed.

		Testing $g$ visits $w$ and adds, in order, the right turn $e$, the straight edge
		$f$, and the reverse of $h$.  The last edge is deleted because its head $z$ is
		visited, so the active list is a sublist of $(e,f)$.  The square of $f$ has $f$ as its east side and $h$
		as its south side.  So if $f$ is tested, then the closed edge $h$ is its
		$S$-side, and Lemma~\ref{lem:square-active-list} shows that $f$ is not forced.  If $e$ is tested and that test is forced, the same lemma
		makes it the sole active edge.  The preceding description then applies after a
		right-angle rotation, with $e$ in the role of $g$.

		Apply this description successively.  Write $e_i$ for the right turn, $f_i$ for
		the straight edge, and $h_i$ for the previously closed edge associated with
		$g_i$, and set $g_{i+1}=e_i$ whenever $e_i$ is tested and forced.  If $g_1$ and
		$g_2$ both occur, the three forced edges trace three sides of one dual square,
		as in Figure~\ref{fig:square-cascade}.
		
		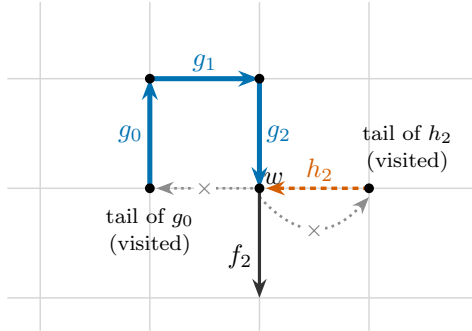
\begin{figure}[H]
			\centering
			\definecolor{cbBlue}{HTML}{0072B2}%
			\definecolor{cbVerm}{HTML}{D55E00}%
			\begin{tikzpicture}[scale=1.45,
				lat/.style={draw=black!16,line width=0.55pt},
				forced/.style={-{Stealth[length=2.8mm]},draw=cbBlue,line width=1.9pt},
				shutedge/.style={-{Stealth[length=2.5mm]},draw=cbVerm,line width=1.3pt,
					dash pattern=on 2.2pt off 1.8pt},
				poss/.style={-{Stealth[length=2.5mm]},draw=black!80,line width=1.2pt},
				del/.style={-{Stealth[length=2.3mm]},draw=black!45,line width=1.0pt,
					dash pattern=on 1pt off 1.6pt},
				hd/.style={circle,fill=black,inner sep=1.3pt},
				lb/.style={font=\small,inner sep=2.5pt}]
				\foreach \i in {-1,0,1,2} \draw[lat] (\i,-1.3) -- (\i,1.7);
				\foreach \j in {-1,0,1} \draw[lat] (-1.3,\j) -- (3.0,\j);
				\draw[forced] (0,0) -- (0,1);
				\draw[forced] (0,1) -- (1,1);
				\draw[forced] (1,1) -- (1,0);
				\draw[shutedge] (2,0) -- (1.06,0);
				\draw[del] (0.94,0) -- (0.06,0);
				\node[font=\scriptsize,text=black!55,fill=white,inner sep=0.8pt]
					at (0.5,0) {$\times$};
				\draw[del] (1,-0.08) .. controls (1.35,-0.46) and (1.65,-0.46)
					.. (2,-0.08);
				\node[font=\scriptsize,text=black!55,fill=white,inner sep=0.8pt]
					at (1.5,-0.39) {$\times$};
				\draw[poss] (1,0) -- (1,-1);
				\node[hd] at (0,0) {};  \node[hd] at (0,1) {};
				\node[hd] at (1,1) {};  \node[hd] at (1,0) {};
				\node[hd] at (2,0) {};
				\node[lb,left,text=cbBlue]  at (0,0.5)  {$g_0$};
				\node[lb,above,text=cbBlue] at (0.5,1)  {$g_1$};
				\node[lb,right,text=cbBlue] at (1,0.5)  {$g_2$};
				\node[lb,above,text=cbVerm] at (1.55,0) {$h_2$};
				\node[lb,left] at (1,-0.62) {$f_2$};
				\node[lb,above right,inner sep=1.5pt] at (1,0) {$w$};
				\node[lb,below,align=center,font=\scriptsize] at (0,-0.12)
					{tail of $g_0$\\(visited)};
				\node[lb,above,align=center,font=\scriptsize] at (2.35,0.08)
					{tail of $h_2$\\(visited)};
			\end{tikzpicture}
			\caption{Three successive forced tests.  The blue edges satisfy $g_1=e_0$
				and $g_2=e_1$.  From the
				head $w$ of $g_2$, the right turn returns to the tail of $g_0$, while
				the left turn reverses the previously closed edge $h_2$.  Both targets
				are visited, so both turns are deleted, shown crossed out, and only
				$f_2$ can remain.}
			\label{fig:square-cascade}
		\end{figure}

		If, at either of the first two stages, $e_i$ is deleted or is tested without
		being forced, require every edge among $e_i$ and $f_i$ that is tested to close.
		These are at most two nonforced tests, and their closure terminates the
		exploration.

		If instead both $e_0$ and $e_1$ are tested and forced, Figure~\ref{fig:square-cascade}
		shows that after $g_2$ is tested only the nonforced edge $f_2$ can remain.  If
		$f_2$ is tested closed, the exploration terminates before a third subsequent
		forced test.  Thus in every case termination is ensured by the closure of at
		most two nonforced tests.  By
		Lemma~\ref{lem:square-exploration}\textup{(iii)}, each closes with conditional
		probability at least $1/2$, so the termination event has conditional probability
		at least $1/4$.

		It follows that, after every forced test, the conditional probability that three
		more forced tests occur before termination is at most $3/4$.  Iterating gives
		$\P_0\{K\geq3k+1\}\leq(3/4)^k$ for every integer $k\geq0$; monotonicity in $m$
		and a change of constants give \eqref{eq:square-forced-tail}.
	\end{proof}
	
	\begin{proof}[Proof of Proposition~\ref{prop:square-passage}]
		By Lemma~\ref{lem:square-forced-tests}, $K<\infty$ almost surely.  In
		chronological order, the nonforced tests before the first forced test, between
		successive forced tests, and after the last forced test form $K+1$ intervals,
		some possibly empty.  Associate the first interval with $f$ and interval $j\geq2$
		with the head of the $(j-1)$st forced edge.  Let $D_j$ be the supremum of the
		$\ell^1$ distances from this face to the faces first visited during interval
		$j$, with $D_j=0$ if the interval is empty.  Let $\mathcal H_1$ be the initial
		trivial sigma-field and, for $j\geq2$, let $\mathcal H_j$ be the sigma-field
		generated by the exploration through the $(j-1)$st forced test.

		\emph{Step~1.}  We prove that, uniformly in $j$,
		\begin{equation}\label{eq:square-interval-tail}
			\P_0\{D_j\geq s\mid\mathcal H_j\}\leq Ce^{-cs}
			\qquad (s\geq1)\, .
		\end{equation}
		Condition on $\mathcal H_j$.  When an untested bond $b$ is selected, let $p_b$
		be its probability of being open given the preceding outcomes, and take an
		independent uniform variable $U_b$.  Declare the test open when $U_b\leq p_b$
		and the corresponding Bernoulli bond open when $U_b\leq1/2$; assign independent
		Bernoulli-$1/2$ states to all other bonds.  This couples the interval to critical
		bond percolation.  No test in the interval is forced, so $p_b\leq1/2$, and every
		edge that first visits a face during the interval is open in the Bernoulli
		percolation.

		By Lemma~\ref{lem:square-active-list}, these first-visit edges form a tree rooted
		at the starting face.  Its root-to-vertex paths are open, directed, and contain
		no translate of $P_\star$.  Lemma~\ref{lem:square-constrained-bonds} therefore
		gives \eqref{eq:square-interval-tail}.

		\emph{Step~2.}  We observe that, for every integer $R\geq1$,
		\begin{equation}\label{eq:square-reach-tail}
			\P_0\left\{K+\sum_{j=1}^{K+1}D_j\geq R\right\}\leq Ce^{-cR}\, .
		\end{equation}
		Indeed, every visited face is at $\ell^1$ distance at most
		$K+\sum_{j=1}^{K+1}D_j$ from $f$.  Lemma~\ref{lem:square-forced-tests}
		and an iteration of the conditional estimate~\eqref{eq:square-interval-tail} therefore give
		\eqref{eq:square-reach-tail}.

		The visited set is therefore almost surely bounded.  Since no bond is tested
		twice, an infinite exploration would leave every finite ball, so the
		exploration terminates almost surely. Lemma~\ref{lem:square-dual-path} then gives
		\eqref{eq:square-live-path-tail}.
	\end{proof}

	\section{A doubly periodic plane graph on which the walk is transient}
	\label{sec:counterexample}
	
	Recall from Section~\ref{ssec:main-results} that $G_M$ is the square
	lattice with $M$ leaves attached to every lattice vertex.  We show that for
	all large $M$ the walk on $G_M$ has $T(1)=\infty$ with positive probability.
	This implies transience of the walk.

	\begin{proof}[Proof of Proposition~\ref{prop:pendant-counterexample}]
		Start at a lattice vertex $o$.  Deleting each two-step visit to a leaf leaves
		a rotor walk on $\Z^2$ whose initial rotors are independent and point west
		with probability $(M+1)/(M+4)$:
		\[
		\P\{\rho(v)=W\}=\frac{M+1}{M+4},
		\qquad
		\P\{\rho(v)=x\}=\frac1{M+4}
		\quad\text{for }x \in\{N,E,S\}\, .
		\]
		The idea now is that a return of this walk to $o$ would produce a
		finite set $U$ at which the steps of the walk sum to
		zero, forcing at least one third of the rotors in $U$ to not point west.  For large $M$ a counting estimate rules this out:
		\[
		\P\left\{\begin{array}{c}
			\text{some finite $\ell^\infty$-connected set $U$ containing $o$}\\
			\text{has at least $|U|/3$ vertices $v$ with $\rho(v)\ne W$}
		\end{array}\right\}
		\leq\sum_{n\geq1}
		\left(128\left(\frac{3}{M+4}\right)^{1/3}\right)^n<1\, .
		\]
		Fix a rotor configuration in the complement of the event in the above display and suppose for sake 
		of contradiction that the induced walk returns to $o$.  Let $F$ be the directed edges traversed up to and including
		its first return, and set
		$k(v)\coloneqq\#\{v\to w\in F\}$
		(Figure~\ref{fig:rotor-excursion}).  By
		Lemma~\ref{lem:one-circuit}, the edges of $F$ are distinct, and every vertex
		has indegree and outdegree $k(v)$, where $0\leq k(v)\leq4$ and $k(o)=1$.
		Let $U$ be the set of vertices that can be reached from $o$ within
		$\{v:1\leq k(v)\leq3\}$ by steps of $\ell^\infty$ distance one.

	\begin{figure}[!ht]
		\centering
		\begin{tikzpicture}[scale=1.75,
			prim/.style={draw=black!20,line width=0.6pt},
			vert/.style={circle,fill=black,inner sep=1.7pt},
			rot/.style={-{Stealth[length=2.3mm]},draw=black,line width=1.0pt},
			fe/.style={-{Stealth[length=2.4mm]},draw=RoyalBlue,line width=1.3pt},
			num/.style={font=\tiny,inner sep=0.6pt,fill=white,text=RoyalBlue}]
			\foreach \x in {0,1,2,3} \draw[prim] (\x,-0.5) -- (\x,2.5);
			\foreach \y in {0,1,2} \draw[prim] (-0.5,\y) -- (3.5,\y);
			\draw[fe] (0.82,0) -- (0.18,0);      \node[num] at (0.5,0.17)  {1};
			\draw[fe] (0,0.18) -- (0,0.82);      \node[num] at (0.17,0.5)  {2};
			\draw[fe] (0.18,1) -- (0.82,1);      \node[num] at (0.5,1.17)  {3};
			\draw[fe] (1,0.82) -- (1,0.18);      \node[num] at (1.17,0.5)  {12};
			\draw[fe] (1.18,1.08) -- (1.82,1.08);  \node[num] at (1.5,1.26)  {4};
			\draw[fe] (1.82,0.92) -- (1.18,0.92);  \node[num] at (1.5,0.74)  {11};
			\draw[fe] (1.92,1.18) -- (1.92,1.82);  \node[num] at (1.72,1.5)  {5};
			\draw[fe] (2.08,1.82) -- (2.08,1.18);  \node[num] at (2.28,1.5)  {6};
			\draw[fe] (2.18,1.08) -- (2.82,1.08);  \node[num] at (2.5,1.26)  {7};
			\draw[fe] (2.82,0.92) -- (2.18,0.92);  \node[num] at (2.5,0.74)  {8};
			\draw[fe] (2.08,0.82) -- (2.08,0.18);  \node[num] at (2.28,0.5)  {9};
			\draw[fe] (1.92,0.18) -- (1.92,0.82);  \node[num] at (1.72,0.5)  {10};
			\draw[rot] (1,-0.07) -- (1,-0.36);
			\draw[rot] (-0.07,0) -- (-0.36,0);
			\draw[rot] (1.93,0) -- (1.64,0);
			\draw[rot] (3,0.07) -- (3,0.36);
			\draw[rot] (0,1.07) -- (0,1.36);
			\draw[rot] (1,1.07) -- (1,1.36);
			\draw[rot] (1.93,1) -- (1.68,1);
			\draw[rot] (3,0.93) -- (3,0.64);
			\draw[rot] (0.07,2) -- (0.36,2);
			\draw[rot] (1,2.07) -- (1,2.36);
			\draw[rot] (2.07,2) -- (2.36,2);
			\draw[rot] (3,1.93) -- (3,1.64);
			\foreach \x in {0,1,2,3} \foreach \y in {0,1,2} \node[vert] at (\x,\y) {};
			\node[font=\small] at (1.27,-0.19) {$o$};
			\draw[rot] (0.05,-0.95) -- (0.42,-0.95);
			\node[right,font=\scriptsize,inner sep=2.5pt] at (0.42,-0.95) {initial rotor};
			\draw[fe] (1.85,-0.95) -- (2.22,-0.95);
			\node[right,font=\scriptsize,inner sep=2.5pt] at (2.22,-0.95) {traversed edge};
		\end{tikzpicture}
		\caption{An initial rotor configuration in black, and in blue the directed
			edges $F$ traversed by the walk from $o$ up to its first return to $o$,
			numbered in the order they are traversed.}
		\label{fig:rotor-excursion}
	\end{figure}
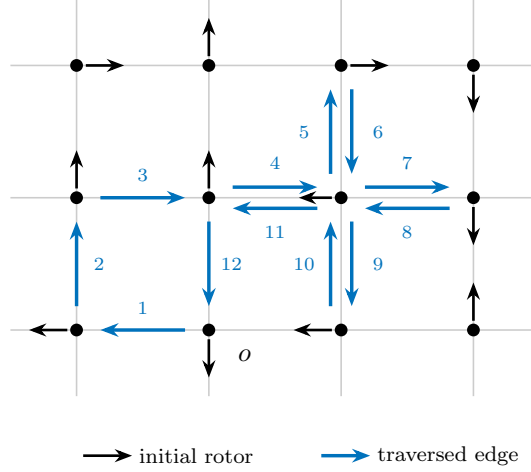

		We claim that 
		\begin{equation}\label{eq:boundary-cancellation}
			\sum_{v\in U}\sum_{v\to w\in F}(w-v)=0\, .
		\end{equation}
		Assuming~\eqref{eq:boundary-cancellation}, the edges in $F$ with tail $v\in U$
		are the first $k(v)$ edges after $\rho(v)$. Taking the sum of the two coordinates in
		\eqref{eq:boundary-cancellation}, a vertex with $\rho(v)=W$ contributes
		$1,2,1$ for $k(v)=1,2,3$, respectively, while every other vertex contributes
		at least $-2$.  Hence
		\[
		0\geq |U|-3\bigl|\{v\in U:\rho(v)\ne W\}\bigr|>0\, ,
		\]
		a contradiction. 

		It remains to prove \eqref{eq:boundary-cancellation}.  Write
		$H\coloneqq\{v:k(v)=4\}$.  Every $F$-edge with exactly one endpoint in $U$
		joins $U$ to $H$.  Each vertex of $H$ uses every outgoing edge, so every edge
		between $H$ and $U$ is traversed from $H$ to $U$.  Since $F$ has equal
		indegree and outdegree at every vertex, each such edge is also traversed from
		$U$ to $H$.  Removing these paired edges leaves a balanced directed graph on
		$U$, and the vectors $w-v$ over its edges sum to zero.  Hence the left-hand
		side of \eqref{eq:boundary-cancellation} equals
		\[
			\sum_{\substack{v\in U,\ h\in H\\ v\sim h}}(h-v)\, .
		\]
		No vertex of $H$ neighbors a vertex with $k=0$, since it uses every outgoing
		edge.  By the definition of $U$, the vectors in the last display are therefore
		the inward unit normals along a union of complete components of the edge
		boundary of $H$.  Their sum is zero by the discrete divergence theorem, proving
		\eqref{eq:boundary-cancellation}.
	\end{proof}

	\section{Open questions}\label{sec:further-questions}
	Our method reduces Theorem~\ref{thm:main} to the almost-live-path estimate
	\eqref{eq:overview-path-estimate}: on a graph where an almost-live path is unlikely to
	reach distance $R$, Proposition~\ref{prop:path-reduction} gives recurrence, a
	deterministic limit shape, and a deterministic limit for $|R_t|t^{-2/3}$.  The estimate is
	sufficient but not necessary.  As discussed in Problem~\ref{prob:planar} 
	below, numerics suggest there are lattices where live paths percolate, 
	yet the conclusions of Theorem~\ref{thm:main} hold. 

	For some lattices the live-path estimate can be checked directly, as we now sketch. 
	On the Manhattan lattice, drawn on the left in Figure~\ref{fig:directed-lattices}, the mirror representation of \citet{FlorescuLevinePeres2016} and the exponential confinement of \citet[Theorem~1.2]{Li2021} give the required block inequality
	of Section~\ref{sec:block} and hence \eqref{eq:overview-path-estimate}, giving a deterministic limit shape and the
	exact $t^{2/3}$ range law.

	By contrast, on the F-lattice, on the right, the estimate fails and so do the
	conclusions.  There the mirror paths are critical bond-percolation interfaces, so circuit
	estimates give passage of order $\log R$ to distance $R$.  The inner and outer
	radii after $n$ circuits grow exponentially at different rates, so the shape
	degenerates and $\lim_{t\to\infty}\log|R_t|/\log t=1$.

	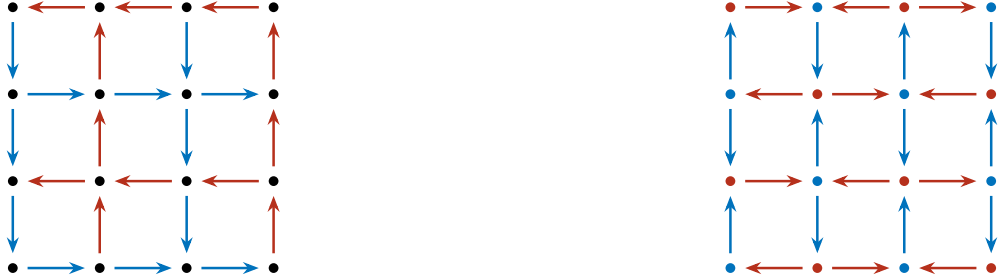
\begin{figure}[!ht]
		\centering
		\begin{tikzpicture}[scale=1.15,
			ev/.style={-{Stealth[length=2.1mm]},draw=RoyalBlue,line width=0.95pt},
			od/.style={-{Stealth[length=2.1mm]},draw=BrickRed,line width=0.95pt},
			v/.style={circle,fill=black,inner sep=1.3pt},
			vev/.style={circle,fill=RoyalBlue,inner sep=1.3pt},
			vod/.style={circle,fill=BrickRed,inner sep=1.3pt}]
			\begin{scope}[xshift=0pt]
				\draw[ev] (0.17,0.00) -- (0.83,0.00);
				\draw[ev] (1.17,0.00) -- (1.83,0.00);
				\draw[ev] (2.17,0.00) -- (2.83,0.00);
				\draw[od] (0.83,1.00) -- (0.17,1.00);
				\draw[od] (1.83,1.00) -- (1.17,1.00);
				\draw[od] (2.83,1.00) -- (2.17,1.00);
				\draw[ev] (0.17,2.00) -- (0.83,2.00);
				\draw[ev] (1.17,2.00) -- (1.83,2.00);
				\draw[ev] (2.17,2.00) -- (2.83,2.00);
				\draw[od] (0.83,3.00) -- (0.17,3.00);
				\draw[od] (1.83,3.00) -- (1.17,3.00);
				\draw[od] (2.83,3.00) -- (2.17,3.00);
				\draw[ev] (0.00,0.83) -- (0.00,0.17);
				\draw[ev] (0.00,1.83) -- (0.00,1.17);
				\draw[ev] (0.00,2.83) -- (0.00,2.17);
				\draw[od] (1.00,0.17) -- (1.00,0.83);
				\draw[od] (1.00,1.17) -- (1.00,1.83);
				\draw[od] (1.00,2.17) -- (1.00,2.83);
				\draw[ev] (2.00,0.83) -- (2.00,0.17);
				\draw[ev] (2.00,1.83) -- (2.00,1.17);
				\draw[ev] (2.00,2.83) -- (2.00,2.17);
				\draw[od] (3.00,0.17) -- (3.00,0.83);
				\draw[od] (3.00,1.17) -- (3.00,1.83);
				\draw[od] (3.00,2.17) -- (3.00,2.83);
				\node[v] at (0,0) {};
				\node[v] at (0,1) {};
				\node[v] at (0,2) {};
				\node[v] at (0,3) {};
				\node[v] at (1,0) {};
				\node[v] at (1,1) {};
				\node[v] at (1,2) {};
				\node[v] at (1,3) {};
				\node[v] at (2,0) {};
				\node[v] at (2,1) {};
				\node[v] at (2,2) {};
				\node[v] at (2,3) {};
				\node[v] at (3,0) {};
				\node[v] at (3,1) {};
				\node[v] at (3,2) {};
				\node[v] at (3,3) {};
			\end{scope}
			\begin{scope}[xshift=0.5\textwidth]
				\draw[od] (0.83,0.00) -- (0.17,0.00);
				\draw[od] (1.17,0.00) -- (1.83,0.00);
				\draw[od] (2.83,0.00) -- (2.17,0.00);
				\draw[od] (0.17,1.00) -- (0.83,1.00);
				\draw[od] (1.83,1.00) -- (1.17,1.00);
				\draw[od] (2.17,1.00) -- (2.83,1.00);
				\draw[od] (0.83,2.00) -- (0.17,2.00);
				\draw[od] (1.17,2.00) -- (1.83,2.00);
				\draw[od] (2.83,2.00) -- (2.17,2.00);
				\draw[od] (0.17,3.00) -- (0.83,3.00);
				\draw[od] (1.83,3.00) -- (1.17,3.00);
				\draw[od] (2.17,3.00) -- (2.83,3.00);
				\draw[ev] (0.00,0.17) -- (0.00,0.83);
				\draw[ev] (0.00,1.83) -- (0.00,1.17);
				\draw[ev] (0.00,2.17) -- (0.00,2.83);
				\draw[ev] (1.00,0.83) -- (1.00,0.17);
				\draw[ev] (1.00,1.17) -- (1.00,1.83);
				\draw[ev] (1.00,2.83) -- (1.00,2.17);
				\draw[ev] (2.00,0.17) -- (2.00,0.83);
				\draw[ev] (2.00,1.83) -- (2.00,1.17);
				\draw[ev] (2.00,2.17) -- (2.00,2.83);
				\draw[ev] (3.00,0.83) -- (3.00,0.17);
				\draw[ev] (3.00,1.17) -- (3.00,1.83);
				\draw[ev] (3.00,2.83) -- (3.00,2.17);
				\node[vev] at (0,0) {};
				\node[vod] at (0,1) {};
				\node[vev] at (0,2) {};
				\node[vod] at (0,3) {};
				\node[vod] at (1,0) {};
				\node[vev] at (1,1) {};
				\node[vod] at (1,2) {};
				\node[vev] at (1,3) {};
				\node[vev] at (2,0) {};
				\node[vod] at (2,1) {};
				\node[vev] at (2,2) {};
				\node[vod] at (2,3) {};
				\node[vod] at (3,0) {};
				\node[vev] at (3,1) {};
				\node[vod] at (3,2) {};
				\node[vev] at (3,3) {};
			\end{scope}
		\end{tikzpicture}
		\caption{The Manhattan lattice, left, and the F-lattice, right, with even in
			blue and odd in red.  The rows and columns of the Manhattan lattice
			alternate in direction with the parity of their index, and each vertex of
			the F-lattice emits its two edges vertically or horizontally according to
			the parity of $v_1+v_2$.}
		\label{fig:directed-lattices}
	\end{figure}

	\begin{problem}[Other planar lattices]\label{prob:planar}
		For which doubly periodic plane graphs does Theorem~\ref{thm:main} hold?  The
		graph $G_M$ of Proposition~\ref{prop:pendant-counterexample} gives a
		counterexample in general. However, $G_M$ does not have a harmonic embedding
		as a plane graph.  Recall that an embedding is harmonic when every
		vertex is the barycenter of its neighbors. We expect Theorem~\ref{thm:main} to hold for harmonic embeddings of doubly
		periodic plane graphs.  Simulations suggest that the rotor mechanism matters as well:
		Figure~\ref{fig:tetrakis-orders} shows that on the tetrakis square lattice the
		clockwise order gives range consistent with the $t^{2/3}$ law, while another
		order on the same graph appears to give linear range.
		
		The main difficulty is that live paths percolate on some planar lattices.  Even on
		the triangular lattice with the clockwise mechanism, and on the square lattice
		with cyclic order $N,S,E,W$, an infinite live path exists with positive
		probability, so Proposition~\ref{prop:live-recurrence} does not apply.
		
		We further caution that recurrence is hard to detect in simulations.  Indeed,
		for critical uniform rotors on the binary tree,
		\citet[Theorem~7\textup{(i)}]{AngelHolroyd2011} prove recurrence although the
		walk reaches depth doubly exponential in $n$ before its $n$th return to the
			root.
		\end{problem}

	\begin{problem}[Monotonicity in the bias]
		Let the initial rotors on the clockwise square lattice be independent, and
		point west with probability $p$ and in each other direction with probability
		$(1-p)/3$.
		Proposition~\ref{prop:small-perturbations} gives recurrence for $p$ small, while the argument of Section~\ref{sec:counterexample} gives transience
		for $p$ near $1$.  Is there a single critical $p$ separating recurrence from
		transience?  More generally, which probability vectors on the four directions
		give a recurrent walk?  We conjecture that they form a set that is star-shaped
		about the uniform vector.  Is it convex?
	\end{problem}

	\begin{problem}[Roundness of the limit shape]
		Is the limit shape $B$ in Theorem~\ref{thm:main} a Euclidean disk on the square
		lattice?  \citet{KapriDhar2009} conjectured this via simulation. Our simulations agree, and suggest that the
		honeycomb, Manhattan, and triangular limit shapes are also disks. 
	\end{problem}

	\begin{problem}[Diffusive scaling in higher dimensions]
		For a fixed translation-invariant rotor mechanism on $\Z^d$ with $d>2$ and
		independent uniform initial rotors, does the walk converge to Brownian motion
		under diffusive scaling?  \citet{PriezzhevDharDharKrishnamurthy1996} predicted
		diffusive mean-square displacement in every dimension $d>2$.
	\end{problem}

		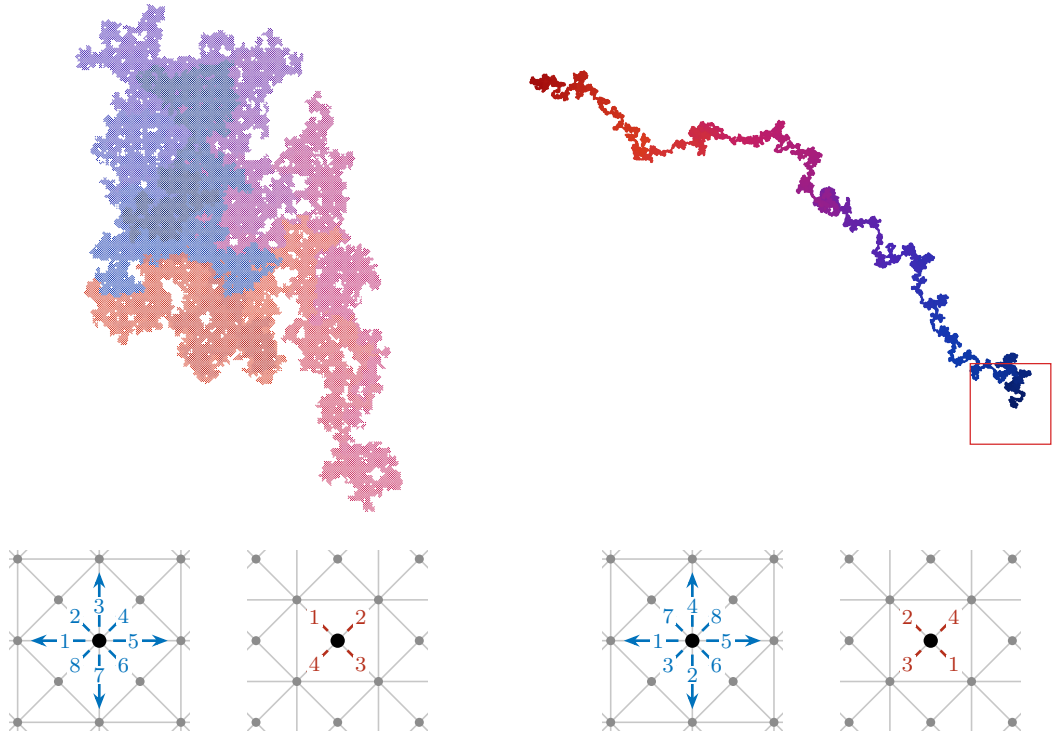
\begin{figure}[H]
		\centering
		\includegraphics[width=0.88\textwidth]{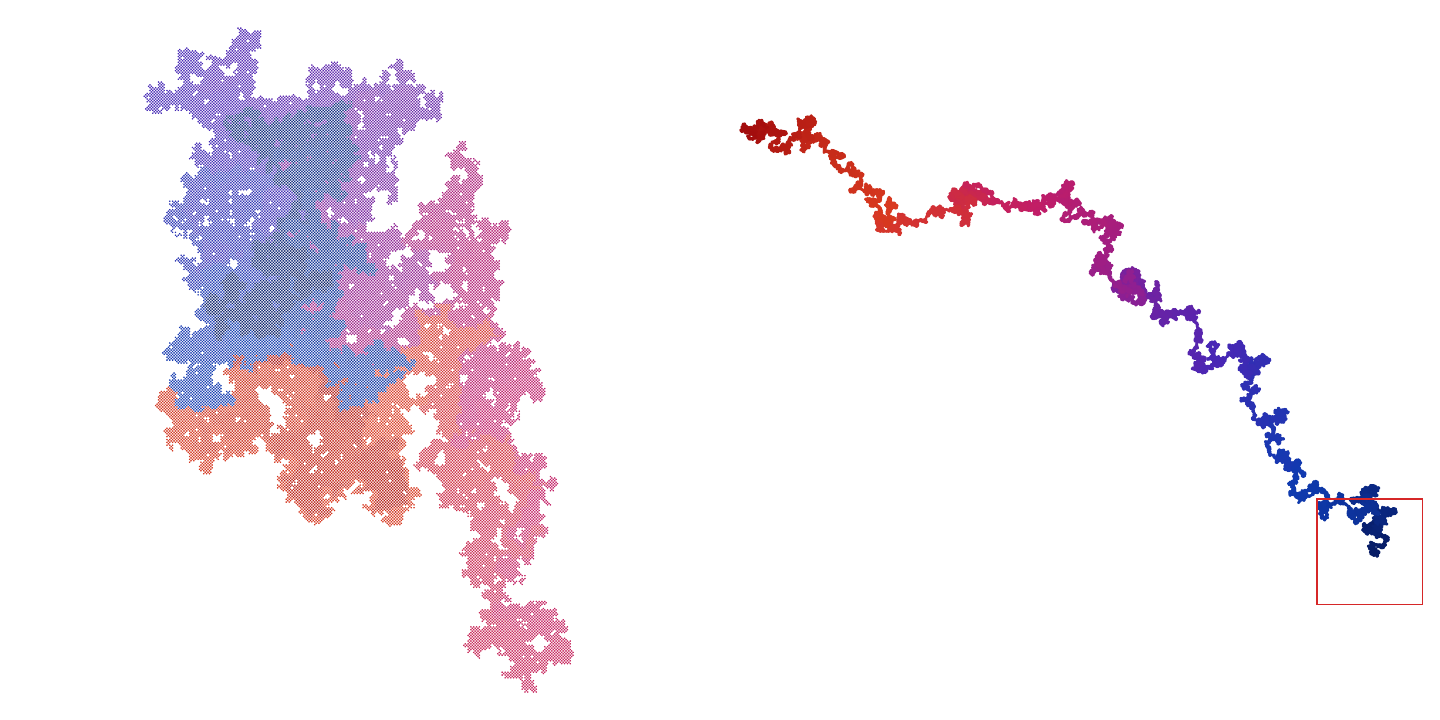}
		\\[6pt]
		\begin{tikzpicture}[scale=1.08,
			lat/.style={draw=black!22,line width=0.6pt},
			big/.style={circle,fill=black,inner sep=1.9pt},
			small/.style={circle,fill=black!45,inner sep=1.2pt},
			ed/.style={-{Stealth[length=2.1mm]},draw=RoyalBlue,line width=1.0pt},
			ed2/.style={-{Stealth[length=2.1mm]},draw=BrickRed,line width=1.0pt},
			num/.style={font=\scriptsize,inner sep=0.8pt,fill=white,text=RoyalBlue},
			num2/.style={font=\scriptsize,inner sep=0.8pt,fill=white,text=BrickRed}]
			\begin{scope}[xshift=0pt-1.46cm]
				\clip (-1.10,-1.10) rectangle (1.10,1.10);
				\draw[lat] (-3.0,-2) -- (-3.0,2);
				\draw[lat] (-2,-3.0) -- (2,-3.0);
				\draw[lat] (-2.0,-2) -- (-2.0,2);
				\draw[lat] (-2,-2.0) -- (2,-2.0);
				\draw[lat] (-1.0,-2) -- (-1.0,2);
				\draw[lat] (-2,-1.0) -- (2,-1.0);
				\draw[lat] (0.0,-2) -- (0.0,2);
				\draw[lat] (-2,0.0) -- (2,0.0);
				\draw[lat] (1.0,-2) -- (1.0,2);
				\draw[lat] (-2,1.0) -- (2,1.0);
				\draw[lat] (2.0,-2) -- (2.0,2);
				\draw[lat] (-2,2.0) -- (2,2.0);
				\draw[lat] (3.0,-2) -- (3.0,2);
				\draw[lat] (-2,3.0) -- (2,3.0);
				\draw[lat] (-1.5,-1.5) -- (-2.0,-2.0);
				\draw[lat] (-1.5,-1.5) -- (-1.0,-2.0);
				\draw[lat] (-1.5,-1.5) -- (-1.0,-1.0);
				\draw[lat] (-1.5,-1.5) -- (-2.0,-1.0);
				\draw[lat] (-1.5,-0.5) -- (-2.0,-1.0);
				\draw[lat] (-1.5,-0.5) -- (-1.0,-1.0);
				\draw[lat] (-1.5,-0.5) -- (-1.0,0.0);
				\draw[lat] (-1.5,-0.5) -- (-2.0,0.0);
				\draw[lat] (-1.5,0.5) -- (-2.0,0.0);
				\draw[lat] (-1.5,0.5) -- (-1.0,0.0);
				\draw[lat] (-1.5,0.5) -- (-1.0,1.0);
				\draw[lat] (-1.5,0.5) -- (-2.0,1.0);
				\draw[lat] (-1.5,1.5) -- (-2.0,1.0);
				\draw[lat] (-1.5,1.5) -- (-1.0,1.0);
				\draw[lat] (-1.5,1.5) -- (-1.0,2.0);
				\draw[lat] (-1.5,1.5) -- (-2.0,2.0);
				\draw[lat] (-0.5,-1.5) -- (-1.0,-2.0);
				\draw[lat] (-0.5,-1.5) -- (0.0,-2.0);
				\draw[lat] (-0.5,-1.5) -- (0.0,-1.0);
				\draw[lat] (-0.5,-1.5) -- (-1.0,-1.0);
				\draw[lat] (-0.5,-0.5) -- (-1.0,-1.0);
				\draw[lat] (-0.5,-0.5) -- (0.0,-1.0);
				\draw[lat] (-0.5,-0.5) -- (0.0,0.0);
				\draw[lat] (-0.5,-0.5) -- (-1.0,0.0);
				\draw[lat] (-0.5,0.5) -- (-1.0,0.0);
				\draw[lat] (-0.5,0.5) -- (0.0,0.0);
				\draw[lat] (-0.5,0.5) -- (0.0,1.0);
				\draw[lat] (-0.5,0.5) -- (-1.0,1.0);
				\draw[lat] (-0.5,1.5) -- (-1.0,1.0);
				\draw[lat] (-0.5,1.5) -- (0.0,1.0);
				\draw[lat] (-0.5,1.5) -- (0.0,2.0);
				\draw[lat] (-0.5,1.5) -- (-1.0,2.0);
				\draw[lat] (0.5,-1.5) -- (0.0,-2.0);
				\draw[lat] (0.5,-1.5) -- (1.0,-2.0);
				\draw[lat] (0.5,-1.5) -- (1.0,-1.0);
				\draw[lat] (0.5,-1.5) -- (0.0,-1.0);
				\draw[lat] (0.5,-0.5) -- (0.0,-1.0);
				\draw[lat] (0.5,-0.5) -- (1.0,-1.0);
				\draw[lat] (0.5,-0.5) -- (1.0,0.0);
				\draw[lat] (0.5,-0.5) -- (0.0,0.0);
				\draw[lat] (0.5,0.5) -- (0.0,0.0);
				\draw[lat] (0.5,0.5) -- (1.0,0.0);
				\draw[lat] (0.5,0.5) -- (1.0,1.0);
				\draw[lat] (0.5,0.5) -- (0.0,1.0);
				\draw[lat] (0.5,1.5) -- (0.0,1.0);
				\draw[lat] (0.5,1.5) -- (1.0,1.0);
				\draw[lat] (0.5,1.5) -- (1.0,2.0);
				\draw[lat] (0.5,1.5) -- (0.0,2.0);
				\draw[lat] (1.5,-1.5) -- (1.0,-2.0);
				\draw[lat] (1.5,-1.5) -- (2.0,-2.0);
				\draw[lat] (1.5,-1.5) -- (2.0,-1.0);
				\draw[lat] (1.5,-1.5) -- (1.0,-1.0);
				\draw[lat] (1.5,-0.5) -- (1.0,-1.0);
				\draw[lat] (1.5,-0.5) -- (2.0,-1.0);
				\draw[lat] (1.5,-0.5) -- (2.0,0.0);
				\draw[lat] (1.5,-0.5) -- (1.0,0.0);
				\draw[lat] (1.5,0.5) -- (1.0,0.0);
				\draw[lat] (1.5,0.5) -- (2.0,0.0);
				\draw[lat] (1.5,0.5) -- (2.0,1.0);
				\draw[lat] (1.5,0.5) -- (1.0,1.0);
				\draw[lat] (1.5,1.5) -- (1.0,1.0);
				\draw[lat] (1.5,1.5) -- (2.0,1.0);
				\draw[lat] (1.5,1.5) -- (2.0,2.0);
				\draw[lat] (1.5,1.5) -- (1.0,2.0);
				\node[small] at (-1.0,-1.0) {};
				\node[small] at (-0.5,-0.5) {};
				\node[small] at (-1.0,0.0) {};
				\node[small] at (-0.5,0.5) {};
				\node[small] at (-1.0,1.0) {};
				\node[small] at (0.0,-1.0) {};
				\node[small] at (0.5,-0.5) {};
				\node[small] at (0.0,0.0) {};
				\node[small] at (0.5,0.5) {};
				\node[small] at (0.0,1.0) {};
				\node[small] at (1.0,-1.0) {};
				\node[small] at (1.0,0.0) {};
				\node[small] at (1.0,1.0) {};
				\draw[ed] (-0.170,0.000) -- (-0.840,0.000);
				\draw[ed] (-0.120,0.120) -- (-0.387,0.387);
				\draw[ed] (0.000,0.170) -- (0.000,0.840);
				\draw[ed] (0.120,0.120) -- (0.387,0.387);
				\draw[ed] (0.170,0.000) -- (0.840,0.000);
				\draw[ed] (0.120,-0.120) -- (0.387,-0.387);
				\draw[ed] (0.000,-0.170) -- (0.000,-0.840);
				\draw[ed] (-0.120,-0.120) -- (-0.387,-0.387);
				\node[big] at (0,0) {};
				\node[num] at (-0.420,0.000) {1};
				\node[num] at (-0.297,0.297) {2};
				\node[num] at (0.000,0.420) {3};
				\node[num] at (0.297,0.297) {4};
				\node[num] at (0.420,0.000) {5};
				\node[num] at (0.297,-0.297) {6};
				\node[num] at (0.000,-0.420) {7};
				\node[num] at (-0.297,-0.297) {8};
			\end{scope}
			\begin{scope}[xshift=0pt+1.46cm]
				\clip (-1.10,-1.10) rectangle (1.10,1.10);
				\draw[lat] (-3.5,-2) -- (-3.5,2);
				\draw[lat] (-2,-3.5) -- (2,-3.5);
				\draw[lat] (-2.5,-2) -- (-2.5,2);
				\draw[lat] (-2,-2.5) -- (2,-2.5);
				\draw[lat] (-1.5,-2) -- (-1.5,2);
				\draw[lat] (-2,-1.5) -- (2,-1.5);
				\draw[lat] (-0.5,-2) -- (-0.5,2);
				\draw[lat] (-2,-0.5) -- (2,-0.5);
				\draw[lat] (0.5,-2) -- (0.5,2);
				\draw[lat] (-2,0.5) -- (2,0.5);
				\draw[lat] (1.5,-2) -- (1.5,2);
				\draw[lat] (-2,1.5) -- (2,1.5);
				\draw[lat] (2.5,-2) -- (2.5,2);
				\draw[lat] (-2,2.5) -- (2,2.5);
				\draw[lat] (-2.0,-2.0) -- (-2.5,-2.5);
				\draw[lat] (-2.0,-2.0) -- (-1.5,-2.5);
				\draw[lat] (-2.0,-2.0) -- (-1.5,-1.5);
				\draw[lat] (-2.0,-2.0) -- (-2.5,-1.5);
				\draw[lat] (-2.0,-1.0) -- (-2.5,-1.5);
				\draw[lat] (-2.0,-1.0) -- (-1.5,-1.5);
				\draw[lat] (-2.0,-1.0) -- (-1.5,-0.5);
				\draw[lat] (-2.0,-1.0) -- (-2.5,-0.5);
				\draw[lat] (-2.0,0.0) -- (-2.5,-0.5);
				\draw[lat] (-2.0,0.0) -- (-1.5,-0.5);
				\draw[lat] (-2.0,0.0) -- (-1.5,0.5);
				\draw[lat] (-2.0,0.0) -- (-2.5,0.5);
				\draw[lat] (-2.0,1.0) -- (-2.5,0.5);
				\draw[lat] (-2.0,1.0) -- (-1.5,0.5);
				\draw[lat] (-2.0,1.0) -- (-1.5,1.5);
				\draw[lat] (-2.0,1.0) -- (-2.5,1.5);
				\draw[lat] (-2.0,2.0) -- (-2.5,1.5);
				\draw[lat] (-2.0,2.0) -- (-1.5,1.5);
				\draw[lat] (-2.0,2.0) -- (-1.5,2.5);
				\draw[lat] (-2.0,2.0) -- (-2.5,2.5);
				\draw[lat] (-1.0,-2.0) -- (-1.5,-2.5);
				\draw[lat] (-1.0,-2.0) -- (-0.5,-2.5);
				\draw[lat] (-1.0,-2.0) -- (-0.5,-1.5);
				\draw[lat] (-1.0,-2.0) -- (-1.5,-1.5);
				\draw[lat] (-1.0,-1.0) -- (-1.5,-1.5);
				\draw[lat] (-1.0,-1.0) -- (-0.5,-1.5);
				\draw[lat] (-1.0,-1.0) -- (-0.5,-0.5);
				\draw[lat] (-1.0,-1.0) -- (-1.5,-0.5);
				\draw[lat] (-1.0,0.0) -- (-1.5,-0.5);
				\draw[lat] (-1.0,0.0) -- (-0.5,-0.5);
				\draw[lat] (-1.0,0.0) -- (-0.5,0.5);
				\draw[lat] (-1.0,0.0) -- (-1.5,0.5);
				\draw[lat] (-1.0,1.0) -- (-1.5,0.5);
				\draw[lat] (-1.0,1.0) -- (-0.5,0.5);
				\draw[lat] (-1.0,1.0) -- (-0.5,1.5);
				\draw[lat] (-1.0,1.0) -- (-1.5,1.5);
				\draw[lat] (-1.0,2.0) -- (-1.5,1.5);
				\draw[lat] (-1.0,2.0) -- (-0.5,1.5);
				\draw[lat] (-1.0,2.0) -- (-0.5,2.5);
				\draw[lat] (-1.0,2.0) -- (-1.5,2.5);
				\draw[lat] (0.0,-2.0) -- (-0.5,-2.5);
				\draw[lat] (0.0,-2.0) -- (0.5,-2.5);
				\draw[lat] (0.0,-2.0) -- (0.5,-1.5);
				\draw[lat] (0.0,-2.0) -- (-0.5,-1.5);
				\draw[lat] (0.0,-1.0) -- (-0.5,-1.5);
				\draw[lat] (0.0,-1.0) -- (0.5,-1.5);
				\draw[lat] (0.0,-1.0) -- (0.5,-0.5);
				\draw[lat] (0.0,-1.0) -- (-0.5,-0.5);
				\draw[lat] (0.0,0.0) -- (-0.5,-0.5);
				\draw[lat] (0.0,0.0) -- (0.5,-0.5);
				\draw[lat] (0.0,0.0) -- (0.5,0.5);
				\draw[lat] (0.0,0.0) -- (-0.5,0.5);
				\draw[lat] (0.0,1.0) -- (-0.5,0.5);
				\draw[lat] (0.0,1.0) -- (0.5,0.5);
				\draw[lat] (0.0,1.0) -- (0.5,1.5);
				\draw[lat] (0.0,1.0) -- (-0.5,1.5);
				\draw[lat] (0.0,2.0) -- (-0.5,1.5);
				\draw[lat] (0.0,2.0) -- (0.5,1.5);
				\draw[lat] (0.0,2.0) -- (0.5,2.5);
				\draw[lat] (0.0,2.0) -- (-0.5,2.5);
				\draw[lat] (1.0,-2.0) -- (0.5,-2.5);
				\draw[lat] (1.0,-2.0) -- (1.5,-2.5);
				\draw[lat] (1.0,-2.0) -- (1.5,-1.5);
				\draw[lat] (1.0,-2.0) -- (0.5,-1.5);
				\draw[lat] (1.0,-1.0) -- (0.5,-1.5);
				\draw[lat] (1.0,-1.0) -- (1.5,-1.5);
				\draw[lat] (1.0,-1.0) -- (1.5,-0.5);
				\draw[lat] (1.0,-1.0) -- (0.5,-0.5);
				\draw[lat] (1.0,0.0) -- (0.5,-0.5);
				\draw[lat] (1.0,0.0) -- (1.5,-0.5);
				\draw[lat] (1.0,0.0) -- (1.5,0.5);
				\draw[lat] (1.0,0.0) -- (0.5,0.5);
				\draw[lat] (1.0,1.0) -- (0.5,0.5);
				\draw[lat] (1.0,1.0) -- (1.5,0.5);
				\draw[lat] (1.0,1.0) -- (1.5,1.5);
				\draw[lat] (1.0,1.0) -- (0.5,1.5);
				\draw[lat] (1.0,2.0) -- (0.5,1.5);
				\draw[lat] (1.0,2.0) -- (1.5,1.5);
				\draw[lat] (1.0,2.0) -- (1.5,2.5);
				\draw[lat] (1.0,2.0) -- (0.5,2.5);
				\draw[lat] (2.0,-2.0) -- (1.5,-2.5);
				\draw[lat] (2.0,-2.0) -- (2.5,-2.5);
				\draw[lat] (2.0,-2.0) -- (2.5,-1.5);
				\draw[lat] (2.0,-2.0) -- (1.5,-1.5);
				\draw[lat] (2.0,-1.0) -- (1.5,-1.5);
				\draw[lat] (2.0,-1.0) -- (2.5,-1.5);
				\draw[lat] (2.0,-1.0) -- (2.5,-0.5);
				\draw[lat] (2.0,-1.0) -- (1.5,-0.5);
				\draw[lat] (2.0,0.0) -- (1.5,-0.5);
				\draw[lat] (2.0,0.0) -- (2.5,-0.5);
				\draw[lat] (2.0,0.0) -- (2.5,0.5);
				\draw[lat] (2.0,0.0) -- (1.5,0.5);
				\draw[lat] (2.0,1.0) -- (1.5,0.5);
				\draw[lat] (2.0,1.0) -- (2.5,0.5);
				\draw[lat] (2.0,1.0) -- (2.5,1.5);
				\draw[lat] (2.0,1.0) -- (1.5,1.5);
				\draw[lat] (2.0,2.0) -- (1.5,1.5);
				\draw[lat] (2.0,2.0) -- (2.5,1.5);
				\draw[lat] (2.0,2.0) -- (2.5,2.5);
				\draw[lat] (2.0,2.0) -- (1.5,2.5);
				\node[small] at (-1.0,-1.0) {};
				\node[small] at (-1.0,0.0) {};
				\node[small] at (-1.0,1.0) {};
				\node[small] at (0.0,-1.0) {};
				\node[small] at (-0.5,-0.5) {};
				\node[small] at (0.0,0.0) {};
				\node[small] at (-0.5,0.5) {};
				\node[small] at (0.0,1.0) {};
				\node[small] at (1.0,-1.0) {};
				\node[small] at (0.5,-0.5) {};
				\node[small] at (1.0,0.0) {};
				\node[small] at (0.5,0.5) {};
				\node[small] at (1.0,1.0) {};
				\draw[ed2] (-0.120,0.120) -- (-0.387,0.387);
				\draw[ed2] (0.120,0.120) -- (0.387,0.387);
				\draw[ed2] (0.120,-0.120) -- (0.387,-0.387);
				\draw[ed2] (-0.120,-0.120) -- (-0.387,-0.387);
				\node[big] at (0,0) {};
				\node[num2] at (-0.283,0.283) {1};
				\node[num2] at (0.283,0.283) {2};
				\node[num2] at (0.283,-0.283) {3};
				\node[num2] at (-0.283,-0.283) {4};
			\end{scope}
			\begin{scope}[xshift=0.44\textwidth-1.46cm]
				\clip (-1.10,-1.10) rectangle (1.10,1.10);
				\draw[lat] (-3.0,-2) -- (-3.0,2);
				\draw[lat] (-2,-3.0) -- (2,-3.0);
				\draw[lat] (-2.0,-2) -- (-2.0,2);
				\draw[lat] (-2,-2.0) -- (2,-2.0);
				\draw[lat] (-1.0,-2) -- (-1.0,2);
				\draw[lat] (-2,-1.0) -- (2,-1.0);
				\draw[lat] (0.0,-2) -- (0.0,2);
				\draw[lat] (-2,0.0) -- (2,0.0);
				\draw[lat] (1.0,-2) -- (1.0,2);
				\draw[lat] (-2,1.0) -- (2,1.0);
				\draw[lat] (2.0,-2) -- (2.0,2);
				\draw[lat] (-2,2.0) -- (2,2.0);
				\draw[lat] (3.0,-2) -- (3.0,2);
				\draw[lat] (-2,3.0) -- (2,3.0);
				\draw[lat] (-1.5,-1.5) -- (-2.0,-2.0);
				\draw[lat] (-1.5,-1.5) -- (-1.0,-2.0);
				\draw[lat] (-1.5,-1.5) -- (-1.0,-1.0);
				\draw[lat] (-1.5,-1.5) -- (-2.0,-1.0);
				\draw[lat] (-1.5,-0.5) -- (-2.0,-1.0);
				\draw[lat] (-1.5,-0.5) -- (-1.0,-1.0);
				\draw[lat] (-1.5,-0.5) -- (-1.0,0.0);
				\draw[lat] (-1.5,-0.5) -- (-2.0,0.0);
				\draw[lat] (-1.5,0.5) -- (-2.0,0.0);
				\draw[lat] (-1.5,0.5) -- (-1.0,0.0);
				\draw[lat] (-1.5,0.5) -- (-1.0,1.0);
				\draw[lat] (-1.5,0.5) -- (-2.0,1.0);
				\draw[lat] (-1.5,1.5) -- (-2.0,1.0);
				\draw[lat] (-1.5,1.5) -- (-1.0,1.0);
				\draw[lat] (-1.5,1.5) -- (-1.0,2.0);
				\draw[lat] (-1.5,1.5) -- (-2.0,2.0);
				\draw[lat] (-0.5,-1.5) -- (-1.0,-2.0);
				\draw[lat] (-0.5,-1.5) -- (0.0,-2.0);
				\draw[lat] (-0.5,-1.5) -- (0.0,-1.0);
				\draw[lat] (-0.5,-1.5) -- (-1.0,-1.0);
				\draw[lat] (-0.5,-0.5) -- (-1.0,-1.0);
				\draw[lat] (-0.5,-0.5) -- (0.0,-1.0);
				\draw[lat] (-0.5,-0.5) -- (0.0,0.0);
				\draw[lat] (-0.5,-0.5) -- (-1.0,0.0);
				\draw[lat] (-0.5,0.5) -- (-1.0,0.0);
				\draw[lat] (-0.5,0.5) -- (0.0,0.0);
				\draw[lat] (-0.5,0.5) -- (0.0,1.0);
				\draw[lat] (-0.5,0.5) -- (-1.0,1.0);
				\draw[lat] (-0.5,1.5) -- (-1.0,1.0);
				\draw[lat] (-0.5,1.5) -- (0.0,1.0);
				\draw[lat] (-0.5,1.5) -- (0.0,2.0);
				\draw[lat] (-0.5,1.5) -- (-1.0,2.0);
				\draw[lat] (0.5,-1.5) -- (0.0,-2.0);
				\draw[lat] (0.5,-1.5) -- (1.0,-2.0);
				\draw[lat] (0.5,-1.5) -- (1.0,-1.0);
				\draw[lat] (0.5,-1.5) -- (0.0,-1.0);
				\draw[lat] (0.5,-0.5) -- (0.0,-1.0);
				\draw[lat] (0.5,-0.5) -- (1.0,-1.0);
				\draw[lat] (0.5,-0.5) -- (1.0,0.0);
				\draw[lat] (0.5,-0.5) -- (0.0,0.0);
				\draw[lat] (0.5,0.5) -- (0.0,0.0);
				\draw[lat] (0.5,0.5) -- (1.0,0.0);
				\draw[lat] (0.5,0.5) -- (1.0,1.0);
				\draw[lat] (0.5,0.5) -- (0.0,1.0);
				\draw[lat] (0.5,1.5) -- (0.0,1.0);
				\draw[lat] (0.5,1.5) -- (1.0,1.0);
				\draw[lat] (0.5,1.5) -- (1.0,2.0);
				\draw[lat] (0.5,1.5) -- (0.0,2.0);
				\draw[lat] (1.5,-1.5) -- (1.0,-2.0);
				\draw[lat] (1.5,-1.5) -- (2.0,-2.0);
				\draw[lat] (1.5,-1.5) -- (2.0,-1.0);
				\draw[lat] (1.5,-1.5) -- (1.0,-1.0);
				\draw[lat] (1.5,-0.5) -- (1.0,-1.0);
				\draw[lat] (1.5,-0.5) -- (2.0,-1.0);
				\draw[lat] (1.5,-0.5) -- (2.0,0.0);
				\draw[lat] (1.5,-0.5) -- (1.0,0.0);
				\draw[lat] (1.5,0.5) -- (1.0,0.0);
				\draw[lat] (1.5,0.5) -- (2.0,0.0);
				\draw[lat] (1.5,0.5) -- (2.0,1.0);
				\draw[lat] (1.5,0.5) -- (1.0,1.0);
				\draw[lat] (1.5,1.5) -- (1.0,1.0);
				\draw[lat] (1.5,1.5) -- (2.0,1.0);
				\draw[lat] (1.5,1.5) -- (2.0,2.0);
				\draw[lat] (1.5,1.5) -- (1.0,2.0);
				\node[small] at (-1.0,-1.0) {};
				\node[small] at (-0.5,-0.5) {};
				\node[small] at (-1.0,0.0) {};
				\node[small] at (-0.5,0.5) {};
				\node[small] at (-1.0,1.0) {};
				\node[small] at (0.0,-1.0) {};
				\node[small] at (0.5,-0.5) {};
				\node[small] at (0.0,0.0) {};
				\node[small] at (0.5,0.5) {};
				\node[small] at (0.0,1.0) {};
				\node[small] at (1.0,-1.0) {};
				\node[small] at (1.0,0.0) {};
				\node[small] at (1.0,1.0) {};
				\draw[ed] (-0.170,0.000) -- (-0.840,0.000);
				\draw[ed] (0.000,-0.170) -- (0.000,-0.840);
				\draw[ed] (-0.120,-0.120) -- (-0.387,-0.387);
				\draw[ed] (0.000,0.170) -- (0.000,0.840);
				\draw[ed] (0.170,0.000) -- (0.840,0.000);
				\draw[ed] (0.120,-0.120) -- (0.387,-0.387);
				\draw[ed] (-0.120,0.120) -- (-0.387,0.387);
				\draw[ed] (0.120,0.120) -- (0.387,0.387);
				\node[big] at (0,0) {};
				\node[num] at (-0.420,0.000) {1};
				\node[num] at (0.000,-0.420) {2};
				\node[num] at (-0.297,-0.297) {3};
				\node[num] at (0.000,0.420) {4};
				\node[num] at (0.420,0.000) {5};
				\node[num] at (0.297,-0.297) {6};
				\node[num] at (-0.297,0.297) {7};
				\node[num] at (0.297,0.297) {8};
			\end{scope}
			\begin{scope}[xshift=0.44\textwidth+1.46cm]
				\clip (-1.10,-1.10) rectangle (1.10,1.10);
				\draw[lat] (-3.5,-2) -- (-3.5,2);
				\draw[lat] (-2,-3.5) -- (2,-3.5);
				\draw[lat] (-2.5,-2) -- (-2.5,2);
				\draw[lat] (-2,-2.5) -- (2,-2.5);
				\draw[lat] (-1.5,-2) -- (-1.5,2);
				\draw[lat] (-2,-1.5) -- (2,-1.5);
				\draw[lat] (-0.5,-2) -- (-0.5,2);
				\draw[lat] (-2,-0.5) -- (2,-0.5);
				\draw[lat] (0.5,-2) -- (0.5,2);
				\draw[lat] (-2,0.5) -- (2,0.5);
				\draw[lat] (1.5,-2) -- (1.5,2);
				\draw[lat] (-2,1.5) -- (2,1.5);
				\draw[lat] (2.5,-2) -- (2.5,2);
				\draw[lat] (-2,2.5) -- (2,2.5);
				\draw[lat] (-2.0,-2.0) -- (-2.5,-2.5);
				\draw[lat] (-2.0,-2.0) -- (-1.5,-2.5);
				\draw[lat] (-2.0,-2.0) -- (-1.5,-1.5);
				\draw[lat] (-2.0,-2.0) -- (-2.5,-1.5);
				\draw[lat] (-2.0,-1.0) -- (-2.5,-1.5);
				\draw[lat] (-2.0,-1.0) -- (-1.5,-1.5);
				\draw[lat] (-2.0,-1.0) -- (-1.5,-0.5);
				\draw[lat] (-2.0,-1.0) -- (-2.5,-0.5);
				\draw[lat] (-2.0,0.0) -- (-2.5,-0.5);
				\draw[lat] (-2.0,0.0) -- (-1.5,-0.5);
				\draw[lat] (-2.0,0.0) -- (-1.5,0.5);
				\draw[lat] (-2.0,0.0) -- (-2.5,0.5);
				\draw[lat] (-2.0,1.0) -- (-2.5,0.5);
				\draw[lat] (-2.0,1.0) -- (-1.5,0.5);
				\draw[lat] (-2.0,1.0) -- (-1.5,1.5);
				\draw[lat] (-2.0,1.0) -- (-2.5,1.5);
				\draw[lat] (-2.0,2.0) -- (-2.5,1.5);
				\draw[lat] (-2.0,2.0) -- (-1.5,1.5);
				\draw[lat] (-2.0,2.0) -- (-1.5,2.5);
				\draw[lat] (-2.0,2.0) -- (-2.5,2.5);
				\draw[lat] (-1.0,-2.0) -- (-1.5,-2.5);
				\draw[lat] (-1.0,-2.0) -- (-0.5,-2.5);
				\draw[lat] (-1.0,-2.0) -- (-0.5,-1.5);
				\draw[lat] (-1.0,-2.0) -- (-1.5,-1.5);
				\draw[lat] (-1.0,-1.0) -- (-1.5,-1.5);
				\draw[lat] (-1.0,-1.0) -- (-0.5,-1.5);
				\draw[lat] (-1.0,-1.0) -- (-0.5,-0.5);
				\draw[lat] (-1.0,-1.0) -- (-1.5,-0.5);
				\draw[lat] (-1.0,0.0) -- (-1.5,-0.5);
				\draw[lat] (-1.0,0.0) -- (-0.5,-0.5);
				\draw[lat] (-1.0,0.0) -- (-0.5,0.5);
				\draw[lat] (-1.0,0.0) -- (-1.5,0.5);
				\draw[lat] (-1.0,1.0) -- (-1.5,0.5);
				\draw[lat] (-1.0,1.0) -- (-0.5,0.5);
				\draw[lat] (-1.0,1.0) -- (-0.5,1.5);
				\draw[lat] (-1.0,1.0) -- (-1.5,1.5);
				\draw[lat] (-1.0,2.0) -- (-1.5,1.5);
				\draw[lat] (-1.0,2.0) -- (-0.5,1.5);
				\draw[lat] (-1.0,2.0) -- (-0.5,2.5);
				\draw[lat] (-1.0,2.0) -- (-1.5,2.5);
				\draw[lat] (0.0,-2.0) -- (-0.5,-2.5);
				\draw[lat] (0.0,-2.0) -- (0.5,-2.5);
				\draw[lat] (0.0,-2.0) -- (0.5,-1.5);
				\draw[lat] (0.0,-2.0) -- (-0.5,-1.5);
				\draw[lat] (0.0,-1.0) -- (-0.5,-1.5);
				\draw[lat] (0.0,-1.0) -- (0.5,-1.5);
				\draw[lat] (0.0,-1.0) -- (0.5,-0.5);
				\draw[lat] (0.0,-1.0) -- (-0.5,-0.5);
				\draw[lat] (0.0,0.0) -- (-0.5,-0.5);
				\draw[lat] (0.0,0.0) -- (0.5,-0.5);
				\draw[lat] (0.0,0.0) -- (0.5,0.5);
				\draw[lat] (0.0,0.0) -- (-0.5,0.5);
				\draw[lat] (0.0,1.0) -- (-0.5,0.5);
				\draw[lat] (0.0,1.0) -- (0.5,0.5);
				\draw[lat] (0.0,1.0) -- (0.5,1.5);
				\draw[lat] (0.0,1.0) -- (-0.5,1.5);
				\draw[lat] (0.0,2.0) -- (-0.5,1.5);
				\draw[lat] (0.0,2.0) -- (0.5,1.5);
				\draw[lat] (0.0,2.0) -- (0.5,2.5);
				\draw[lat] (0.0,2.0) -- (-0.5,2.5);
				\draw[lat] (1.0,-2.0) -- (0.5,-2.5);
				\draw[lat] (1.0,-2.0) -- (1.5,-2.5);
				\draw[lat] (1.0,-2.0) -- (1.5,-1.5);
				\draw[lat] (1.0,-2.0) -- (0.5,-1.5);
				\draw[lat] (1.0,-1.0) -- (0.5,-1.5);
				\draw[lat] (1.0,-1.0) -- (1.5,-1.5);
				\draw[lat] (1.0,-1.0) -- (1.5,-0.5);
				\draw[lat] (1.0,-1.0) -- (0.5,-0.5);
				\draw[lat] (1.0,0.0) -- (0.5,-0.5);
				\draw[lat] (1.0,0.0) -- (1.5,-0.5);
				\draw[lat] (1.0,0.0) -- (1.5,0.5);
				\draw[lat] (1.0,0.0) -- (0.5,0.5);
				\draw[lat] (1.0,1.0) -- (0.5,0.5);
				\draw[lat] (1.0,1.0) -- (1.5,0.5);
				\draw[lat] (1.0,1.0) -- (1.5,1.5);
				\draw[lat] (1.0,1.0) -- (0.5,1.5);
				\draw[lat] (1.0,2.0) -- (0.5,1.5);
				\draw[lat] (1.0,2.0) -- (1.5,1.5);
				\draw[lat] (1.0,2.0) -- (1.5,2.5);
				\draw[lat] (1.0,2.0) -- (0.5,2.5);
				\draw[lat] (2.0,-2.0) -- (1.5,-2.5);
				\draw[lat] (2.0,-2.0) -- (2.5,-2.5);
				\draw[lat] (2.0,-2.0) -- (2.5,-1.5);
				\draw[lat] (2.0,-2.0) -- (1.5,-1.5);
				\draw[lat] (2.0,-1.0) -- (1.5,-1.5);
				\draw[lat] (2.0,-1.0) -- (2.5,-1.5);
				\draw[lat] (2.0,-1.0) -- (2.5,-0.5);
				\draw[lat] (2.0,-1.0) -- (1.5,-0.5);
				\draw[lat] (2.0,0.0) -- (1.5,-0.5);
				\draw[lat] (2.0,0.0) -- (2.5,-0.5);
				\draw[lat] (2.0,0.0) -- (2.5,0.5);
				\draw[lat] (2.0,0.0) -- (1.5,0.5);
				\draw[lat] (2.0,1.0) -- (1.5,0.5);
				\draw[lat] (2.0,1.0) -- (2.5,0.5);
				\draw[lat] (2.0,1.0) -- (2.5,1.5);
				\draw[lat] (2.0,1.0) -- (1.5,1.5);
				\draw[lat] (2.0,2.0) -- (1.5,1.5);
				\draw[lat] (2.0,2.0) -- (2.5,1.5);
				\draw[lat] (2.0,2.0) -- (2.5,2.5);
				\draw[lat] (2.0,2.0) -- (1.5,2.5);
				\node[small] at (-1.0,-1.0) {};
				\node[small] at (-1.0,0.0) {};
				\node[small] at (-1.0,1.0) {};
				\node[small] at (0.0,-1.0) {};
				\node[small] at (-0.5,-0.5) {};
				\node[small] at (0.0,0.0) {};
				\node[small] at (-0.5,0.5) {};
				\node[small] at (0.0,1.0) {};
				\node[small] at (1.0,-1.0) {};
				\node[small] at (0.5,-0.5) {};
				\node[small] at (1.0,0.0) {};
				\node[small] at (0.5,0.5) {};
				\node[small] at (1.0,1.0) {};
				\draw[ed2] (0.120,-0.120) -- (0.387,-0.387);
				\draw[ed2] (-0.120,0.120) -- (-0.387,0.387);
				\draw[ed2] (-0.120,-0.120) -- (-0.387,-0.387);
				\draw[ed2] (0.120,0.120) -- (0.387,0.387);
				\node[big] at (0,0) {};
				\node[num2] at (0.283,-0.283) {1};
				\node[num2] at (-0.283,0.283) {2};
				\node[num2] at (-0.283,-0.283) {3};
				\node[num2] at (0.283,0.283) {4};
			\end{scope}
		\end{tikzpicture}
		\caption{Tetrakis square-lattice ranges after $200{,}000$ steps, colored by
			first-visit time, with the corresponding cyclic orders below.  The clockwise
			order, left, gives radius $178$; the order on the right gives radius $1782$.
			The ranges are scaled separately, and the red square frames the left range.}
		\label{fig:tetrakis-orders}
	\end{figure}
	
	\bibliographystyle{plainnat}
	\bibliography{references}

\begin{thebibliography}{48}
\providecommand{\natexlab}[1]{#1}
\providecommand{\url}[1]{\texttt{#1}}
\expandafter\ifx\csname urlstyle\endcsname\relax
  \providecommand{\doi}[1]{doi: #1}\else
  \providecommand{\doi}{doi: \begingroup \urlstyle{rm}\Url}\fi

\bibitem[Angel and Holroyd(2011)]{AngelHolroyd2011}
Omer Angel and Alexander~E. Holroyd.
\newblock Rotor walks on general trees.
\newblock \emph{SIAM J. Discrete Math.}, 25\penalty0 (1):\penalty0 423--446,
  2011.
\newblock \doi{10.1137/100814299}.
\newblock URL \url{https://doi.org/10.1137/100814299}.

\bibitem[Angel and Holroyd(2012)]{AngelHolroyd2012}
Omer Angel and Alexander~E. Holroyd.
\newblock Recurrent rotor-router configurations.
\newblock \emph{J. Comb.}, 3\penalty0 (2):\penalty0 185--194, 2012.
\newblock \doi{10.4310/JOC.2012.v3.n2.a3}.
\newblock URL \url{https://doi.org/10.4310/JOC.2012.v3.n2.a3}.

\bibitem[Bak et~al.(1987)Bak, Tang, and Wiesenfeld]{BakTangWiesenfeld1987}
Per Bak, Chao Tang, and Kurt Wiesenfeld.
\newblock Self-organized criticality: an explanation of the $1/f$ noise.
\newblock \emph{Phys. Rev. Lett.}, 59\penalty0 (4):\penalty0 381--384, 1987.
\newblock \doi{10.1103/PhysRevLett.59.381}.
\newblock URL \url{https://doi.org/10.1103/PhysRevLett.59.381}.

\bibitem[Benjamini and Schramm(1996)]{BenjaminiSchramm1996}
Itai Benjamini and Oded Schramm.
\newblock Percolation beyond {$\mathbb{Z}^d$}, many questions and a few
  answers.
\newblock \emph{Electron. Commun. Probab.}, 1:\penalty0 71--82, 1996.
\newblock \doi{10.1214/ECP.v1-978}.
\newblock URL \url{https://doi.org/10.1214/ECP.v1-978}.

\bibitem[Bou-Rabee(2021)]{BouRabee2021}
Ahmed Bou-Rabee.
\newblock Convergence of the random {A}belian sandpile.
\newblock \emph{Ann. Probab.}, 49\penalty0 (6):\penalty0 3168--3196, 2021.
\newblock \doi{10.1214/21-AOP1528}.
\newblock URL \url{https://doi.org/10.1214/21-AOP1528}.

\bibitem[Bou-Rabee(2024{\natexlab{a}})]{BouRabee2024FLattice}
Ahmed Bou-Rabee.
\newblock Integer superharmonic matrices on the {F}-lattice.
\newblock \emph{Adv. Math.}, 436:\penalty0 Paper No. 109400,
  2024{\natexlab{a}}.
\newblock \doi{10.1016/j.aim.2023.109400}.
\newblock URL \url{https://doi.org/10.1016/j.aim.2023.109400}.

\bibitem[Bou-Rabee(2024{\natexlab{b}})]{BouRabee2024Shape}
Ahmed Bou-Rabee.
\newblock A shape theorem for exploding sandpiles.
\newblock \emph{Ann. Appl. Probab.}, 34\penalty0 (1A):\penalty0 714--742,
  2024{\natexlab{b}}.
\newblock \doi{10.1214/23-AAP1976}.
\newblock URL \url{https://doi.org/10.1214/23-AAP1976}.

\bibitem[Chan(2019)]{Chan2019}
Swee~Hong Chan.
\newblock Rotor walks on transient graphs and the wired spanning forest.
\newblock \emph{SIAM J. Discrete Math.}, 33\penalty0 (4):\penalty0 2369--2393,
  2019.
\newblock \doi{10.1137/18M1217139}.
\newblock URL \url{https://doi.org/10.1137/18M1217139}.

\bibitem[Chan(2020)]{Chan2020}
Swee~Hong Chan.
\newblock A rotor configuration with maximum escape rate.
\newblock \emph{Electron. Commun. Probab.}, 25:\penalty0 Paper No. 19, 2020.
\newblock \doi{10.1214/20-ECP298}.
\newblock URL \url{https://doi.org/10.1214/20-ECP298}.

\bibitem[Chan(2023)]{Chan2023}
Swee~Hong Chan.
\newblock Recurrence of horizontal--vertical walks.
\newblock \emph{Ann. Inst. Henri Poincar\'e Probab. Stat.}, 59\penalty0
  (2):\penalty0 578--605, 2023.
\newblock \doi{10.1214/22-AIHP1277}.
\newblock URL \url{https://doi.org/10.1214/22-AIHP1277}.

\bibitem[Chan et~al.(2021)Chan, Greco, Levine, and Li]{ChanGrecoLevineLi2021}
Swee~Hong Chan, Lila Greco, Lionel Levine, and Peter Li.
\newblock Random walks with local memory.
\newblock \emph{J. Stat. Phys.}, 184\penalty0 (1):\penalty0 Paper No. 6, 2021.
\newblock \doi{10.1007/s10955-021-02791-5}.
\newblock URL \url{https://doi.org/10.1007/s10955-021-02791-5}.

\bibitem[Chen and Kudler-Flam(2020)]{ChenKudlerFlam2020}
Joe~P. Chen and Jonah Kudler-Flam.
\newblock Laplacian growth and sandpiles on the {S}ierpi{\'n}ski gasket: limit
  shape universality and exact solutions.
\newblock \emph{Ann. Inst. Henri Poincar\'e D}, 7\penalty0 (4):\penalty0
  585--664, 2020.
\newblock \doi{10.4171/AIHPD/95}.
\newblock URL \url{https://doi.org/10.4171/AIHPD/95}.

\bibitem[Chen et~al.(2020)Chen, Huss, Sava-Huss, and
  Teplyaev]{ChenHussSavaHussTeplyaev2020}
Joe~P. Chen, Wilfried Huss, Ecaterina Sava-Huss, and Alexander Teplyaev.
\newblock Internal {DLA} on {S}ierpi{\'n}ski gasket graphs.
\newblock In Matthias Keller, Daniel Lenz, and Radoslaw~K. Wojciechowski,
  editors, \emph{Analysis and Geometry on Graphs and Manifolds}, volume 461 of
  \emph{London Math. Soc. Lecture Note Ser.}, pages 126--155. Cambridge Univ.
  Press, 2020.
\newblock \doi{10.1017/9781108615259.008}.
\newblock URL \url{https://doi.org/10.1017/9781108615259.008}.

\bibitem[Cooper and Spencer(2006)]{CooperSpencer2006}
Joshua~N. Cooper and Joel Spencer.
\newblock Simulating a random walk with constant error.
\newblock \emph{Combin. Probab. Comput.}, 15\penalty0 (6):\penalty0 815--822,
  2006.
\newblock \doi{10.1017/S0963548306007565}.
\newblock URL \url{https://doi.org/10.1017/S0963548306007565}.

\bibitem[Coupier et~al.(2024)Coupier, Henry, Jahnel, and
  K{\"o}ppl]{CoupierHenryJahnelKoppl2024}
David Coupier, Beno{\^\i}t Henry, Benedikt Jahnel, and Jonas K{\"o}ppl.
\newblock The planar lattice two-neighbor graph percolates.
\newblock To appear in Ann. Probab., 2024.
\newblock URL \url{https://doi.org/10.48550/arXiv.2412.20781}.

\bibitem[Dhar(1990)]{Dhar1990}
Deepak Dhar.
\newblock Self-organized critical state of sandpile automaton models.
\newblock \emph{Phys. Rev. Lett.}, 64\penalty0 (14):\penalty0 1613--1616, 1990.
\newblock \doi{10.1103/PhysRevLett.64.1613}.
\newblock URL \url{https://doi.org/10.1103/PhysRevLett.64.1613}.

\bibitem[Dvoretzky and Erd\H{o}s(1951)]{DvoretzkyErdos1951}
Aryeh Dvoretzky and Paul Erd\H{o}s.
\newblock Some problems on random walk in space.
\newblock In \emph{Proceedings of the Second Berkeley Symposium on Mathematical
  Statistics and Probability}, pages 353--367. University of California Press,
  Berkeley, 1951.
\newblock \doi{10.1525/9780520411586-026}.
\newblock URL \url{https://doi.org/10.1525/9780520411586-026}.

\bibitem[Florescu et~al.(2014)Florescu, Ganguly, Levine, and
  Peres]{FlorescuGangulyLevinePeres2014}
Laura Florescu, Shirshendu Ganguly, Lionel Levine, and Yuval Peres.
\newblock Escape rates for rotor walks in {$\mathbb{Z}^d$}.
\newblock \emph{SIAM J. Discrete Math.}, 28\penalty0 (1):\penalty0 323--334,
  2014.
\newblock \doi{10.1137/130908646}.
\newblock URL \url{https://doi.org/10.1137/130908646}.

\bibitem[Florescu et~al.(2016)Florescu, Levine, and
  Peres]{FlorescuLevinePeres2016}
Laura Florescu, Lionel Levine, and Yuval Peres.
\newblock The range of a rotor walk.
\newblock \emph{Amer. Math. Monthly}, 123\penalty0 (7):\penalty0 627--642,
  2016.
\newblock \doi{10.4169/amer.math.monthly.123.7.627}.
\newblock URL \url{https://doi.org/10.4169/amer.math.monthly.123.7.627}.

\bibitem[Grimmett(1989)]{Grimmett1989}
Geoffrey Grimmett.
\newblock \emph{Percolation}.
\newblock Springer, New York, 1989.
\newblock \doi{10.1007/978-1-4757-4208-4}.
\newblock URL \url{https://doi.org/10.1007/978-1-4757-4208-4}.

\bibitem[Holroyd and Propp(2010)]{HolroydPropp2010}
Alexander~E. Holroyd and James Propp.
\newblock Rotor walks and {M}arkov chains.
\newblock In \emph{Algorithmic Probability and Combinatorics}, volume 520 of
  \emph{Contemp. Math.}, pages 105--126. Amer. Math. Soc., Providence, RI,
  2010.
\newblock \doi{10.1090/conm/520/10256}.
\newblock URL \url{https://doi.org/10.1090/conm/520/10256}.

\bibitem[Holroyd et~al.(2008)Holroyd, Levine, M{\'e}sz{\'a}ros, Peres, Propp,
  and Wilson]{HolroydLevineMeszarosPeresProppWilson2008}
Alexander~E. Holroyd, Lionel Levine, Karola M{\'e}sz{\'a}ros, Yuval Peres,
  James Propp, and David~B. Wilson.
\newblock Chip-firing and rotor-routing on directed graphs.
\newblock In \emph{In and Out of Equilibrium 2}, volume~60 of \emph{Progress in
  Probability}, pages 331--364. Birkh\"auser, Basel, 2008.
\newblock \doi{10.1007/978-3-7643-8786-0_17}.
\newblock URL \url{https://doi.org/10.1007/978-3-7643-8786-0_17}.

\bibitem[Huss and Sava(2011)]{HussSava2011}
Wilfried Huss and Ecaterina Sava.
\newblock Rotor-router aggregation on the comb.
\newblock \emph{Electron. J. Combin.}, 18\penalty0 (1):\penalty0 Paper 224,
  2011.
\newblock \doi{10.37236/711}.
\newblock URL \url{https://doi.org/10.37236/711}.

\bibitem[Huss and Sava(2012)]{HussSava2012}
Wilfried Huss and Ecaterina Sava.
\newblock Transience and recurrence of rotor-router walks on directed covers of
  graphs.
\newblock \emph{Electron. Commun. Probab.}, 17:\penalty0 Paper No. 41, 2012.
\newblock \doi{10.1214/ECP.v17-2096}.
\newblock URL \url{https://doi.org/10.1214/ECP.v17-2096}.
\newblock Erratum: \emph{Electron. Commun. Probab.} \textbf{19} (2014), 1--6,
  \url{https://doi.org/10.1214/ECP.v19-3848}.

\bibitem[Huss and Sava-Huss(2020{\natexlab{a}})]{HussSavaHuss2020Periodic}
Wilfried Huss and Ecaterina Sava-Huss.
\newblock A law of large numbers for the range of rotor walks on periodic
  trees.
\newblock \emph{Markov Process. Related Fields}, 26\penalty0 (3):\penalty0
  467--485, 2020{\natexlab{a}}.
\newblock URL \url{https://math-mprf.org/journal/articles/id1582/}.

\bibitem[Huss and Sava-Huss(2020{\natexlab{b}})]{HussSavaHuss2020Range}
Wilfried Huss and Ecaterina Sava-Huss.
\newblock Range and speed of rotor walks on trees.
\newblock \emph{J. Theoret. Probab.}, 33\penalty0 (3):\penalty0 1657--1690,
  2020{\natexlab{b}}.
\newblock \doi{10.1007/s10959-019-00904-1}.
\newblock URL \url{https://doi.org/10.1007/s10959-019-00904-1}.

\bibitem[Huss et~al.(2015)Huss, M{\"u}ller, and
  Sava-Huss]{HussMullerSavaHuss2015}
Wilfried Huss, Sebastian M{\"u}ller, and Ecaterina Sava-Huss.
\newblock Rotor-routing on {G}alton--{W}atson trees.
\newblock \emph{Electron. Commun. Probab.}, 20:\penalty0 Paper No. 49, 2015.
\newblock \doi{10.1214/ECP.v20-4000}.
\newblock URL \url{https://doi.org/10.1214/ECP.v20-4000}.

\bibitem[Huss et~al.(2018)Huss, Levine, and Sava-Huss]{HussLevineSavaHuss2018}
Wilfried Huss, Lionel Levine, and Ecaterina Sava-Huss.
\newblock Interpolating between random walk and rotor walk.
\newblock \emph{Random Structures Algorithms}, 52\penalty0 (2):\penalty0
  263--282, 2018.
\newblock \doi{10.1002/rsa.20747}.
\newblock URL \url{https://doi.org/10.1002/rsa.20747}.

\bibitem[Kager and Levine(2010)]{KagerLevine2010}
Wouter Kager and Lionel Levine.
\newblock Rotor-router aggregation on the layered square lattice.
\newblock \emph{Electron. J. Combin.}, 17\penalty0 (1):\penalty0 Paper No.
  R152, 16, 2010.
\newblock \doi{10.37236/424}.
\newblock URL \url{https://doi.org/10.37236/424}.

\bibitem[Kaiser and Sava-Huss(2024)]{KaiserSavaHuss2024}
Robin Kaiser and Ecaterina Sava-Huss.
\newblock Random rotor walks and i.i.d. sandpiles on {S}ierpi\'nski graphs.
\newblock \emph{Statist. Probab. Lett.}, 209:\penalty0 Paper No. 110090, 2024.
\newblock \doi{10.1016/j.spl.2024.110090}.
\newblock URL \url{https://doi.org/10.1016/j.spl.2024.110090}.

\bibitem[Kapri and Dhar(2009)]{KapriDhar2009}
Rajeev Kapri and Deepak Dhar.
\newblock Asymptotic shape of the region visited by an {E}ulerian walker.
\newblock \emph{Phys. Rev. E}, 80\penalty0 (5):\penalty0 051118, 2009.
\newblock \doi{10.1103/PhysRevE.80.051118}.
\newblock URL \url{https://doi.org/10.1103/PhysRevE.80.051118}.

\bibitem[Kesten(1980)]{Kesten1980}
Harry Kesten.
\newblock The critical probability of bond percolation on the square lattice
  equals {$1/2$}.
\newblock \emph{Comm. Math. Phys.}, 74\penalty0 (1):\penalty0 41--59, 1980.
\newblock \doi{10.1007/BF01197577}.
\newblock URL \url{https://doi.org/10.1007/BF01197577}.

\bibitem[Kingman(1968)]{Kingman1968}
John F.~C. Kingman.
\newblock The ergodic theory of subadditive stochastic processes.
\newblock \emph{J. Roy. Statist. Soc. Ser. B}, 30\penalty0 (3):\penalty0
  499--510, 1968.
\newblock \doi{10.1111/j.2517-6161.1968.tb00749.x}.
\newblock URL \url{https://doi.org/10.1111/j.2517-6161.1968.tb00749.x}.

\bibitem[Landau and Levine(2009)]{LandauLevine2009}
Itamar Landau and Lionel Levine.
\newblock The rotor-router model on regular trees.
\newblock \emph{J. Combin. Theory Ser. A}, 116\penalty0 (2):\penalty0 421--433,
  2009.
\newblock \doi{10.1016/j.jcta.2008.05.012}.
\newblock URL \url{https://doi.org/10.1016/j.jcta.2008.05.012}.

\bibitem[Lawler et~al.(1992)Lawler, Bramson, and
  Griffeath]{LawlerBramsonGriffeath1992}
Gregory~F. Lawler, Maury Bramson, and David Griffeath.
\newblock Internal diffusion limited aggregation.
\newblock \emph{Ann. Probab.}, 20\penalty0 (4):\penalty0 2117--2140, 1992.
\newblock \doi{10.1214/aop/1176989542}.
\newblock URL \url{https://doi.org/10.1214/aop/1176989542}.

\bibitem[Levine and Peres(2009)]{LevinePeres2009}
Lionel Levine and Yuval Peres.
\newblock Strong spherical asymptotics for rotor-router aggregation and the
  divisible sandpile.
\newblock \emph{Potential Anal.}, 30\penalty0 (1):\penalty0 1--27, 2009.
\newblock \doi{10.1007/s11118-008-9104-6}.
\newblock URL \url{https://doi.org/10.1007/s11118-008-9104-6}.

\bibitem[Levine and Peres(2017)]{LevinePeres2017}
Lionel Levine and Yuval Peres.
\newblock Laplacian growth, sandpiles, and scaling limits.
\newblock \emph{Bull. Amer. Math. Soc. (N.S.)}, 54\penalty0 (3):\penalty0
  355--382, 2017.
\newblock \doi{10.1090/bull/1573}.
\newblock URL \url{https://doi.org/10.1090/bull/1573}.

\bibitem[Levine et~al.(2016)Levine, Pegden, and Smart]{LevinePegdenSmart2016}
Lionel Levine, Wesley Pegden, and Charles~K. Smart.
\newblock Apollonian structure in the {A}belian sandpile.
\newblock \emph{Geom. Funct. Anal.}, 26\penalty0 (1):\penalty0 306--336, 2016.
\newblock \doi{10.1007/s00039-016-0358-7}.
\newblock URL \url{https://doi.org/10.1007/s00039-016-0358-7}.

\bibitem[Levine et~al.(2017)Levine, Pegden, and Smart]{LevinePegdenSmart2017}
Lionel Levine, Wesley Pegden, and Charles~K. Smart.
\newblock The {A}pollonian structure of integer superharmonic matrices.
\newblock \emph{Ann. of Math. (2)}, 186\penalty0 (1):\penalty0 1--67, 2017.
\newblock \doi{10.4007/annals.2017.186.1.1}.
\newblock URL \url{https://doi.org/10.4007/annals.2017.186.1.1}.

\bibitem[Li(2021)]{Li2021}
Linjun Li.
\newblock On the {M}anhattan pinball problem.
\newblock \emph{Electron. Commun. Probab.}, 26:\penalty0 Paper No. 25, 2021.
\newblock \doi{10.1214/21-ECP394}.
\newblock URL \url{https://doi.org/10.1214/21-ECP394}.

\bibitem[Liggett et~al.(1997)Liggett, Schonmann, and
  Stacey]{LiggettSchonmannStacey1997}
Thomas~M. Liggett, Roberto~H. Schonmann, and Alan~M. Stacey.
\newblock Domination by product measures.
\newblock \emph{Ann. Probab.}, 25\penalty0 (1):\penalty0 71--95, 1997.
\newblock \doi{10.1214/aop/1024404279}.
\newblock URL \url{https://doi.org/10.1214/aop/1024404279}.

\bibitem[Pegden and Smart(2013)]{PegdenSmart2013}
Wesley Pegden and Charles~K. Smart.
\newblock Convergence of the {A}belian sandpile.
\newblock \emph{Duke Math. J.}, 162\penalty0 (4):\penalty0 627--642, 2013.
\newblock \doi{10.1215/00127094-2079677}.
\newblock URL \url{https://doi.org/10.1215/00127094-2079677}.

\bibitem[Povolotsky et~al.(1998)Povolotsky, Priezzhev, and
  Shcherbakov]{PovolotskyPriezzhevShcherbakov1998}
Alexander~M. Povolotsky, Vyacheslav~B. Priezzhev, and R.~R. Shcherbakov.
\newblock Dynamics of {E}ulerian walkers.
\newblock \emph{Phys. Rev. E}, 58\penalty0 (5):\penalty0 5449--5454, 1998.
\newblock \doi{10.1103/PhysRevE.58.5449}.
\newblock URL \url{https://doi.org/10.1103/PhysRevE.58.5449}.

\bibitem[Priezzhev et~al.(1996)Priezzhev, Dhar, Dhar, and
  Krishnamurthy]{PriezzhevDharDharKrishnamurthy1996}
Vyacheslav~B. Priezzhev, Deepak Dhar, Abhishek Dhar, and Supriya Krishnamurthy.
\newblock {E}ulerian walkers as a model of self-organized criticality.
\newblock \emph{Phys. Rev. Lett.}, 77\penalty0 (25):\penalty0 5079--5082, 1996.
\newblock \doi{10.1103/PhysRevLett.77.5079}.
\newblock URL \url{https://doi.org/10.1103/PhysRevLett.77.5079}.

\bibitem[Propp(2003)]{Propp2003}
James Propp.
\newblock Random walk and random aggregation, derandomized.
\newblock Lecture, Microsoft Research, 2003.
\newblock URL
  \url{https://www.microsoft.com/en-us/research/video/random-walk-and-random-aggregation-derandomized/}.

\bibitem[Propp(2010)]{Propp2010}
James Propp.
\newblock Discrete analog computing with rotor-routers.
\newblock \emph{Chaos}, 20\penalty0 (3):\penalty0 037110, 2010.
\newblock \doi{10.1063/1.3489886}.
\newblock URL \url{https://doi.org/10.1063/1.3489886}.

\bibitem[Rabani et~al.(1998)Rabani, Sinclair, and
  Wanka]{RabaniSinclairWanka1998}
Yuval Rabani, Alistair Sinclair, and Rolf Wanka.
\newblock Local divergence of {M}arkov chains and the analysis of iterative
  load-balancing schemes.
\newblock In \emph{39th Annual Symposium on Foundations of Computer Science},
  pages 694--703. IEEE Computer Society, 1998.
\newblock \doi{10.1109/SFCS.1998.743520}.
\newblock URL \url{https://doi.org/10.1109/SFCS.1998.743520}.

\bibitem[Wagner et~al.(1996)Wagner, Lindenbaum, and
  Bruckstein]{WagnerLindenbaumBruckstein1996}
Israel~A. Wagner, Michael Lindenbaum, and Alfred~M. Bruckstein.
\newblock Smell as a computational resource: a lesson we can learn from the
  ant.
\newblock In \emph{Fourth Israel Symposium on Theory of Computing and Systems
  {(ISTCS '96)}}, pages 219--230. IEEE Computer Society, 1996.
\newblock URL
  \url{https://cris.technion.ac.il/en/publications/smell-as-a-computational-resource-a-lesson-we-can-learn-from-the-}.

\end{thebibliography}
	
\end{document}